\documentclass[11pt, reqno]{amsart}

\usepackage{amsmath}
\usepackage{amscd}
\usepackage{mathrsfs}
\usepackage{amsfonts}
\usepackage{amssymb}
\usepackage{enumerate}
\usepackage{setspace}

\newcounter{hypo}

\newcommand{\R}{\ensuremath{\mathbb{R}}}

\newcommand{\N}{\ensuremath{\mathbb{N}}}

\newcommand{\ep}{\varepsilon}

\newcommand{\beq}[1]{\begin{equation}\label{#1}}
\newcommand{\eeq}{\end{equation}}
\newcommand{\beqs}{\begin{equation*}}
\newcommand{\eeqs}{\end{equation*}}
\newcommand{\set}[1]{\left\{#1\right\}}

\newcommand{\dd}{{d}}

\usepackage{array}

\usepackage{color}

\usepackage[hyperindex, colorlinks=true]{hyperref}

\usepackage[margin=1in]{geometry}

\newcommand{\eqn}{\begin{eqnarray}}
\newcommand{\een}{\end{eqnarray}}

\newcommand{\D}{\Lambda}

\def\R{\Bbb R}

\def\iint{\int\int}

\newtheorem{theorem}{Theorem}[section]

\newtheorem{lemma}{Lemma}[section]
\newtheorem{proposition}[theorem]{Proposition}

\theoremstyle{definition}
\newtheorem{definition}{Definition}[section]

\newtheorem{remark}{Remark}

\numberwithin{equation}{section}

\begin{document}

\title[Regularity Results for generalized SQG equations]{Regularity results for a class of generalized surface quasi-geostrophic equations}

\author{Omar Lazar}
\address{Instituto de Ciencias Matem\'aticas (ICMAT),  Consejo Superior de Investigaciones Cient\'ificas, Madrid, and Departamento de An\'alisis Matem\'atico \& IMUS, Universidad de Sevilla, C/ Tarifa s/n, Campus Reina Mercedes, 41012 Sevilla, Spain}
\email{omar.lazar@icmat.es}

\author{Liutang Xue}
\address{Laboratory of Mathematics and Complex Systems (MOE), School of Mathematical Sciences,  Beijing Normal University, Beijing 100875, P.R. China}
\email{xuelt@bnu.edu.cn}


\date{\today}



\begin{abstract}
We show a global existence result of weak solutions for a class of generalized Surface Quasi-Geostrophic equation in the inviscid case.  We also prove the global regularity of such solutions for the equation with slightly supercritical dissipation, which turns out to correspond to a logarithmically supercritical diffusion due to the singular nature of the velocity. Our last result is the eventual regularity in the supercritical cases for such weak solutions.
The main idea in the proof of the existence part is based on suitable commutator estimates along with a careful cutting into low/high frequencies and inner/outer spatial scales to pass to the limit; while the proof of both the global regularity result and the eventual regularity for the supercritical diffusion are essentially based on the use of the so-called {\it{modulus of continuity method}}. 
\end{abstract}

\maketitle
\section{Introduction}
 In this article, we study the following generalized Surface Quasi-Geostrophic equation
\begin{equation}\label{eq:gSQG}
(gSQG)_{\beta} : \begin{cases}
  \partial_t \theta + u\cdot\nabla \theta + \nu \Lambda^\beta\theta = 0, &\quad (t,x)\in \mathbb{R}^+\times \mathbb{R}^2, \\
  u = \nabla^\perp \Lambda^{\beta-2} m(\Lambda)\theta, \\
  \theta(0,x)= \theta_0(x), & \quad x\in\mathbb{R}^2,
\end{cases}
\end{equation}
where $\nu\geq 0$ and $\beta\in (0,1]$. Recall that $\Lambda^\beta$ is the usual fractional Laplacian operator defined as
\[\Lambda^{\beta}\theta \equiv (-\Delta)^{\beta/2} \theta=C_\beta \,P.V.\int_{\mathbb{R}^{2}}{\frac{\theta(x)-\theta(x-y)}{|y|^{2+\beta}}\dd y}, \]
where $C_{\beta}>0$ is a positive constant. The operator $m(\Lambda)$ that appears in the velocity is defined using the Fourier transform via the following identity
\begin{equation*}
  m(\Lambda)f(x) =  \frac{1}{(2\pi)^2} \int_{\mathbb{R}^2} e^{i x\cdot \zeta}m(|\zeta|) \widehat{f}(\zeta) \mathrm{d}\zeta,
\end{equation*}
where $m$  satisfies the following conditions:
\begin{enumerate}[(\textrm{A}1)]
\item
$m(\zeta)=m(|\zeta|)$ is a radial nondecreasing function such that $m\in C^\infty(\mathbb{R}^2\setminus \{0\})$ and $m>0$ for all $\zeta\neq 0$.
\item
$m$ is in the Mikhlin-H\"ormander class (see e.g. \cite{Stein}), that is, there exists a constant $b_0>0$ so that
\begin{equation}\label{eq:cond3}
  \vert\zeta\vert^{k}|\partial_\zeta^km(\zeta)|\leq b_0\, m(|\zeta|), \qquad \forall k\in \{1,2,3,4,5\}, \forall \zeta\neq 0.
\end{equation}
\item
 There exists a constant $\alpha\in (0,1)$ and a constant $b_1\geq 1$ such that
\begin{equation}\label{eq:cond5}
  |\zeta|m'(|\zeta|)\leq \alpha\, m(|\zeta|), \qquad \forall\,|\zeta|\geq b_1.
\end{equation}
\item
There exists a constant $\lambda\in [0,1)$ and a constant $b_2\geq 1$ such that
\begin{equation}\label{eq:cond4}
  b_2^{-1} \leq \frac{m(\zeta)}{|\zeta|^\lambda} \leq b_2,\qquad \mathrm{for \ all} \ \zeta \in \R^{2}\setminus\{(0,0)\} \ \mathrm{such \ that} \ \vert\zeta\vert\leq 1.
\end{equation}
\end{enumerate}
Moreover, instead of (A3), we also consider the following more restrictive assumption \\
\noindent \textrm{(A5)} There exists a constant $\alpha\in (0,1)$ such that
\begin{equation}\label{eq:cond2}
  \vert\zeta \vert m'(\vert\zeta\vert)\leq \alpha\, m(\vert\zeta\vert), \qquad \mathrm{for \ all} \ \zeta \in \R^{2}\setminus\{(0,0)\}.
\end{equation}


If $\nu \geq 0$ and $m(\Lambda)=\Lambda^{1-\beta}$ (which obviously satisfies (A1)-(A5) with $\alpha=\lambda=1-\beta$), then the $(gSQG)_{\beta}$ equation reduces to the so-called Surface Quasi-Geostrophic equation. This equation reads as follows
\begin{equation}\label{qg}
\ (SQG)_{\beta}: \\ \left\{
\aligned
&\partial_{t}\theta(x,t)+u\cdot\nabla\theta+ \nu\Lambda^{\beta}\theta = 0,
\\
& u= \mathcal{R}^\perp \theta=(-\partial_{x_{2}} \Lambda^{-1}\theta,\partial_{x_{1}} \Lambda^{-1}\theta),
\\
& \theta(0,x)=\theta_{0}(x).
\endaligned
\right.
\end{equation}
The $(SQG)_\beta$ equation is an important model in geophysical sciences that serves for instance in  meteorology and atmospheric sciences as well as to understand turbulence flows (see e.g. \cite{Ped}). Besides, it turns out to be an important mathematical equation because of its similarity with the 3D Euler equations and the 3D Navier-Stokes equations (see \cite{CMT}). In the dissipative case $\nu>0$, due to the scaling of the equation, one may consider 3 cases that are $\beta \in (0,1)$, $\beta=1$ and $\beta \in (1,2)$ which are respectively called, supercritical, critical and subcritical. \\

The operator $m(\Lambda)$ in the equation $(gSQG)_\beta$ defined in \eqref{eq:gSQG} plays an important role. It would allows us to consider a larger class of $SQG$ equation with singular velocity. Indeed, by choosing an adequate $m$ we can allow the velocity to be $\gamma$-order singular ({\it{i.e.}} the symbol is of order $\gamma\in (0,1)$)
or $\gamma$-order regular ({\it{i.e.}} being of order $-\gamma\in (-1,0)$) or logarithmically $\gamma$-order singular/regular ({\it{i.e.}} being of order $\gamma$/$-\gamma$ up to a logarithmic factor).
This would make the study of the Cauchy problem \eqref{eq:gSQG}  much more delicate than the usual $(SQG)_\beta$ equation \eqref{qg}.
Equation $(gSQG)_\beta$ with general $m(\Lambda)$ and in the case $\nu>0$ and $\beta=1$ was initiated by Dabkowski, Kiselev and Vicol \cite{DKV} where the authors have studied a double logarithmically (0-order) singular velocity. This is equivalent to say that the equation has slight supercriticality in the velocity. 
In this paper, we shall basically follow this approach by considering a general family of multiplier operator $m(\Lambda)$ that gives rise to a logarithmically $\gamma$-order singular/regular velocity,
and we shall focus on the construction of global weak $L^{1}\cap L^2$ solutions in the inviscid case and then study their global regularity or eventual regularity in the viscous case, essentially critical and supercritical. 
\\

\noindent Let us first recall some well known results about the classical $(SQG)_\beta$ equation. 
After the first mathematical study initiated by Constantin, Majda and Tabak in \cite{CMT}, there is a huge mathematical literature dealing with the $(SQG)_{\beta}$ 
equation (one may see the long reference list in \cite{CCW,CCCGW}).
In the inviscid case $\nu=0$, $L^{2}$-weak solutions were shown to exist globally in Resnick's tesis \cite{Resnick}, then  Marchand \cite{Mar} was able to extend Resnick's result to a more general class of weak solutions, namely, he proved the existence of global weak $L^p$-solutions with $p>4/3$. While the uniqueness of Leray-Hopf solutions is still a challenging open problem,  Buckmaster, Shkoller and Vicol \cite{BSV} have been able to make an important step toward this problem by showing the nonuniqueness for a class of solutions having negative Sobolev regularity.
In the dissipative case, that is $\nu>0$ and $0<\beta\leq2$, the equation has been studied by many authors.
While in the subcritical case ($1<\beta\leq 2$), the global wellposedness is by now very well understood (see for instance, Resnick \cite{Resnick}, Constantin-Wu \cite{CW3}, Dong-Li \cite{DL}), the global regularity in the critical case ($\beta=1$) turned out to be more delicate to prove and has been successfully obtained slighlty more than a decade ago by Kiselev-Nazarov-Volberg \cite{KNV} and Caffarelli-Vasseur \cite{CV}; both proofs used different techniques and have appeared almost at the same time. The method used by Kiselev, Nazarov and Volberg is based on a preservation of a well chosen modulus of continuity while the Caffarelli and Vasseur proof is based on the De Giorgi iteration method.  Two other proofs of the global regularity have been obtained independently by Kiselev-Nazarov \cite{KN} using a duality approach and by Constantin-Vicol \cite{CVic} using what they called the {\it{nonlinear maximum principle}}. \\

In contrast with the subcritical and the critical case, the global regularity issue in the supercritical remains an outstanding open problem.
However, some authors  have been able to make a slight step toward this challenging problem. The first results in the supercritical regime were local existence for regular enough data and global existence for small initial data (see e.g. Chae-Lee \cite{ChaeLee}, Dong-Li \cite{DL2}, Hmidi-Keraani \cite{HK}, Ju \cite{Ju}, Chen-Miao-Zhang \cite{CMZ}, and Wu \cite{Wu}).
Then, some authors have been able to show that, despite of the fact that the drift term is stronger than the dissipation (at least for small scales),
there are still many regularity criteria which ensure that some class of solutions might become regular after a short period of time providing that some quantity is controlled.
For instance, Constantin and Wu \cite{CW2} were able to prove that if the velocity is $C^{1-\beta}$ then the solution becomes H\"older continuous for all $t>0$ (this result has been proved to be sharp by Silvestre, Vicol and Zlato\v{s} \cite{SVZ}). In a different paper, Constantin and Wu \cite{CW} proved that if the Leray-Hopf solution is in the subcritical space $C^\delta$, $\delta>1-\beta$ on a time interval $[t_1, t_2]$  then the solution becomes smooth on $(t_1,t_2]$. This last result has been extended by Silvestre \cite{Silv}(and also \cite{Silv2}) to any vector fields that are not necessarily divergence free.
Dong and Pavlovi\'c \cite{DP} have been able to prove the same result as \cite{CW} in a critical Besov space. \\

Another kind of results, called {\it{eventual regularity}}, which means that the weak solution becomes smooth after some time, have been established in the supercritical regime. 
More precisely, it was proved by Silvestre \cite{Silv3} that, for $\beta$ very close to 1, there exists a time so that the solution has some H\"older continuity after this time which together with \cite{CW} implies higher regularity and then smoothness. The {\it{eventual regularity}} result of \cite{Silv3}  has been improved by Dabkowski in  \cite{Dab} where he has been able to prove the same result for the whole supercritical case $\beta \in (0,1)$ by applying the duality method originated in \cite{KN}.   
Later, Kiselev \cite{Kis} reproved this result in the supercritical case by using the modulus of continuity method by showing the preservation of a well-chosen family of time-dependent moduli of continuity.
Such an {\it{eventual regularity}} result is also obtained by Coti Zelati and Vicol \cite{CZV}. Besides, they were able to show that the time of {\it{eventual regularity}} (which is explicit) goes to 0 as the dissipation approaches the critical case.
The authors in \cite{CZV} also proved that for all $\theta_0\in H^2$, the $(SQG)_\beta$ equation has a global classical solution when the dissipation index $\beta$ sufficiently close to $1$ (depending on the norm of $\theta_0$). \\

If $\nu\geq0$, $\beta \in (0,1]$ and  $m(\Lambda)=\mathrm{Id}$ (the identity operator),  then the $(gSQG)_{\beta}$ equation corresponds to the modified dissipative $(mSQG)_{\beta}$ equation introduced by Constantin, Iyer and Wu in \cite{CIW}. This modified equation interpolates between an equation that arises in the study of the evolution of a 2D damped invisicd fluid equation ($\beta=0$) and the classical $SQG$ equation ($\beta=1$);
The equation reads as follows
\begin{equation}\label{GSQG}
(mSQG)_{\beta} : \begin{cases}
  \partial_t \theta + u\cdot\nabla \theta + \nu \Lambda^\beta\theta = 0, &\quad (t,x)\in \mathbb{R}^+\times \mathbb{R}^2, \\
  u = \nabla^\perp \Lambda^{\beta-2} \theta, \\
  \theta(0,x)= \theta_0(x), & \quad x\in\mathbb{R}^2,
\end{cases}
\end{equation}
When $\beta \in (0,1]$ and $\nu>0$ this equation scales as the critical $SQG$ equation and global existence of global smooth solutions starting from $L^{2}$ data is proved (in the same spirit as \cite{CV}).
\\

For $\nu=0$ and for $m(\Lambda)= \Lambda^\alpha$ with $\alpha\in (1-\beta,1)$, $\beta\in (0,1]$ (noting that $m$ satisfies (A1)-(A5) with $\alpha=\lambda$), then the
$(gSQG)_\beta$ defined in \eqref{eq:gSQG} reduces to the inviscid generalized $SQG$ equation introduced by Chae, Constantin, C\'ordoba, Gancedo and Wu in \cite{CCCGW},
which reads as follows
\begin{equation}\label{GSQG}
\begin{cases}
  \partial_t \theta + u\cdot\nabla \theta  = 0, &\quad (t,x)\in \mathbb{R}^+\times \mathbb{R}^2, \\
  u = \nabla^\perp \Lambda^{\alpha+\beta-2} \theta, \\
  \theta(0,x)= \theta_0(x), & \quad x\in\mathbb{R}^2,
\end{cases}
\end{equation}
This equation has a singular velocity of order $\alpha+\beta-1\in(0,1)$, which makes the study much more delicate. Despite this fact, the authors of \cite{CCCGW} have been able to prove a local well-posedness result for initial data in $H^4$  as well as a global existence theorem of weak solutions for $\theta_0\in L^2(\mathbb{T}^2)$ with mean zero. This latter result is an improvement of Resnick's result \cite{Resnick}. To overcome the lack of regularity to close the estimates, the authors have used a new commutator estimate in terms of the stream function $\psi=\Lambda^{\beta-2}\theta$ which turned out to be crucial.
We also want to mention that, if one intends to construct Morrey-Campanato type solutions for instance, then, having a singular velocity may help since it would give more decay to the convolution kernel of the singular integral $u$ (see e.g. \cite{OL}).
So far, in the inviscid case or the supercritical dissipative case (i.e. $\nu>0$ and with dissipation $\nu\Lambda^\delta\theta$, $0<\delta<\beta$ on the left-hand side of \eqref{GSQG}),
the problem of the global regularity versus finite time blow-up remains widely open. \\

 The active scalar equation with a general Fourier multiplier in the velocities has been first studied by Chae, Constantin and Wu \cite{CCW}. They considered $u= \nabla^\perp \Lambda^{-2} m(\Lambda)\theta$ which corresponds to equation \eqref{eq:gSQG} in the case $\beta=0$. This equation may be viewed as a generalized inviscid Euler vorticity equation. The authors of \cite{CCW} proved the global well-posedness of strong solution for the so-called LogLog-Euler equation (see also \cite{Elg}). In the same spirit as \cite{CCW}, \cite{Tao}, Dabkowski, Kiselev and Vicol \cite{DKV} have studied the  $(gSQG)_\beta$ equation \eqref{eq:gSQG} in the case $\beta=1$, $\nu>0$. They considered the case of a singular velocity having a double logarithmically growth. They proved th
 e existence of a global unique strong solution by using the {\it{modulus of continuity method}}.
The second name author and Zheng \cite{XZ} have been able to improve this global result for the $(gSQG)_\beta$ equation in the case $\beta\in (0,1]$ and $\nu>0$ by showing that a logarithmically growth for the singular velocity is enough for the well-posedness.
We note that both proofs are for strong solutions with smooth initial data and are based on the construction of suitable stationary modulus of continuity.  \\

Our aim in this paper is to slightly go beyond the available existence results of weak solutions both in the inviscd case as well as in the supercritical case.
By studying equation $(gSQG)_{\beta}$ with $m$ satisfying (A1)-(A5), we will be able to improve several well-known theorems.

\section{Main Results and Comments}
The first result is the global existence of weak solutions for the inviscid $(gSQG)_\beta$ equation
\eqref{eq:gSQG}. 

\begin{theorem}\label{thm:GE}
  Let $\nu=0$, $\beta\in (0,1]$ and $\theta_0\in L^1\cap L^2(\R^2)$.  Assume that $m$ satisfies the assumptions (A1)-(A4) with $\alpha\in(0,1)$, $\lambda\in[0,1)$.
Then, the $(gSQG)_\beta$ equation \eqref{eq:gSQG} has a global weak solution which verifies, for all $T>0$, $\theta\in L^\infty([0,T]; L^1\cap L^2(\R^2))$.
\end{theorem}

Then, we consider the $(gSQG)_\beta$ equation \eqref{eq:gSQG} with dissipation (without loss of generality, we may assume $\nu=1$).
We obtain a global regularity result of weak solutions in the logarithmically supercritical case as well as the eventual regularity of global weak solutions. 

\begin{theorem}\label{thm:RegLog}
  Let $\nu=1$, $\beta\in (0,1]$ and $\theta_0\in L^1\cap L^2(\R^2)$. 
\begin{enumerate}[(1)]
\item
Assume that $m$ satisfies (A1)-(A4) with $\alpha\in(0,1)$, $\lambda\in[0,1)$, then, if there exists $\mu\in [0,1]$ and $b_3\geq 1$ so that
\begin{equation}\label{mcd1}
  b_3^{-1} \leq m(r) \leq b_3 (\log r)^\mu,\qquad \forall r \geq b_3.
\end{equation}
Then, for all time $t_*>0$, the global weak solution $\theta$ constructed in Theorem \ref{thm:GE} satisfies $\theta\in C^\infty ([t_*,\infty)\times\R^2)$. \\

\item
Assume that $m$ satisfies (A1)(A2)(A4)(A5) with $\alpha\in(0,1)$, $\lambda\in[0,1)$.
Then for any $t'>0$, there exists $T_{*}>0$ which depends only on $t'$, $\|\theta_0\|_{L^2(\R^2)},\alpha,\beta$ so that the global weak solution $\theta$ constructed in Theorem \ref{thm:GE}
satisfies that $\theta\in C^\infty([t'+T_{*},\infty)\times\R^2)$.
In particular, if $\beta\in(0,1)$, and for each fixed $\theta_0\in L^2(\R^2)$ and $t'>0$, we have
\begin{equation}\label{T*}
\displaystyle \lim_{\alpha\rightarrow 0}T_{*}(t',\alpha,\beta,\theta_0)=0.
\end{equation}
\end{enumerate}
\end{theorem}

The novelty of the results obtained in this paper are highlighted in the following. \\

The first result of this article, that is Theorem \ref{thm:GE}, is a global existence theorem of weak solutions for equation $(gSQG)_{\beta}$
driven by a family of (logarithmically) $\gamma$-order singular/regular velocities,
for example,
\begin{equation}
  \textrm{$u=\nabla^\perp \Lambda^{\alpha+\beta-2}\theta$,\quad or \quad  $u=\nabla^\perp \Lambda^{\alpha'+\beta-2}\ln(1+\Lambda)\theta$ \quad with $\beta\in (0,1]$, $0<\alpha'<\alpha<1$}.
\end{equation}
Recalling that Constantin, Chae, C\'ordoba, Gancedo and Wu  in \cite{CCCGW} obtained the existence of $L^{2}$ weak solutions for equation \eqref{GSQG} in the inviscid case and for
$u=\nabla^\perp \Lambda^{\beta-2}\theta$ with $\beta \in (1,2)$, we here extend this result to a more general class of (logarithmically) singular/regular velocities including the logarithmically supercritical cases. \\

The second result of this paper, which is the first point of Theorem \ref{thm:RegLog}, is a global regularity result of weak solutions 
for the $(gSQG)_\beta$ equation with $\beta\in (0,1]$ and under the mild growth condition \eqref{mcd1}, which corresponds to a logarithmically supercritical case.
This theorem generalizes the classical regularity result of Caffarelli-Vasseur \cite{CV} and Kiselev-Nazarov \cite{KN} to an even more singular velocity than the Riesz transform.
This result is also comparable to the one obtained in \cite{DKV} and \cite{XZ} where global well-posedness for $(gSQG)_\beta$ for the (double) logarithmically supercritical case is obtained for smooth solutions.
However, both results \cite{DKV},\cite{XZ} were dealing with strong solutions, and we here lower significatively the initial condition by considering just Leray-Hopf type weak solutions. \\

The third and last result of this paper states that weak solutions established in Theorem \ref{thm:GE} in the supercritical cases eventually become regular. Besides,
the time of eventual regularity is explicitly obtained. Consequently, we are able to get that, for $\beta\in(0,1)$, such a time goes to 0 as the supercritical index $\alpha$ goes to the critical index 0. 
The last point is  in the same spirit as the result of Coti Zelati and Vicol \cite{CZV}, which shows the continuity of the solution map associated to the $(gSQG)_\beta$ equation as $\alpha\rightarrow 0$.

\begin{remark}\label{rmk-CzV}
  If $\nu>0$, $\beta\in (0,1)$ and $m(\Lambda)=\D^{1-\beta}$, then, as we said before, the  $(gSQG)_\beta$ equation \eqref{eq:gSQG} becomes the supercritical $(SQG)_\beta$ equation \eqref{qg}. Following the idea of the proof of \eqref{T*},
one may prove that, given an initial data $\theta_0\in L^\infty(\R^2)$, the eventual regularity time $T_*$ converges to 0 as $\beta\rightarrow 1$, which allows us to recover a result of Coti Zelati and Vicol in \cite{CZV}.
Indeed, in this case we have $\alpha=1-\beta$ and $\beta$ is close to 1. Then, by choosing  $\sigma=\frac{1}{2}$, and  $C_1>0$ (the constant that appears in Remark \ref{rmk-moc}), and following the same steps as the proof of \eqref{eq:t1conv}
(one would actually use $\|\theta^\epsilon\|_{L^\infty([0,\infty)\times \R^2)} \leq \|\theta_0\|_{L^\infty}$ in stead of \eqref{LinfEst}),
one may find the following convergence result
\begin{equation*}
  T_* (\beta)= \frac{1}{C_0\beta} \left(\frac{4 C_0 m(1)}{\beta^3}(1-\beta) \|\theta_0\|_{L^\infty}\right)^{\frac{\beta}{1-\beta}}\rightarrow 0,\quad \textrm{as   }\beta\rightarrow1.
\end{equation*}
\end{remark}

\begin{remark}\label{rmk:assum}
  In both the proof of the global regularity result stated in Theorem \ref{thm:RegLog}  result as well as the eventual regularity, the main task is to show that the moduli of continuity verify the criterion \eqref{Targ3} (or \eqref{Targ4}) for all $0<\xi\leq A_0$ where $A_0=\xi_0(0)$ is the starting point of $\xi_0$.
The main difference is that $A_0$ in the eventual regularity issue  is not small in general, while $A_0$ in the global regularity problem can be chosen arbitrarily small,
so that we need a slightly restrictive assumption, namely (A5) instead of (A3), in the eventual regularity result.
We think that after a more complicated and careful analysis one might still prove the same eventual regularity result with the assumption (A3),
but that will not be our aim here.
\end{remark}

For the proof of Theorem \ref{thm:GE}, the idea is to first consider a regularized equation with an additional viscous term and then construct a global solution by compactness arguments.
The main problem is to show the convergence of the nonlinear transport terms in the sense of distribution.
By considering the stream function $\psi^\epsilon=\Lambda^{\beta-2}m(\Lambda)\theta^\epsilon$ and using the structure of the singular velocity,
we can rewrite the nonlinear term (when tested against a cut-off) as a controlled commutator, see \eqref{eq:Exp}.
We need to localise the considered quantities into low and high frequencies as \eqref{eq:decom2}, and since we work on the whole space $\R^2$,
we also introduce several cut-off functions to split the terms in \eqref{eq:decom2} with respect to the space variable. In particular,  we rewrite the term involving high-high interactions as \eqref{eq:Exp2} which crucially contains a commutator.
Then, we prove some useful commutator estimates. These estimates allow us to manage to prove the locally strong convergence of the high-frequency part of the (regularized) stream function $\mathcal{H}_0\psi^\epsilon$ and some functions involving $S_0\theta^\epsilon$ or $\mathcal{H}_0\psi^\epsilon$ in the locally $L^2$-topology, so that we can pass to the limit in each of those terms in the above decomposition. \\

As for the proof of Theorem \ref{thm:RegLog}, it is based on the so-called {\it{modulus of continuity method}} as it is stated in Kiselev's paper \cite{Kis} (originated from \cite{KNV}).
We first show a H\"older regularity criterion (Lemma \ref{lem:RCpres}) in terms of the modulus of continuity $\omega(\xi)$ given by \eqref{moc1}, which is well suited to the study of the $(gSQG)_\beta$ equation. Besides, this modulus of continuity is an unbounded function contrarily to the one introduced in \cite{Kis} (where, it is chosen to be constant for $\xi$ large enough, namely for all $\xi>\delta$, $\omega(\xi)=\omega(\delta)$).
Then, we introduce a new family of time-dependent moduli of continuity $\omega(\xi,\xi_0(t))$ defined by \eqref{moc2}-\eqref{moc3} and \eqref{xi0},
which is obtained as a suitable modification of the stationary modulus of continuity \eqref{moc1}.
Then, using some tedious (however elementary) computations, we can prove that the (uniformly in $\epsilon$) bounded solution $\theta^\epsilon$ would (uniformly in $\epsilon$)
preserve such moduli of continuity $\omega(\xi,\xi_0(t))$ for all time $t$ so that $\xi_0(t)>0$.
This implies that at some finite time the solution $\theta^\epsilon$ obeys the modulus of continuity $\omega(\xi,0+)=\omega(\xi)$ given by \eqref{moc1}.
Combined with the regularity criteria of Lemmas \ref{lem:RC} and \ref{lem:RCpres}, we get the desired eventual regularity result.
In this process, the time of eventual regularity has an explicit expression. This allows us to show the formula \eqref{T*} and conclude the global regularity result in the logarithmically supercritical case. \\

The outline of this paper is as follows. In the next section we give some preliminary results and some useful lemmas that we shall use throughout the paper.  The fourth section is devoted to the proof of Theorem \ref{thm:GE}. In the fifth section, we prove Theorem \ref{thm:RegLog}.
The last  section, which is an appendix, deals with the proof of a result used in the subsection \ref{sec:ws}.

\section{Auxiliary lemmas and modulus of continuity}\label{sec:prel}

In this section, we recall some well-known definitions and results and we prove several auxiliary lemmas. In particular, we collect some useful  dealing with the modulus of continuity in Subsection \ref{subsec:MOC}.

\subsection{Preliminary and auxiliary lemmas}

The following lemma will be frequently used throughout the article.
\begin{lemma}\label{lem:m}
  Assume that $m(r)$ $(r>0)$ is a smooth non-decreasing positive function satisfying \eqref{eq:cond5} for some $\alpha\in (0,1)$ and $b_1>0$.
Then for every $\rho\geq \alpha$,
\begin{equation}\label{m-fac}
  \textrm{the function  } r\mapsto r^\rho m(r^{-1}) \textrm{ for all $0<r\leq b_1^{-1}$ is non-decreasing},
\end{equation}
and
\begin{equation}\label{m-fac2}
  \textrm{the function  }|\zeta|\mapsto |\zeta|^{-\rho} m(|\zeta|) \textrm{ for all $|\zeta|\geq b_1$ is non-increasing}.
\end{equation}
In particular, if  \eqref{eq:cond2} is satisfied, then we have the following claim
\begin{equation}\label{m-fac3}
  \textrm{the properties \eqref{m-fac} and \eqref{m-fac2} hold for all $r>0$ and $|\zeta|>0$, respectively}.
\end{equation}
\end{lemma}

\begin{proof}[{\bf Proof of Lemma \ref{lem:m}}]
  From \eqref{eq:cond5}, we directly have that for all $0<r<b_1^{-1}$,
$$(r^\rho m(r^{-1}))'= \rho r^{\rho-1} m(r^{-1}) - r^{\rho-2} m'(r^{-1})\geq (\rho-\alpha) r^{\rho-1} m(r^{-1})\geq 0,$$
which leads to \eqref{m-fac}. The fact \eqref{m-fac2} and \eqref{m-fac3} can be similarly proved.
\end{proof}

Now we recall the definition of the dyadic blocks (see e.g. \cite{BCD11} or \cite{PGLR}).
Let $\widetilde{\chi} \in \mathcal{D}(\mathbb R^2)$ be a non negative function such that $\widetilde{\chi}(x)=1$ if $\vert x \vert \leq 1/2$ and $0$ if $\vert x \vert \geq 1$. Let us define another function $\widetilde{\varphi} \in \mathcal{D}(\mathbb R^2)$ by $\widetilde{\varphi}(x)=\widetilde{\chi}(x/2)-\widetilde{\chi}(x)$ which is therefore supported on a corona.
Then, we define the Fourier multiplier $S_j$ and $\Delta_j$ ($j\in\mathbb{Z}$) by
$$
\widehat{S_{j} f}(\zeta)=\widetilde{\chi}(2^{-j}\zeta) \widehat{ f}(\zeta) \ \ \text{and} \ \  \widehat{\Delta_{j} f}(\zeta)=\widetilde{\varphi}(2^{-j}\zeta) \widehat{ f} (\zeta).
$$
\noindent From these operators we deduce the Littlewood-Paley decomposition of a distribution  $f\in \mathcal{S}'$, that is, for all $N\in \mathbb Z$, we have
$$
f=S_N f+ \sum_{j\geq N} \Delta_j f \quad \text{in} \ \mathcal{S}'(\mathbb{R}^{2}).
$$
We then define the low frequency and high frequency cutting operators as
\begin{equation}\label{eq:lowLP}
S_j f=\mathcal{F}^{-1}\left(\widetilde{\chi}(2^{-j}\zeta) \right)= \sum_{k \leq j-1} \Delta_{k}f
\end{equation}
as well as the high frequency cutting operator as
\begin{equation}\label{eq:LPop}
  \mathcal{H}_{j} f=(\mathrm{Id}-S_{j})f=\sum_{k \geq j} \Delta_{k}f.
\end{equation}
From these operator we define, for $s \in \mathbb{R}$ and $(p,q) \in [1,\infty]^2$ the inhomogeneous Besov spaces as the set of $f \in \mathcal{S}'(\R^2)$ so that the following quantity is finite
$$
\Vert f \Vert_{B^{s}_{p,q}(\R^2)}\equiv \|S_0f\|_{L^p} + \left\|\{2^{js} \|\Delta_j f\|_{L^p}\}_{j\in\mathbb{N}} \right\|_{\ell^q}. 
$$
In particular, we have the equivalence of $L^2$-based Sobolev space $H^s(\R^2) = B^s_{2,2}(\R^2)$. \\

The following lemma deals with the action of the Fourier multiplier $m(\Lambda)$ into the dyadic blocs.
\begin{lemma}\label{lem:mBern}
  Let $m$ be a function satisfying (A1)(A2) (A3). Then, there exists a constant $C>0$ depending only on $b_0$, $m(b_1)$ such that for every $p\in [1,\infty]$ and $j\in\mathbb{N}$, we have
\begin{equation}\label{eq:mBern}
  \|m(\Lambda)\Delta_j f\|_{L^p(\R^2)} \leq C 2^{j \alpha} \|\Delta_j f\|_{L^p(\R^2)}.
\end{equation}
\end{lemma}

\begin{proof}[{\bf Proof of Lemma \ref{lem:mBern}}]
Note that from \eqref{m-fac2} and the nondecreasing property of $m$, we have for all $|\zeta|\geq 1/2$,
$$ m(\zeta) \leq \max\{m(b_1), b_1^{-\alpha}m(b_1) \vert\zeta\vert^\alpha\}\leq 2^\alpha m(b_1) |\zeta|^\alpha ,$$
thus thanks to the assumption (A2), for all $\zeta\in \R^2\setminus B_{1/2}$ we find
$$ \vert\partial^km(\zeta)\vert \leq b_0 |\zeta|^{-k}m(\zeta) \leq 2 b_0 m(b_1) \vert\zeta\vert^{\alpha-k},\quad \forall k\in\{1,2,3,4\}, $$
so that we may apply  Lemma 2.2, p. 53 of \cite{BCD11} to immediately obtain \eqref{eq:mBern}.
\end{proof}

The purpose of the following lemma is to estimate the convolution kernels of some operators involving $m(\Lambda)$.
\begin{lemma}\label{lem:kernel}
Let $\beta\in(0,1]$ and $m(\zeta)=m(|\zeta|)$ be a non-decreasing function satisfying the assumptions (A1)-(A4).
\begin{enumerate}[(1)]
\item
Let $K_{\beta,j}(x)$ be the kernel of the operator $\partial_j \D^{\beta-2}m(\Lambda)$ ($j=1,2$). Then for all $x\in\mathbb{R}^2\setminus\{0\}$,
\begin{equation}\label{kernel1}
  K_{\beta,j}(x)= \frac{x_j}{|x|} H_\beta(x), \quad \textrm{with}\quad |H_\beta(x)|\leq C_0' |x|^{-1-\beta}m(|x|^{-1}),
\end{equation}
and
\begin{equation}\label{kernel2}
  |\nabla K_{\beta,j}(x)|\leq C_0' |x|^{-2-\beta} m(|x|^{-1}),
\end{equation}
where the constant $C_0'>0$ depends on $b_1,b_2$ but depends neither on $\alpha$ nor $\beta$.
\item
Let $\widetilde{K}_{\beta,j}(x)$ be the kernel of the operator $\frac{\D^{2-\beta}\partial_j}{m(\Lambda)}$ ($j=1,2$). Then for every $x\in\mathbb{R}^2\setminus\{0\}$,
\begin{equation}\label{kernel3}
  \widetilde{K}_{\beta,j}(x)= \frac{x_j}{|x|} \widetilde{H}_\beta(x), \quad \textrm{with}\quad |\widetilde{H}_\beta(x)|\leq C_0+ \frac{C}{|x|^{5-\beta}m(|x|^{-1})},
\end{equation}
and
\begin{equation}\label{kernel4}
  |\nabla \widetilde{K}_{\beta,j}(x)|\leq C_0' + \frac{C}{ |x|^{6-\beta} m(|x|^{-1})},
\end{equation}
where $C>0$ is a constant depending on $\alpha$ and $\beta$.
\end{enumerate}
\end{lemma}

\begin{proof}[{\bf{Proof of Lemma \ref{lem:kernel}}}]

(1) For the proof of \eqref{kernel1}-\eqref{kernel2}, one can refer to that of \cite{DKV} (Lemma 4.1), or, as well, it may be seen in the proof of \eqref{kernel3} (see below).
The constant $C_0'$ does not depends on $\alpha$ nor $\beta$, which will be easy to see since these constants will be explicit in the proof. \\

(2) Let $\phi$ be a smooth radial function which is supported on $[-1,1]^{2}$ and which satisfies $\phi(x)=1$ on $[-1/2,1/2]^{2}$.
Let us set $\phi_R(x) = \phi(\frac{x}{R})$ for $R>0$, and $\widetilde{L}_\beta(x)$ be the kernel
of the convolution operator $\frac{\D^{2-\beta}}{m(\Lambda)}$, then for some $R>0$ to be chosen later, we write
\begin{equation*}
\begin{split}
  \widetilde{L}_\beta(x) & = C_0 \int_{\mathbb{R}^2}e^{i x\cdot \zeta} \frac{|\zeta|^{2-\beta}}{m(|\zeta|)}{d}\zeta
  = C_0 \int_{\mathbb{R}^2}e^{i |x| \zeta_1} \frac{|\zeta|^{2-\beta}}{m(|\zeta|)}\mathrm{d}\zeta \\
  & = C_0 \int_{\mathbb{R}^2}e^{i |x|\zeta_1} \phi_R(\zeta)\frac{|\zeta|^{2-\beta}}{m(\zeta)} {d}\zeta
  + C_0 \int_{\mathbb{R}^2}e^{i |x|\zeta_1} (1-\phi_R(\zeta))\frac{|\zeta|^{2-\beta}}{m(\zeta)} {d}\zeta \\
  & = C_0 \int_{\mathbb{R}^2}e^{i |x|\zeta_1} \phi_R(\zeta)\frac{|\zeta|^{2-\beta}}{m(\zeta)} {d}\zeta
  +  C_0 |x|^{-5} \int_{\mathbb{R}^2}e^{i |x|\zeta_1} \,\partial_{\zeta_1}^5\left((1-\phi_R(\zeta))\frac{|\zeta|^{2-\beta}}{m(\zeta)}\right) \mathrm{d}\zeta 
\end{split}
\end{equation*}
where $C_0>0$ is a fixed constant. Thus, we see that
$\widetilde{K}_{\beta,i}(x)=\partial_j \widetilde{L}_\beta(x) = \frac{x_j}{|x|} \widetilde{H}_\beta(x)$
with
\begin{equation*}
\begin{split}
  \widetilde{H}_\beta(x) = &\, C_0 \int_{\mathbb{R}^2}e^{i |x|\zeta_1}\, i\zeta_1\phi_R(\zeta)\frac{|\zeta|^{2-\beta}}{m(\zeta)} {d}\zeta
  +  C_0 |x|^{-5} \int_{\mathbb{R}^2}e^{i |x|\zeta_1} \,i\zeta_1\partial_{\zeta_1}^5\left((1-\phi_R(\zeta))\frac{|\zeta|^{2-\beta}}{m(\zeta)}\right){d}\zeta \\
  & + C_0 |x|^{-6} \int_{\mathbb{R}^2}e^{i |x|\zeta_1} \,\partial_{\zeta_1}^5\left((1-\phi_R(\zeta))\frac{|\zeta|^{2-\beta}}{m(\zeta)}\right){d}\zeta \\
  \equiv &\, {I}_{1,\beta} + {I}_{2,\beta} + {I}_{3,\beta}.
\end{split}
\end{equation*}
To estimate ${I}_{1,\beta}$, we use the assumption (A4) and Lemma \ref{lem:m} to find that if $R>b_1$, then
\begin{align*}
  |{I}_{1,\beta}| & \leq C_0 \int_{B_R} \frac{|\zeta|^{3-\beta}}{m(\zeta)} {d}\zeta
  \leq C_0 \int_{|\zeta|\leq 1} \frac{|\zeta|^{3-\beta}}{m(\zeta)} \dd \eta + C_0 \int_{1\leq |\zeta|\leq b_1} \frac{|\zeta|^{3-\beta}}{m(\zeta)} \dd \eta + C_0 \int_{b_1\leq|\zeta|\leq R} \frac{|\zeta|^{3-\beta}}{m(\zeta)} \dd \zeta \\
  &\leq C_0b_2 \int_{|\zeta|\leq 1} |\zeta|^{3-\beta-\lambda} \dd \eta +  C_0 b_2 \int_{1\leq |\zeta|\leq b_1} |\zeta|^{3-\beta} \dd \eta + \frac{C_0}{R^{-\alpha}m(R)} \int_{b_1\leq|\zeta|\leq R} |\zeta|^{3-\beta-\alpha} \dd \zeta \\
  & \leq C_0 b_2 (b_1^3+b_1^5) +  \frac{C_0 R^{5-\beta}}{ m(R)},
\end{align*}
while if $R\leq b_1$,
\begin{align*}
  |{I}_{1,\beta}| \leq C_0 \int_{|\zeta|\leq b_1} \frac{|\zeta|^{3-\beta}}{m(\zeta)} \mathrm{d}\zeta \leq C_0 b_2 (b_1^3+b_1^5).
\end{align*}
For ${I}_{2,\beta}$,  we apply \eqref{eq:cond3} along with Lemma \ref{lem:m} to  obtain
\begin{equation*}
\begin{split}
  |{I}_{2,\beta}| & \leq C |x|^{-5} \left(\int_{|\zeta|\geq R/2} \frac{1}{|\zeta|^{\beta+2}m(\zeta)} \mathrm{d}\zeta + \sum_{j=1}^5 R^{-j}\int_{R/2\leq |\zeta|\leq R} \frac{1}{|\zeta|^{\beta+2 -j}m(\zeta)} \dd \zeta \right)\\
  & \leq C |x|^{-5} \left( \frac{1}{ m(R)} \int_{|\zeta|\geq R/2} \frac{1}{ |\zeta|^{\beta+2}} \mathrm{d}\zeta + \frac{1}{R^\beta m(R)} \right)
  \leq C \frac{|x|^{-5}}{ R^\beta m(R)}.
\end{split}
\end{equation*}
For the last term, following what we did previously, we get
\begin{equation*}
  |{I}_{3,\beta}|\leq C \frac{|x|^{-5}}{ R^{\beta+1} m(R)}.
\end{equation*}
Hence, by choosing $R= |x|^{-1}$ one get the desired estimate \eqref{kernel3}. The bound \eqref{kernel4} can be obtained in the same fashion and we omit the details.
\end{proof}
The next lemma will be very useful to control some commutator and differential operators acting on product, which are crucial estimates used in the proof of Theorem \ref{thm:GE}.
\begin{lemma}\label{lem:comm}
  Let $m$ be a function satisfying (A1)-(A4).
Let $s> \max\{\alpha,\lambda\}$ be a real number.
\begin{enumerate}[(1)]
\item
For all $j=1,2$ and all $\epsilon>0$, we have
\begin{equation}\label{eq:comm}
  \left\|\left[\frac{\D^s\partial_j}{ m(\Lambda)}, f\right]g \right\|_{L^2(\R^2)} \leq C \left( \|f\|_{H^{2+\epsilon}(\R^2)}\left\|\frac{\D^s}{ m(\Lambda)}g\right\|_{L^2(\R^2)}
  + \|f\|_{H^{s +2+\epsilon}(\R^2)} \|g\|_{L^2(\R^2)}\right),
\end{equation}
where $C>0$ is a constant depending only on $s,\alpha$, $\epsilon$ and $b_0,b_1,b_2$.
\item
Let us denote by $\mathcal{M}_s(\Lambda)$ the multiplier operator with the symbol $\mathcal{M}_s(|\zeta|)$ which is given by
\begin{equation}\label{Ms}
  \mathcal{M}_s(|\zeta|)\equiv
  \begin{cases}
    \frac{m(\zeta)}{|\zeta|^s},\quad &\mathrm{if}\; |\zeta|\geq 1, \\
    m(1),\quad &\mathrm{if}\; |\zeta|\leq 1.
  \end{cases}
\end{equation}
Then, for every $j=1,2$ and every $\epsilon>0$, we get
\begin{equation}\label{eq:comm2}
  \left\|\mathcal{M}_s(\Lambda)\bigg(\left[\frac{\D^s\partial_j}{m(\Lambda)}, f\right]g\bigg) \right\|_{L^2(\R^2)}\leq C \|f\|_{H^{s+2+\epsilon}(\R^2)} \|g\|_{L^2(\R^2)},
\end{equation}
with $C>0$ a constant depending only on $s,\alpha$, $\epsilon$ and $b_0,b_1,b_2$.
\item
We have the following product estimates that for every $\epsilon>0$,
\begin{equation}\label{eq:prodEs}
  \left\|\frac{\D^s}{m(\Lambda)}(f g)\right\|_{L^2(\R^2)} \leq C \left(\left\|\frac{\D^s}{m(\Lambda)}f\right\|_{L^2(\R^2)} \|g\|_{H^{1+\epsilon}(\R^2)} + \|f\|_{L^2(\R^2)} \|g\|_{H^{s+1+\epsilon}(\R^2)}\right),
\end{equation}
and
\begin{equation}\label{eq:prodEs2}
  \left\|\frac{1}{\mathcal{M}_s(\Lambda)}(fg)\right\|_{L^2(\R^2)} \leq C \left(\left\|\frac{\D^s}{m(\Lambda)}f\right\|_{L^2(\R^2)} \|g\|_{H^{1+\epsilon}(\R^2)} + \|f\|_{L^2(\R^2)} \|g\|_{H^{s+1+\epsilon}(\R^2)}\right),
\end{equation}
with $C>0$ some constant depending only on $s$, $\epsilon$ and $b_1,b_2$.
\end{enumerate}
\end{lemma}

\begin{proof}[{\bf{Proof of Lemma \ref{lem:comm}}}]
(1) We shall first prove \eqref{eq:comm}. By using the Fourier transform, we have
\begin{equation}\label{eq:commFT}
  \mathcal{F}\left(\left[\frac{\D^s}{m(\Lambda)}\partial_j,f\right]g\right)(\zeta)=
  \int_{\R^2} \left(\frac{|\zeta|^s}{m(\zeta)} \zeta_j - \frac{|\eta|^s}{ m(\eta)} \eta_j\right)\hat{f}(\zeta-\eta) \hat{g}(\eta) \dd \eta.
\end{equation}
By a direct computation, we find that
\begin{equation}\label{eq:Fes0}
  \left|\frac{|\zeta|^s }{m(\zeta)} \zeta_j- \frac{|\eta|^s}{ m(\eta)}\eta_j\right|\leq (s+1+b_0) \max\left\{b_2b_1^s,\frac{|\zeta|^s}{m(\zeta)}, \frac{|\eta|^s}{ m(\eta)}\right\} |\zeta-\eta|.
\end{equation}
Indeed, by setting $F(\zeta)=\frac{|\zeta|^s\zeta_j}{m(\zeta)}$, one observes that, by using the assumption (A2),
$$|\partial_{\zeta_k} F(\zeta)|\leq (s+1) \frac{|\zeta|^s}{m(\zeta)} + \frac{|\zeta|^{s+1} m'(|\zeta|)}{m(\zeta)^2}\leq  (s+1+b_0) \frac{|\zeta|^s}{m(\zeta)},\quad \textrm{for  $k=1,2$,}$$ we see that
\begin{align}\label{eq:Fes1}
  |F(\zeta)-F(\eta)| & =\left|\int_0^1 \frac{\partial}{\partial \tau}F(\tau \zeta +(1-\tau)\eta) \dd \tau \right| \nonumber \\
  & \leq \int_0^1 |(\zeta-\eta)\cdot\nabla F(\tau \zeta + (1-\tau)\eta)| \dd \tau \nonumber\\
  & \leq (s+1+b_0)|\zeta-\eta| \,\int_0^1 \frac{|\tau \zeta +(1-\tau)\eta|^s}{ m(\tau\zeta+(1-\tau)\eta)} \dd \tau.
\end{align}
If $|\tau\zeta +(1-\tau)\eta|\leq b_1$, by considering the cases $|\tau\zeta+(1-\tau)\tau|\leq 1$ and $1\leq |\tau \zeta +(1-\tau)\eta|\leq b_1$ separately, we have
\begin{align*}
  \frac{|\tau\zeta +(1-\tau)\eta|^s}{m(\tau\zeta +(1-\tau)\eta)} \leq \max\{b_2,b_2 b_1^s\}=b_2 b_1^s;
\end{align*}
if $|\tau\zeta +(1-\tau)\eta|\geq b_1$ and $|\zeta|\geq |\eta|$, we get $|\tau \zeta +(1-\tau)\eta|\leq |\zeta|$, thus \eqref{m-fac2} implies that
\begin{equation*}
  \frac{|\tau \zeta + (1-\tau)\eta|^s}{m(\tau \zeta +(1-\tau)\eta)}= \frac{|\tau \zeta + (1-\tau)\eta|^{s-\alpha}}{|\tau\zeta +(1-\tau)\eta|^{-\alpha}m(\tau \zeta +(1-\tau)\eta)}
  \leq \frac{|\tau \zeta +(1-\tau)\eta|^{s-\alpha}}{|\zeta|^{-\alpha}m(\zeta)}\leq \frac{|\zeta|^s}{m(\zeta)};
\end{equation*}
while if $|\tau\zeta +(1-\tau)\eta|\geq b_1$ and $|\eta|\geq |\zeta|$ it is not difficult to obtain the same inequality (with $\eta$ instead of $\zeta$); then we conclude that
\begin{equation}\label{eq:Fes2}
  \frac{|\tau \zeta + (1-\tau)\eta|^s}{m(\tau \zeta +(1-\tau)\eta)}\leq \max\left\{b_2b_1^s,\frac{|\zeta|^s}{m(\zeta)}, \frac{|\eta|^s}{m(\eta)} \right\}.
\end{equation}
Hence, by using \eqref{eq:Fes2} to control \eqref{eq:Fes1} one obtains  the desired estimate \eqref{eq:Fes0}. \\

\noindent Then, we observe that if $|\zeta|\geq b_1$,
\begin{equation}\label{eq:fact3}
  \frac{|\zeta|^s}{m(\zeta)} =\frac{|\zeta|^{s-\alpha}}{|\zeta|^{-\alpha}m(\zeta)} \leq \frac{(|\zeta-\eta|+|\eta|)^s}{m(|\zeta-\eta|+|\eta|)}
  \leq \frac{C_s(|\zeta-\eta|^s + |\eta|^s)}{m(|\zeta-\eta|+|\eta|)}\leq C_s \left(\frac{|\zeta-\eta|^s}{m(\zeta-\eta)} + \frac{|\eta|^s}{m(\eta)}\right),
\end{equation}
while if $|\zeta|\leq b_1$,
\begin{equation}\label{eq:fact4}
\frac{|\zeta|^s}{m(\zeta)}\leq b_2b_1^s,
\end{equation}
therefore, by using \eqref{eq:commFT} and \eqref{eq:Fes0}, we infer that
\begin{equation}
\begin{split}
  \left|\mathcal{F}\left(\left[\frac{\D^s\partial_j}{m(\Lambda)},f\right]g\right)(\zeta)\right| \leq
  C'_s \int_{\R^2} \left(\frac{|\zeta-\eta|^{s+1}}{m(\zeta-\eta)} + |\zeta-\eta| \frac{|\eta|^s}{m(\eta)} + |\zeta-\eta|\right) |\hat f(\zeta-\eta)|\, |\hat g(\eta)|\,\dd \eta.
\end{split}
\end{equation}
Using Plancherel's theorem and Young's inequality for convolution, we obtain
\begin{equation}\label{eq:CTes1}
\begin{split}
  \left\| \left[\frac{\D^s\partial_j}{m(\Lambda)},f\right]g\right\|_{L^2(\R^2)} \leq & C_s' \left\|\frac{|\zeta|^{s+1}}{m(\zeta)}|\hat f(\zeta)|\right\|_{L^1(\R^2)} \|g\|_{L^2}
  + C_s' \big\||\zeta||\hat f(\zeta)|\big\|_{L^1(\R^2)} \left\|\frac{\D^s}{m(\Lambda)}g\right\|_{L^2(\R^2)} \\
  & + C_s' \big\||\zeta||\hat f(\zeta)|\big\|_{L^1(\R^2)} \|g\|_{L^2(\R^2)} .
\end{split}
\end{equation}
Now, by using the assumption (A4) along with the nondecreasing property of $m$ and then the H\"older inequality, it follows that for all $\epsilon>0$,
\begin{equation}\label{eq:CTes2}
\begin{split}
  \left\|\frac{|\zeta|^{s+1}}{m(\zeta)}|\hat f(\zeta)|\right\|_{L^1(\R^2)} & = \int_{|\zeta|\leq 1} \frac{|\zeta|^{s+1}}{m(\zeta)} |\hat f(\zeta)|\dd \zeta
  + \int_{|\zeta|\geq 1} \frac{|\zeta|^{s+1}}{m(\zeta)} |\hat f(\zeta)|\dd \zeta \\
  & \leq b_2 \int_{|\zeta|\leq 1} |\zeta|^{s+1-\lambda} |\hat f(\zeta)| \dd \zeta + b_2 \int_{|\zeta|\geq 1} |\zeta|^{s+1} |\hat f(\zeta)| \dd \zeta \\
  & \leq C_0b_2\|f\|_{L^2(\R^2)} + b_2 \Big(\int_{|\zeta|\geq 1} |\zeta|^{-2-2\epsilon}\dd \zeta\Big)^{1/2} \|f\|_{H^{s+2+\epsilon}(\R^2)} \\
  & \leq C_\epsilon b_2 \|f\|_{H^{s+2+\epsilon}(\R^2)} ,
\end{split}
\end{equation}
as well,
\begin{equation}\label{eq:CTes3}
  \||\zeta| |\hat f(\zeta)|\|_{L^1(\R^2)} \leq C_\epsilon \|f\|_{H^{2+\epsilon}(\R^2)}.
\end{equation}
Hence, the estimates \eqref{eq:CTes1}-\eqref{eq:CTes3} allow us to obtain the inequality \eqref{eq:comm}. \\

(2) By taking the Fourier transform, we see that
\begin{equation*}
  \mathcal{F}\left(\mathcal{M}_s(\Lambda)\left(\left[\frac{\D^s}{m(\Lambda)}\partial_j,f\right]g\right)\right)(\zeta)=
  \mathcal{M}_s(|\zeta|)\int_{\R^2} \left(\frac{|\zeta|^s}{m(\zeta)} \zeta_j - \frac{|\eta|^s}{ m(\eta)} \eta_j\right)\hat{f}(\zeta-\eta) \hat{g}(\eta) \dd \eta.
\end{equation*}
Using  \eqref{Ms}, \eqref{eq:Fes0} and the following inequality $\frac{|\eta|^s}{m(\eta)} \leq C_s \left(b_2b_1^s+\frac{|\zeta|^s}{m(\zeta)} + \frac{|\zeta-\eta|^s}{m(\zeta-\eta)}\right)$, we find
\begin{equation*}
\begin{split}
  \mathcal{M}_s(|\zeta|) \left|\frac{|\zeta|^s}{m(\zeta)}\zeta_j - \frac{|\eta|^s}{m(\eta)}\eta_j \right| & \leq C \mathcal{M}_s(|\zeta|)\left(1+ \frac{|\zeta|^s}{m(\zeta)} + \frac{|\zeta-\eta|^s}{m(\zeta-\eta)}\right)|\zeta-\eta| \\
  & \leq C \left(1 + \frac{|\zeta-\eta|^s}{m(\zeta-\eta)}\right) |\zeta-\eta|,
\end{split}
\end{equation*}
where in the last estimate we used the inequalities $\mathcal{M}_s(|\zeta|)\leq m(1)\leq b_2$ and $\mathcal{M}_s(|\zeta|) \frac{|\zeta|^s}{m(\zeta)}\leq b_2^2$.
Hence,  as we did for \eqref{eq:CTes1}, by collecting the above estimates with \eqref{eq:CTes2} and \eqref{eq:CTes3} one gets the desired estimate \eqref{eq:comm2}. \\

(3) We only prove \eqref{eq:prodEs}, and the proof of \eqref{eq:prodEs2} will follow using essentially the same argument. By using, once again, the Fourier transform and \eqref{eq:fact3}-\eqref{eq:fact4}, we get
\begin{equation*}
\begin{split}
  & \left|\mathcal{F}\left(\frac{\D^s}{m(\Lambda)} (fg)\right) (\zeta)\right| =\left| \frac{|\zeta|^s}{m(\zeta)} \int_{\R^2} \hat f(\zeta-\eta) \hat g(\eta) \dd \eta \right| \\
  \leq & C_s \left(\int_{\R^2} \frac{|\zeta-\eta|^s}{m(\zeta-\eta)} |\hat f(\zeta-\eta)| |\hat g(\eta)|\dd \eta + \int_{\R^2} |\hat f(\zeta-\eta)| \frac{|\eta|^s}{m(\eta)} |\hat g(\eta)| \dd \eta
  +\int_{\R^2} |\hat f(\zeta-\eta)| |\hat g(\eta)| \dd \eta\right).
\end{split}
\end{equation*}
The Plancherel theorem and Young's inequality give, for all $\epsilon>0$, the following control
\begin{equation*}
\begin{split}
  \left\| \frac{\D^s}{m(\Lambda)}(fg)\right\|_{L^2(\R^2)} & \leq C_s \left( \left\|\frac{\D^s}{m(\Lambda)}f\right\|_{L^2(\R^2)} \|\hat g\|_{L^1(\R^2)}
  + \|f\|_{L^2} \left\|\frac{|\zeta|^s}{m(\zeta)} |\hat g|\right\|_{L^1(\R^2)} + \|f\|_{L^2} \| \hat{g}\|_{L^1(\R^2)}\right) \\
  & \leq C_{s,\epsilon} \left( \left\|\frac{\D^s}{m(\Lambda)}f\right\|_{L^2(\R^2)} \| g\|_{H^{1+\epsilon}(\R^2)}
  + \|f\|_{L^2(\R^2)} \left\| g\right\|_{H^{s+1+\epsilon}(\R^2)} \right),
\end{split}
\end{equation*}
where in the second line we have used \eqref{eq:CTes2}-\eqref{eq:CTes3} applied to $g$ instead of $f$.
\end{proof}

\subsection{Modulus of continuity}\label{subsec:MOC}

In this subsection we give the definition of a modulus of continuity, and then collect some useful results related to the modulus of continuity.
\begin{definition}
  A function $\omega:(0,\infty) \rightarrow (0,\infty)$ is called a modulus of continuity if $\omega$ is continuous on $(0,\infty)$, nondecreasing, concave, and
  piecewise $C^2$ with one-sided derivatives defined at every point in $(0,\infty)$.
  We say a function $f:\mathbb{R}^d\rightarrow \mathbb{R}^l$ obeys the modulus of continuity $\omega$ if $|f(x)-f(y)| < \omega(|x-y|)$
  for every $x\neq y\in \mathbb{R}^d$.
\end{definition}

First we recall the general criterion of the nonlocal maximum principle for the drift-diffusion equation
(for the proof see e.g. \cite[Th 2.2]{Kis} and \cite[Prop 3.1]{MX11}).

\begin{proposition}\label{prop-GC}
 Let $\theta(x,t)\in C([0,\infty);H^s(\R^d))$, $s>\frac{d}{2}+1$ be a smooth solution of the following whole-space drift-diffusion equation
\begin{equation}\label{appActS}
 \partial_t \theta + u\cdot\nabla \theta +  \nu \D^\beta \theta -\epsilon \Delta \theta=0,\quad \theta(0,x)=\theta_0(x), \; x\in\mathbb{R}^d,
\end{equation}
with $\nu\geq 0$, $\epsilon\geq 0$. Assume that
\\
(1) for every $t\geq 0$, $\omega(\xi,t)$ is a modulus of continuity and satisfies that its inverse function $\omega^{-1}((2+\epsilon_0)\|\theta(\cdot,t)\|_{L^\infty_x},t)<\infty$ (with some $\epsilon_0>0$);
\\
(2) for every fixed point $\xi$, $\omega(\xi,t)$ is piecewise $C^1$ in the time variable with one-sided derivatives defined at each point,
 and that for all $\xi$ near infinity, $\omega(\xi,t)$ is continuous in $t$ uniformly in $\xi$;
\\
(3) $\omega(0+,t)$ and $\partial_\xi \omega(0+,t)$ are continuous in $t$ with values in $\R\cup\{\pm \infty\}$,
 and satisfy that for every $t\geq 0$, either $\omega(0+,t)>0$ or $\partial_\xi\omega( 0+,t)=\infty$ or $\partial_{\xi\xi}\omega (0+,t) =-\infty$.

Let the initial data $\theta_0(x)$ obey $\omega(\xi,0)$, then for some $T>0$, $\theta(x,T)$ obeys the modulus of continuity $\omega(\xi,T)$
provided that for all $t\in ]0, T]$ and $\xi\in \set{\xi>0:\omega(\xi,t)\leq 2\| \theta(\cdot,t)\|_{L^\infty_x}}$, $\omega(\xi,t)$ satisfies
\begin{equation}\label{keyGC}
  \partial_t \omega(\xi,t)> \Omega(\xi,t)\,\partial_\xi\omega(\xi,t) + \nu D(\xi,t) + 2\epsilon \partial_{\xi\xi}\omega(\xi,t),
\end{equation}
where $\Omega(\xi,t)$ and $D(\xi,t)$ (we suppress the dependence on $x,e$ in $\Omega(\xi,t)$, $D(\xi,t)$) are respectively defined as follows,  for all $x\in\mathbb{R}^d$
and all unit vector $e\in \mathbb{S}^{d-1}$ in \eqref{scena}, we set
\begin{equation}\label{GC-Om}
  \qquad\quad\Omega(\xi,t)= |(u(x+\xi e,t)-u(x,t))\cdot e|, \quad \textrm{and }
\end{equation}
\begin{equation}\label{GC-D}
  D(\xi,t)= -\big( \D^\beta \theta(x,t) - \D^\beta\theta(x+\xi e,t)\big),
\end{equation}
under the scenario that
\begin{equation}\label{scena}
\begin{split}
  \theta(x,t)-\theta(x+\xi e ,t) = \omega(\xi,t),\quad &\textrm{and}\\
  |\theta(y,t)-\theta(z,t)| \leq \omega(|y-z|,t), \quad &\forall y,z\in\R^d .
\end{split}
\end{equation}

In \eqref{keyGC}, at the points where $\partial_t\omega(\xi,t)$ (or $\partial_\xi\omega(\xi,t)$) does not exist,
the smaller (or larger) value of the one-sided derivative should be taken.
\end{proposition}

The following classical result is an estimate of the dissipative part in terms of the modulus of continuity under the scenario \eqref{scena} (e.g. \cite{Kis}).
\begin{lemma}\label{lem-mocdiss}
We have the following estimate on $D(\xi,t)$ defined by \eqref{GC-D} under the scenario \eqref{scena} that for any $\xi>0$,
\begin{equation}\label{D2}
\begin{split}
  D(\xi,t)\leq &
  \,C_1 \int_{0}^{\frac{\xi}{2}}  \frac{\omega(\xi+2\eta,t)+\omega (\xi-2\eta,t) -2\omega (\xi,t)}{\eta^{1+\beta}}{d} \eta \\
  & + C_1 \int_{\frac{\xi}{2}}^{\infty}  \frac{\omega (2\eta+\xi,t)-\omega (2\eta-\xi,t) -2\omega (\xi,t)}{\eta^{1+\beta}}{d}\eta,
\end{split}
\end{equation}
where $C_1>0$ is a constant depending only on $\beta$.

\end{lemma}

\begin{remark}\label{rmk-moc}
  Let $\mathcal{P}^{\beta,d}_h(x)$ be the $d$-dimensional kernel of the semigroup operator $e^{-h \Lambda^\beta}$,
then we have $\mathcal{P}^{\beta,d}_h(x)= h^{-\frac{d}{\beta}} \mathcal{P}^{h,d} ( h^{-\frac{1}{\beta}}x)$ and $\mathcal{P}^{h,d}(x)=\mathcal{F}^{-1}(e^{-|\zeta|^\beta})(x)$
satisfies that (see \cite[Th 3.1]{BSS})
\begin{equation}
  \frac{c_{\beta,d}}{1+ |x|^{d+\beta}}\leq \mathcal{P}^{\beta,d}(x) \leq \frac{C_{\beta,d}}{1+ |x|^{d+\beta}},
\end{equation}
with $c_{\beta,d}$ and $C_{\beta,d}$ are positive constants depending only on $\beta,d$.
Note that the constant $C_1$ in \eqref{D2} and the constant $c_{\beta,1}$ are the same.
Besides, according to the proof of Th 3.1 in \cite{BSS}, if $\beta$ satisfies $0<c_0 \leq \beta \leq1$ with some explicit constant $c_0\in(0,1)$, then $C_1$ will not depend on $\beta$.
\end{remark}

Next we consider the estimation of \eqref{GC-Om} under the scenario \eqref{scena}.
\begin{lemma}\label{lem-mocdrf}
  Let $u= \nabla^\perp \D^{\beta-2} m(\Lambda)\theta $, we have the following estimates on $\Omega(\xi,t)$ under scenario \eqref{scena}.
\begin{enumerate}[(1)]
\item
If $m$ is a nondecreasing function satisfying the assumptions (A1)(A2)(A4)(A5), then for all $\xi>0$,
\begin{equation}\label{Ome-es1}
\begin{split}
  \Omega(\xi,t)  \leq -C_2\xi m(\xi^{-1}) D(\xi,t) + C_2 \xi\int_\xi^\infty \frac{\omega(\eta,t)m(\eta^{-1})}{\eta^{1+\beta}}{d}\eta +C_2 \xi^{1-\beta} m(\xi^{-1})\omega(\xi,t) ,
\end{split}
\end{equation}
with $C_2= \frac{C_0}{\beta}$ with $C_0$ a fixed constant.
\item
If $m$ is a nondecreasing function satisfying the assumptions (A1)-(A4) and if we only consider $\xi \leq \frac{1}{2 b_1} $, then
\begin{equation}\label{Ome-es1-2}
\begin{split}
  \Omega(\xi,t)  \leq -C_2\xi m(\xi^{-1}) D(\xi,t) + C_2 \xi\int_\xi^\infty \frac{\omega(\eta,t)m(\eta^{-1})}{\eta^{1+\beta}}{d}\eta +C_2 \xi^{1-\beta} m(\xi^{-1})\omega(\xi,t) .
\end{split}
\end{equation}
\item
We also have
\begin{equation}\label{Ome-es2}
\begin{split}
  \Omega(\xi,t)  \leq  C \int_0^\xi \frac{\omega(\eta,t) m(\eta^{-1})}{\eta^\beta} \dd \eta + C \xi\int_\xi^\infty \frac{\omega(\eta,t)m(\eta^{-1})}{\eta^{1+\beta}}\mathrm{d}\eta,
\end{split}
\end{equation}
with $C>0$ depending only on $\beta$.
\end{enumerate}
\end{lemma}

\begin{proof}[{\bf{Proof of Lemma \ref{lem-mocdrf}}}]
We first prove \eqref{Ome-es1}. For simplicity, we suppress the time variable in $\omega(\xi,t)$, $\Omega(\xi,t)$ and $D(\xi,t)$.
Using lemma \ref{lem:kernel}, we find
\begin{equation*}
\begin{split}
  |u(x)-u(y)| & = \left|{P.V.}\int_{\mathbb{R}^2}\frac{y^\bot}{|y|}H_\beta(y)f(x-y){d}y
  - {P.V.}\int_{\mathbb{R}^2}\frac{y^\bot}{|y|}H_\beta(y)f(x+\xi e -y){d}y\right| \\
  & \leq  |I_{1}(\xi)| + |I_2(\xi)|,
\end{split}
\end{equation*}
where $y^\perp=(-y_2,y_1)$, $H_\beta$ is a radial-valued scalar function satisfying \eqref{kernel1}, and
\begin{equation}\label{Ixi}
  I_1(\xi) \equiv {P.V.}\, \int_{|y|\leq 2\xi} \frac{y^\bot}{|y|}H_\beta(y)\theta(x-y) \dd y - {P.V.} \,\int_{|y|\leq 2\xi} \frac{y^\bot}{|y|}H_\beta(y) \theta(x+\xi e- y) \dd y ,
\end{equation}
and
\begin{equation*}
\begin{split}
  I_2(\xi) \equiv &\int_{|y|\geq 2\xi} \frac{y^\bot}{|y|}H_\beta(y) \theta(x-y) \dd y - \int_{|y|\geq 2\xi} \frac{y^\bot}{|y|}H_\beta(y) \theta(x+\xi e- y) \dd y \\
  = & \int_{|x-y|\geq 2\xi} \frac{(x-y)^\bot}{|x-y|}H_\beta(x-y) \theta(y) \dd y - \int_{|x+\xi e-y|\geq 2\xi} \frac{(x+\xi e-y)^\bot}{|x+\xi e-y|}H_\beta(x+\xi e-y) \theta(y) \dd y.
\end{split}
\end{equation*}
The term $I_2(\xi) $ is controlled as usually (see e.g. \cite{DKV}), and it is bounded by
\begin{equation}\label{IIes}
  C_0\xi\int_{\xi}^\infty \frac{\omega(\eta)m(\eta^{-1})}{\eta^{1+\beta}}{d}\eta + C_0 \xi^{1-\beta}m(\xi^{-1})\omega(\xi),
\end{equation}
where $C_0$ is a fixed constant that does not depend on $\alpha,\beta$.
To estimate $I_1(\xi)$, we observe that, thanks to the zero-average property of $\frac{y^\bot}{|y|}H_\beta(y)$ and the scenario \eqref{scena}, we have
\begin{equation*}
\begin{split}
  I_1(\xi) & =\, \int_{|y|\leq 2\xi} \frac{y^\bot}{|y|}H_\beta(y)(\theta(x-y)-\theta(x)) \dd y
  -\,\int_{|y|\leq 2\xi} \frac{y^\bot}{|y|}H_\beta(y) (\theta(x+\xi e- y)-\theta(x+\xi e)) \dd y  \\
  & = \, \int_{|y|\leq 2\xi} \frac{y^\bot}{|y|}H_\beta(y)\big(\theta(x-y)-\theta(x+\xi e -y) -\omega(\xi)\big) \dd y ,
\end{split}
\end{equation*}
where the integrals have to be understood in the principle value sense if needed.
Recalling that $D(\xi)$ defined by \eqref{GC-D} can be rewritten as (\cite[Thm 1]{DI} or \cite[Prop 2.1]{CC})
\begin{equation}\label{Dexp}
\begin{split}
  D(\xi) & = C_\beta \left( {P.V.} \int_{\R^2} \frac{\theta(x-y)-\theta(x)}{|y|^{2+\beta}} \dd y
  - {P.V.}\int_{\R^2} \frac{\theta(x+\xi e-y )-\theta(x+\xi e)}{|y|^{2+\beta}} \dd y\right)\\
  & = C_\beta\int_{\R^2} \frac{1}{|y|^{2+\beta}} \big( \theta(x-y)-\theta(x+\xi e- y)-\omega(\xi)\big)\dd y ,
\end{split}
\end{equation}
with $C_\beta= \frac{\beta\, \Gamma(1+\beta/2)}{2\pi^{1+\beta} \Gamma(1-\beta/2)}(\geq \frac{\beta}{C_0})$, we obtain that for some constant $B >0$ to be chosen later,
\begin{equation*}
\begin{split}
   I_1(\xi)+B \xi m(\xi^{-1}) D(\xi) \leq \, & \, \int_{|y|\leq 2\xi} \left( \frac{y^\bot}{|y|}H_\beta(y) - C_\beta B \xi m(\xi^{-1}) \frac{1}{|y|^{2+\beta}}\right)
  \big(\omega(\xi) +\theta(x+\xi e -y)-\theta(x-y)\big) \dd y \\
  \leq \, & \, \int_{|y|\leq 2\xi} \left(C_0' \frac{m(|y|^{-1})}{|y|^{1+\beta}} -   C_\beta B \frac{ \xi m(\xi^{-1})}{|y|^{2+\beta}}\right)
  \big(\omega(\xi) +\theta(x+\xi e -y)-\theta(x-y)\big) \dd y \\
  \leq \, & \, \int_{|y|\leq 2\xi} \left( 2C_0' - C_\beta B\right)\frac{\xi m( \xi^{-1})}{|y|^{2+\beta}}
  \big(\omega(\xi) +\theta(x+\xi e +y)-\theta(x+y)\big) \dd y,
\end{split}
\end{equation*}
where in the third line we used that
\begin{equation}\label{eq:fact-m}
  |y| m(|y|^{-1}) \leq (2\xi) m((2\xi)^{-1})\leq 2 \xi m(\xi^{-1}) \quad \textrm{for all  } 0<|y|\leq 2\xi.
\end{equation}
Thus by choosing $B=\frac{4C_0'}{C_\beta} $, we get
\begin{equation}\label{I-es1}
  |I_1(\xi)| \leq - B \xi m(\xi^{-1}) D(\xi),
\end{equation}
which combined with \eqref{IIes} leads to the desired inequality \eqref{Ome-es1}. \\

In order to prove (2), we first observe that since only the case $\xi\leq \frac{1}{2b_1}$ is considered, by using \eqref{m-fac} we see that we still have \eqref{eq:fact-m} and \eqref{I-es1}, thus collecting \eqref{I-es1} and \eqref{IIes} yields \eqref{Ome-es1-2}.

 The proof of \eqref{Ome-es2} is classical, see e.g. \cite{DKV}, and we omit the details.
\end{proof}

\section{Proof of Theorem \ref{thm:GE}: Global existence of weak solution for the inviscid equation \eqref{eq:gSQG}}\label{sec:ws}

We consider the following approximated system (by adding a vanishing viscosity term) of the $(gSQG)_\beta$ equation \eqref{eq:gSQG} in the inviscid case (that is  $\nu=0$):
\begin{equation}\label{ap-gSQG}
  \partial_t\theta^\epsilon + u^\epsilon \cdot\nabla \theta^\epsilon - \epsilon \Delta\theta^\epsilon =0,
  \quad u^\epsilon=\nabla^\perp \D^{\beta-2}m(\Lambda)(\theta^\epsilon),
  \quad  \theta^\epsilon_0=\phi_\epsilon*\theta_0,
\end{equation}
where $\phi_\epsilon(x)=\epsilon^{-d}\phi(\epsilon^{-1}x)$,
and $\phi\in C^\infty_c(\R^2)$ is a radial test function satisfying $\int_{\R^2}\phi=1$. \\
 For the initial data, we observe that since $\theta_0\in L^1\cap L^2(\R^2)$, by Young's inequality we have
$$\|\theta_0^\epsilon\|_{L^1\cap L^2(\R^2)} \leq \|\theta_0\|_{L^1\cap L^2(\R^2)} \ \ {\rm{and}} \ \ \|\theta_0^\epsilon\|_{H^s}\lesssim_{\epsilon,s}\|\theta_0\|_{L^2}, \ \ {\rm{for \ every}} \ s>0. $$

\noindent We have the following global well-posedness result for the approximated system \eqref{ap-gSQG}.
\begin{proposition}\label{prop:glob}
  Let $\epsilon>0$, $\beta\in (0,1]$, and $m$ satisfying (A1)-(A4).
Then, the Cauchy problem for the approximated drift-diffusion equation \eqref{ap-gSQG} admits a unique solution $\theta^\epsilon(x,t)$ such that
$\theta^\epsilon\in C([0,\infty); H^s(\R^2))\cap C^\infty((0,\infty)\times \R^2)$ where $s>2$.
\end{proposition}
Since the local existence part has already been done in \cite{XZ} (see Prop 3.1), we just need to prove the global existence part, and this is done in the Appendix section. \\

Using the usual $L^p$-estimate for transport-diffusion equation (see e.g. \cite{CC}), we have
\begin{equation}\label{L1L2bdd}
  \|\theta^\epsilon(t)\|_{L^1\cap L^2(\R^2)} \leq \|\theta_0\|_{L^1\cap L^2(\R^2)} \quad \textrm{for all  $t\geq 0$,}
\end{equation}
that is, $\theta^\epsilon\in L^\infty(\R^+; L^1\cap L^2(\R^2))$ uniformly in $\epsilon$. Since $L^\infty(\R^+; L^2(\R^2))$ is the dual space of the separable Banach space $L^1(\R^+; L^2(\R^2))$,
we can extract a subsequence $\{\theta^{\epsilon_k}\}_{k\geq0}$ from these solutions $\{\theta^\epsilon\}_{\epsilon>0}$ so that $\theta^{\epsilon_k}$
 converges $*$-weakly to some function $\theta$ in $L^\infty(\R^+; L^2(\R^2))$ (as $k\rightarrow \infty$ and $\epsilon_k\rightarrow 0$) and also in $\mathcal{D}'(\R^+\times \R^2)$.
Actually, the weak convergence of $\theta^{\epsilon_k}$ is not enough to conclude, since we only have $\partial_t\theta + \displaystyle\lim_{k\rightarrow \infty} (\nabla\cdot (u^{\epsilon_k} \theta^{\epsilon_k}))=0$ in $\mathcal{D}'(\R^+\times\R^2)$.

Therefore, it remains to show that the nonlinear term $\nabla\cdot (u^{\epsilon_k}\theta^{\epsilon_k})$ converges to $\nabla\cdot(u\theta)$ in $\mathcal{D}'(\R^+\times\R^2)$ where $u=\nabla^\perp \D^{\beta-2} m(\Lambda)\theta$.
Let $\varphi\in \mathcal{D}(\R^+\times \R^2)$ be a test function, and there exist two nonnegative numbers $R,T$ such that the support of $\varphi$ is contained in $(0,T)\times B_R(0)$.
We will prove that as $k\rightarrow \infty$,
\begin{equation}\label{Targ-conv}
  \int_0^\infty \int_{\R^2} (\theta^{\epsilon_k} u^{\epsilon_k})\cdot\nabla \varphi(x,t) \dd x\dd t \rightarrow \int_0^\infty \int_{\R^2} (\theta u)\cdot\nabla \varphi(x,t) \dd x\dd t.
\end{equation}
\\

Let us set $\psi^\epsilon\equiv \D^{\beta-2} m(\Lambda) \theta^\epsilon$ which is the associated stream function,then, with this notation we infer that
\begin{equation}\label{u-psi-the}
  u^\epsilon =\nabla^\perp \psi^\epsilon,\quad \textrm{and}\quad\theta^\epsilon = \frac{\D^{2-\beta}}{m(\Lambda)}\psi^\epsilon.
\end{equation}
Recalling that $\mathcal{H}_0={Id}-S_0$ ($S_0$ is defined as \eqref{eq:LPop}), and since $\theta^\epsilon$ is a uniformly bounded sequence in the space $L^\infty(0,T; L^1\cap L^2(\R^2))$,
we find that
\begin{equation}\label{eq:psiEs3}
  \mathcal{H}_0\psi^\epsilon \in L^\infty(0,T; H^{2-\alpha-\beta}(\R^2))\quad \textrm{uniformly in $\epsilon$}.
\end{equation}
As a matter of fact, by using the property of the support of the Fourier transform of $\mathcal{H}_0 \psi^\epsilon$ and \eqref{m-fac2}, we obtain
\begin{equation}\label{eq:psiEs2}
\begin{split}
  \|\mathcal{H}_0\psi^\epsilon(t)\|_{H^{2-\alpha-\beta}(\R^2)} & \leq  C_0 \||\zeta|^{2-\alpha-\beta}\widehat{\psi^\epsilon}(\zeta,t)\|_{L^2(B_{1/2}^c)} \\
  & \leq C_0 \||\zeta|^{-\alpha}m(\zeta) \widehat{\theta^\epsilon}(\zeta,t)\|_{L^2(B_{1/2}^c)}  \\
  & \leq C_0 \||\zeta|^{-\alpha}m(\zeta) \widehat{\theta^\epsilon}(\zeta,t)\|_{L^2(\{1/2\leq |\zeta|\leq b_1\})} +  C_0 \||\zeta|^{-\alpha}m(\zeta) \widehat{\theta^\epsilon}(\zeta,t)\|_{L^2(B_{b_1}^c)} \\
  & \leq C_0 m(b_1) \|\theta^\epsilon(t)\|_{L^2(\R^2)} + C_0 b_1^{-\alpha}m(b_1)  \|\theta^\epsilon(t)\|_{L^2(\R^2)} \\
  & \leq C_0 m(b_1)    \|\theta_0\|_{L^2(\R^2)}.
\end{split}
\end{equation}
We then claim that,
\begin{equation}\label{the-td2}
  \partial_t \theta^\epsilon \in L^\infty(0,T; H^{-5}(\R^2))\quad \textrm{uniformly in $\epsilon$}.
\end{equation}
Indeed, by integrating by parts we infer that for any $\phi\in \mathcal{D}(\R^2)$, we have
\begin{align}\label{the-td-es1}
  \left|\int_{\R^2} \partial_t \theta^\epsilon(x,t) \phi(x)\dd x\right| & \leq \left| \int_{\R^2} u^\epsilon\cdot\nabla \theta^\epsilon(x,t)\, \phi(x) \dd x \right|
  + \epsilon \left|\int_{\R^2} \Delta \theta^\epsilon(x,t) \,\phi(x)\dd x \right| \nonumber\\
  & \leq \left| \int_{\R^2} \big(\theta^\epsilon u^\epsilon(x,t)\big)\cdot\nabla \phi(x) \dd x \right|
  + \epsilon \left|\int_{\R^2} \theta^\epsilon(x,t) \,\Delta\phi(x)\dd x \right| \\
  & \leq  \left| \int_{\R^2} \big(\theta^\epsilon u^\epsilon(x,t)\big)\cdot\nabla \phi(x) \dd x \right| + \|\theta_0\|_{L^2(\R^2)} \|\Delta\phi\|_{L^2(\R^2)}.\nonumber
\end{align}
Using \eqref{u-psi-the} and by integrating by parts, we find the following decomposition
\begin{align}\label{eq:decom2}
  & \int_{\R^2}  \theta^\epsilon u^\epsilon \cdot\nabla\phi \,\dd x \nonumber \\
  = & \int_{\R^2}  \theta^\epsilon \,(S_0 u^\epsilon \cdot \nabla )\phi\,\dd x +
  \int_{\R^2} S_0 \theta^\epsilon \,(\mathcal{H}_0 u^\epsilon \cdot \nabla )\phi\,\dd x  +  \int_{\R^2} \mathcal{H}_0 \theta^\epsilon \,(\mathcal{H}_0 u^\epsilon \cdot \nabla )\phi\,\dd x \nonumber\\
  = & \int_{\R^2}  \theta^\epsilon \,(S_0 u^\epsilon \cdot \nabla )\phi\,\dd x -
  \int_{\R^2} \mathcal{H}_0 \psi^\epsilon \,(\nabla^\perp S_0 \theta^\epsilon \cdot \nabla )\phi\,\dd x  +  \int_{\R^2} \mathcal{H}_0 \theta^\epsilon \,(\mathcal{H}_0 u^\epsilon \cdot \nabla )\phi\,\dd x   .
\end{align}
It follows from \eqref{L1L2bdd} and the continuous embedding $L^1\cap L^2(\R^2)\hookrightarrow L^{\frac{2}{2-\lambda-\beta}}(\R^2)\hookrightarrow \dot H^{\lambda+\beta-1}(\R^2)$
valid for $\lambda\in [0,1-\beta]$, $\beta\in (0,1]$ that
\begin{align*}
  \left|\int_{\R^2} \theta^\epsilon \,(S_0 u^\epsilon \cdot \nabla) \phi\,\dd x\right|(t)
  & \leq \|\theta^\epsilon(t)\|_{L^2} \|\nabla^\perp \D^{\beta-2}m(\Lambda)S_0 \theta^\epsilon(t)\|_{L^2} \|\nabla \phi\|_{L^\infty(\R^2)} \\
  & \leq C_0 \|\theta_0\|_{L^2} \|\D^{\lambda+\beta-1}S_0 \theta^\epsilon(t)\|_{L^2} \|\phi\|_{H^3(\R^2)} \\
  & \leq
  \begin{cases}
  C_0 \|\theta_0\|_{L^1\cap L^2(\R^2)}^2 \|\phi\|_{H^3(\R^2)}, \quad & \textrm{for  }\lambda\in [0,1-\beta], \beta\in (0,1],\\
  C_0 \|\theta_0\|_{L^2(\R^2)}^2 \|\phi\|_{H^3(\R^2)},\quad & \textrm{for  }\lambda\in [1-\beta,1), \beta\in (0,1],
  \end{cases}
\end{align*}
For the second term, using H\"older's inequality, Plancherel's theorem and \eqref{eq:psiEs2}, we find
\begin{align*}
   \left|\int_{\R^2} \mathcal{H}_0 \psi^\epsilon (\nabla^\perp S_0 \theta^\epsilon  \cdot \nabla) \phi\,\dd x\right|(t)
  & \leq \|\mathcal{H}_0\psi^\epsilon(t)\|_{L^2} \|\nabla^\perp S_0 \theta^\epsilon(t)\|_{L^2} \|\nabla \phi\|_{L^\infty(\R^2)} \\
  & \leq C_0 \|\mathcal{H}_0 \theta^\epsilon(t)\|_{L^2} \|\theta^\epsilon(t)\|_{L^2} \|\phi\|_{H^3(\R^2)} \\
  & \leq C_0 \|\theta_0\|_{L^2(\R^2)}^2 \|\phi\|_{H^3(\R^2)}.
\end{align*}
In order to estimate the third term in \eqref{eq:decom2}, we need the following lemma which gives a new expression of the convection term in terms of the stream function via a controlled commutator.
\begin{lemma}\label{lem:exp0}
   For every $\phi\in\mathcal{D}(\R^2)$, we have the following equality
\begin{equation}\label{eq:Exp}
  \int_{\R^2}  \theta^\epsilon u^\epsilon \cdot\nabla\phi \,\dd x =
  \frac{1}{2} \int_{\R^2} \psi^\epsilon \, \left(\left[\frac{\D^{2-\beta}\nabla^\perp \cdot}{m(\Lambda)},\nabla\phi \right]\psi^\epsilon \right)\,\dd x.
\end{equation}
\end{lemma}

\begin{proof}[{\bf{Proof of Lemma \ref{lem:exp0}}}]
On one hand, by using to \eqref{u-psi-the}, we see that
\begin{eqnarray}\label{eq:vt-es1}
   \int_{\R^2} \theta^\epsilon u^\epsilon \cdot\nabla\phi\,\dd x &=& \nonumber \int_{\R^2} \big(\theta^\epsilon\nabla^\perp \psi^\epsilon\big)\cdot\nabla \phi \, \dd x  \\
  & =& \nonumber-\int_{\R^2} \psi^\epsilon\, \nabla^\perp \theta^\epsilon \cdot \nabla \phi\, \dd x \\
  &=&  -\int_{\R^2} \psi^\epsilon \,\left(\frac{\D^{2-\beta} \nabla^\perp}{m(\Lambda)}\psi^\epsilon\right)\cdot \nabla \phi \, \dd x,
\end{eqnarray}
On the other hand, we have
\begin{equation}\label{eq:vt-es2}
\begin{split}
  & \int_{\R^2} \theta^\epsilon u^\epsilon \cdot\nabla\phi\,\dd x
  =  \int_{\R^2} \left(\frac{\D^{2-\beta}}{m(\Lambda)}\psi^\epsilon\right)\, \nabla^\perp\psi^\epsilon \cdot\nabla \phi\,\dd x  \\
  & =  \int_{\R^2} \psi^\epsilon\, \left(\frac{\D^{2-\beta}}{m(\Lambda)}\left(\nabla^\perp\psi^\epsilon\cdot \nabla \phi\right)\right)\,\dd x
  =  \int_{\R^2} \psi^\epsilon\, \left(\frac{\D^{2-\beta}\nabla^\perp\cdot}{m(\Lambda)}\big( \psi^\epsilon \nabla \phi \big)\right)\,\dd x,
\end{split}
\end{equation}
To conclude, it suffices to sum the two previous equalities to find the desired commutator  \eqref{eq:Exp}.
\end{proof}
\noindent Then, thanks to Lemma \ref{lem:exp0}, \eqref{eq:comm}, \eqref{eq:psiEs2} and the fact $\frac{\Lambda^{2-\beta}}{m(\Lambda)}\psi^\epsilon= \theta^\epsilon$, we get
\begin{align*}
  \left|\int_{\R^2} \mathcal{H}_0 \theta^\epsilon \,\mathcal{H}_0 u^\epsilon \cdot \nabla \phi\,\dd x \right|(t)
  = & \frac{1}{2} \left|\int_{\R^2} \mathcal{H}_0\psi^\epsilon  \left(\left[\frac{\D^{2-\beta}\nabla^\perp \cdot}{m(\Lambda)},\nabla\phi \right]\mathcal{H}_0\psi^\epsilon \right)\dd x\right|(t) \\
  \leq & C\|\mathcal{H}_0\psi^\epsilon(t)\|_{L^2}\left(\|\mathcal{H}_0 \psi^\epsilon(t)\|_{L^2} + \|\theta^\epsilon(t)\|_{L^2} \right) \|\phi\|_{H^5(\R^2)} \\
  \leq & C \|\theta_0\|_{L^2(\R^2)}^2 \|\phi\|_{H^5(\R^2)}.
\end{align*}
Collecting \eqref{the-td-es1} and the above estimates allows us to conclude that for all $t\in [0,T]$, we have
\begin{align}
  \left|\int_{\R^2} \partial_t \theta^\epsilon(x,t) \phi(x)\dd x\right|\leq C \|\phi\|_{H^5(\R^2)},
\end{align}
which gives \eqref{the-td2}. \\

\noindent Then, using \eqref{the-td2}, one may further show that
\begin{equation}\label{psi-td2}
  \partial_t \mathcal{H}_0\psi^\epsilon \in L^\infty(0,T; H^{-5}(\R^2))\quad \textrm{uniformly in $\epsilon$}.
\end{equation}
Indeed, for all $\phi\in H^5(\R^2)$, using \eqref{the-td2} and the fact that $m(\zeta)\leq C_{b_1} |\zeta|^\alpha$ for all $|\zeta|\geq 1/2$, we find
\begin{equation*}
\begin{split}
  \left|\int_{\R^2} \partial_t \mathcal{H}_0\psi^\epsilon(x,t)\, \phi(x) \dd x\right| & = \left|\int_{\R^2} \partial_t \theta^\epsilon(x,t) \, \D^{\beta-2}m(\Lambda)\mathcal{H}_0\phi(x) \dd x\right| \\
  & \leq  \|\partial_t\theta^\epsilon(t)\|_{H^{-5}(\R^2)} \|\D^{\beta-2}m(\Lambda)\mathcal{H}_0\nabla\phi\|_{H^4(\R^2)} \\
  & \leq C \|\Lambda^{\beta+\alpha-1} \mathcal{H}_0 \phi\|_{H^4(\R^2)} \leq C \|\phi\|_{H^5(\R^2)},
\end{split}
\end{equation*}
therefore, we find \eqref{psi-td2}.\\

\noindent According to \eqref{eq:psiEs3}, \eqref{psi-td2}, the Ascoli's theorem and the weak convergence of $\theta^\epsilon$, we infer that $\mathcal{H}_0\psi^{\epsilon_k}$
up to a subsequence (still denoted $\mathcal{H}_0\psi^{\epsilon_k}$) satisfies, as $k\rightarrow \infty$, that
\begin{equation}\label{psiConv2}
  \mathcal{H}_0\psi^{\epsilon_k} \rightarrow \mathcal{H}_0\psi = \D^{\beta-2}m(\Lambda)\mathcal{H}_0\theta,\quad \textrm{in $L^\infty(0,T;L^2_{\mathrm{loc}}(\R^2))$},
\end{equation}
with $\mathcal{H}_0\psi \in L^\infty(0,T; H^{2-\alpha-\beta}(\R^2))$. \\

Next, we shall prove \eqref{Targ-conv}. For the right-hand side of \eqref{Targ-conv}, we use the decomposition \eqref{eq:decom2}.
Since we work on the whole space $\R^2$, and the strong convergence result \eqref{psiConv2} only holds on a compact domain, we need to make a suitable splitting with respect to the space variable.
Recalling that the support of the test function $\varphi$ lies in $(0,T)\times B_R$,
we shall define 3 other test functions namely $\eta$, $\rho$ and $\chi$ as follows. We first introduce $\eta\in \mathcal{D}(\R^2)$ such that $\eta\equiv 1$ on $B_R$ and $\mathrm{supp}\, \eta \subset B_{2R}$;
and then $\rho\in \mathcal{D}(\R^2)$ such that $\rho\equiv 1$ on $B_{2R}$ and $\mathrm{supp}\, \rho \subset B_{4R}$, and finally $\chi \in \mathcal{D}(\R^2)$ which is such that $\chi\equiv 1$ on $B_{8R}$.
Using these notations, we have
\begin{align}\label{eq:decom3}
  \int_0^T\int_{\R^2} \theta^{\epsilon_k} \,(S_0 u^{\epsilon_k} \cdot \nabla) \varphi\,\dd x \dd t =
  & \int_0^T\int_{\R^2}\theta^{\epsilon_k} \,\eta\left(\nabla^\perp \D^{\beta-2}m(\Lambda)\big( (S_0 \theta^{\epsilon_k}) (1-\chi)\big)\right) \cdot \nabla\varphi\,\dd x \dd t \nonumber\\
  & + \int_0^T\int_{\R^2} \chi\theta^{\epsilon_k} \,\left(\nabla^\perp \D^{\beta-2}m(\Lambda) \big( (S_0\theta^{\epsilon_k})\chi\big)\right) \cdot \nabla\varphi\,\dd x \dd t \nonumber\\
\equiv &\,\mathcal{I}_1(\theta^{\epsilon_k}) + \mathcal{I}_2(\theta^{\epsilon_k}),
\end{align}
and
\begin{align}\label{eq:decom4}
  \int_0^T\int_{\R^2} \mathcal{H}_0 \psi^{\epsilon_k}\,(\nabla^\perp S_0 \theta^{\epsilon_k}  \cdot \nabla) \varphi\,\dd x\dd t = & \int_0^T\int_{\R^2} (\mathcal{H}_0\psi^{\epsilon_k})
  \,\eta \left(\nabla^\perp S_0 \big(\theta^{\epsilon_k} (1-\chi)\big)\right) \cdot \nabla\varphi\,\dd x \dd t \nonumber \\
  & + \int_0^T\int_{\R^2} \eta(\mathcal{H}_0\psi^{\epsilon_k}) \,\big(\nabla^\perp S_0 ( \theta^{\epsilon_k}\chi)\big) \cdot \nabla\varphi\,\dd x \dd t \nonumber \\
 \equiv & \, \mathcal{I}_3(\theta^{\epsilon_k},\psi^{\epsilon_k}) + \mathcal{I}_4(\theta^{\epsilon_k},\psi^{\epsilon_k}).
\end{align}
We may rewrite the last term of the right-hand side of \eqref{eq:decom2}. Indeed, we have the following lemma.
\begin{lemma}\label{lem:exp2}
We have the following equality
\begin{align}\label{eq:Exp2}
  \int_0^T\int_{\R^2} \mathcal{H}_0\theta^{\epsilon_k}\, \mathcal{H}_0u^{\epsilon_k} \cdot\nabla \varphi\,\dd x \dd t
  = \,& \frac{1}{2} \int_0^T \int_{\R^2} (\mathcal{H}_0\psi^{\epsilon_k}) \rho  \left(\left[\frac{\D^{2-\beta}\nabla^\perp\cdot}{m(\Lambda)} , \nabla \varphi \right]\big( (\mathcal{H}_0\psi^{\epsilon_k})\chi\big)\right)\, \dd x \dd t \nonumber\\
  & -  \int_0^T \int_{\R^2} (\mathcal{H}_0\psi^{\epsilon_k}) \rho \left(\frac{\D^{2-\beta}\nabla^\perp}{m(\Lambda)}(\mathcal{H}_0\psi^{\epsilon_k})(1-\chi)\right) \cdot \nabla \varphi\, \dd x \dd t \nonumber \\
  & -  \int_0^T \int_{\R^2}  (\mathcal{H}_0\psi^{\epsilon_k}) \eta \left(\frac{\D^{2-\beta}\nabla^\perp}{m(\Lambda)}(\mathcal{H}_0\psi^{\epsilon_k}) (1-\rho)\right)\cdot \nabla \varphi\, \dd x \dd t \nonumber \\
  \equiv\, & \mathcal{I}_5(\psi^{\epsilon_k}) + \mathcal{I}_6(\psi^{\epsilon_k}) + \mathcal{I}_7(\psi^{\epsilon_k}).
\end{align}
\end{lemma}

\begin{proof}[{\bf Proof of Lemma \ref{lem:exp2}}]
From \eqref{eq:vt-es1}, we see that
\begin{align*}
  \int_0^T\int_{\R^2} \mathcal{H}_0\theta^{\epsilon_k}\, \mathcal{H}_0 u^{\epsilon_k}\cdot\nabla \varphi\,\dd x \dd t
  = & - \int_0^T \int_{\R^2} (\mathcal{H}_0\psi^{\epsilon_k}) \rho \left(\frac{\D^{2-\beta}\nabla^\perp}{m(\Lambda)} \mathcal{H}_0\psi^{\epsilon_k}\right) \cdot \nabla \varphi\, \dd x\dd t \\
  = & - \int_0^T \int_{\R^2} (\mathcal{H}_0\psi^{\epsilon_k}) \rho \left(\frac{\D^{2-\beta}\nabla^\perp}{m(\Lambda)} (\mathcal{H}_0\psi^{\epsilon_k}) \chi\right) \cdot \nabla \varphi \dd x\dd t\\
  &\,-  \int_0^T \int_{\R^2} (\mathcal{H}_0\psi^{\epsilon_k}) \rho \left(\frac{\D^{2-\beta}\nabla^\perp}{m(\Lambda)}(\mathcal{H}_0\psi^{\epsilon_k})(1-\chi)\right) \cdot \nabla \varphi \dd x\dd t.
\end{align*}
On the other hand, from \eqref{eq:vt-es2}, we also get
\begin{equation*}
\begin{split}
  \int_0^T\int_{\R^2} \mathcal{H}_0\theta^{\epsilon_k} \, \mathcal{H}_0 u^{\epsilon_k} \cdot\nabla \varphi\,\dd x \dd t
  = & \int_0^T \int_{\R^2} \mathcal{H}_0\psi^{\epsilon_k}  \left(\frac{\D^{2-\beta}\nabla^\perp\cdot}{m(\Lambda)}\big(\mathcal{H}_0\psi^{\epsilon_k} \nabla \varphi\,\big)\right) \dd x\dd t \\
  = & \int_0^T \int_{\R^2} (\mathcal{H}_0\psi^{\epsilon_k}) \rho \left( \frac{\D^{2-\beta}\nabla^\perp\cdot}{m(\Lambda)}\big( (\mathcal{H}_0\psi^{\epsilon_k})\chi \,\nabla \varphi\big)\right) \dd x\dd t \\
  & - \int_0^T \int_{\R^2} \left(\frac{\D^{2-\beta}\nabla^\perp}{m(\Lambda)} (\mathcal{H}_0\psi^{\epsilon_k}) (1-\rho)\right) \cdot\big((\mathcal{H}_0\psi^{\epsilon_k}) \nabla \varphi\big) \dd x\dd t .
\end{split}
\end{equation*}
Summing the above two equalities yields the desired formula \eqref{eq:Exp2}.
\end{proof}

\noindent Thus the left-hand side of \eqref{Targ-conv} can be decomposed as
\begin{equation}\label{eq:decom0}
  \int_0^T \int_{\R^2} (\theta^{\epsilon_k} u^{\epsilon_k})\cdot\nabla \varphi \dd x\dd t = \sum_{i=1,2} \mathcal{I}_i(\theta^{\epsilon_k})
  + \sum_{i=3,4}\mathcal{I}_i(\theta^{\epsilon_k},\psi^{\epsilon_k}) + \sum_{i=5,6,7} \mathcal{I}_i(\psi^{\epsilon_k}).
\end{equation}
Similarly the right-hand side of \eqref{Targ-conv} has the following decomposition
\begin{equation}\label{eq:decom1}
  \int_0^T \int_{\R^2} \theta u \cdot\nabla \varphi \dd x\dd t = \sum_{i=1,2} \mathcal{I}_i(\theta)
  + \sum_{i=3,4}\mathcal{I}_i(\theta,\psi) + \sum_{i=5,6,7} \mathcal{I}_i(\psi).
\end{equation}
We first prove that $\mathcal{I}_1(\theta^{\epsilon_k})\rightarrow \mathcal{I}_1(\theta)$.
We observe that
\begin{equation}\label{I1decom}
\begin{split}
  \mathcal{I}_1(\theta^{\epsilon_k})- \mathcal{I}_1(\theta)= & \int_0^\infty \int_{\R^2}\theta^{\epsilon_k} \eta \left(\nabla^\perp \D^{\beta-2}m(\Lambda) \big((1-\chi)(S_0\theta^{\epsilon_k}-S_0\theta)\big) \right)\cdot\nabla \varphi(x,t)\, \dd x\dd t\\
  &\, + \int_0^\infty \int_{\R^2}(\theta^{\epsilon_k}-\theta) \eta \left(\nabla^\perp \D^{\beta-2}m(\Lambda) \big((1-\chi)S_0\theta\big) \right)\cdot\nabla \varphi(x,t)\, \dd x\dd t \\
  \equiv\, & I^{\epsilon_k}_{1,1} + I^{\epsilon_k}_{1,2}.
\end{split}
\end{equation}
Let us set $h_j^{\epsilon_k}\equiv \eta \left(\partial_j \D^{\beta-2}m(\Lambda)\big((1-\chi) S_0\theta^{\epsilon_k}\big) \right)$ for $j=1,2$, then we claim that as $k\rightarrow\infty$, we have the following convergence (up to subsequence that is still denoted $h_j^{\epsilon_k}$)
\begin{equation}\label{eq:claim}
  \textrm{$h_j^{\epsilon_k}$ strongly converges to $h_j\equiv \eta \big(\partial_j \D^{\beta-2}m(\Lambda)\big((1-\chi)S_0\theta\big) \big)$ in $L^\infty(0,T; L^2(B_R))$},
\end{equation}
Indeed, one first notices that $K_{\beta,j}$ ($j=1,2$) (which appeared in Lemma \ref{lem:kernel}) is the kernel function of $\partial_j \D^{\beta-2}m(\Lambda)$, then, thanks to \eqref{kernel1}-\eqref{kernel2} along with the support property, one finds that for all $x\in B_{2R}$,
\begin{align*}
  \left|\partial_j \D^{\beta-2} m(\Lambda)\big((1-\chi)S_0\theta^{\epsilon_k}\big)(x)\right| & = \left| \int_{\R^2} K_{\beta,j}(x-y) \big((1-\chi)S_0\theta^{\epsilon_k}\big)(y)\,\dd y\right| \\
  & \leq C \int_{|y|\geq 8R} \frac{m(|x-y|^{-1})}{|x-y|^{1+\beta}} |S_0\theta^{\epsilon_k}(y)| \,\dd y \\
  & \leq C m(R^{-1}) \int_{|y|\geq 8R} \frac{1}{|y|^{1+\beta}} |S_0\theta^{\epsilon_k}(y)| \, \dd y \\
  & \leq C \frac{m(R^{-1})}{R^\beta} \|S_0\theta^{\epsilon_k}\|_{L^2(\R^2)} \leq C \|\theta_0\|_{L^2(\R^2)},
\end{align*}
and
\begin{equation*}
\begin{split}
  \left|\nabla\partial_j \D^{\beta-2} m(\Lambda)\big((1-\chi)S_0\theta^{\epsilon_k}\big)(x)\right| & = \left| \int_{\R^2} \nabla K_{\beta,j}(x-y) \big((1-\chi)S_0\theta^{\epsilon_k}\big)(y)\,\dd y\right| \\
  & \leq C \int_{|y|\geq 8R} \frac{m(|x-y|^{-1})}{|x-y|^{2+\beta}} |S_0\theta^{\epsilon_k}(y)| \,\dd y \leq C \|\theta_0\|_{L^2(\R^2)},
\end{split}
\end{equation*}
thus, we obtain that
\begin{equation}\label{hj-es1}
\begin{split}
  \|h_j^{\epsilon_k}\|_{L^\infty(0,T;H^1(\R^2))}
  &\leq C_0\|\partial_j \D^{\beta-2} m(\Lambda)\big((1-\chi)S_0\theta^{\epsilon_k}\big)\|_{L^\infty(0,T;L^2(B_{2R}))}\\
  & \  \ +  C_0\|\nabla\partial_j \D^{\beta-2} m(\Lambda)\big((1-\chi)S_0\theta^{\epsilon_k}\big)\|_{L^\infty(0,T;L^2(B_{2R}))} \\
  &\leq C \|\theta_0\|_{L^2(\R^2)}.
\end{split}
\end{equation}
Then, in order to show the convergence of $h_j^{\epsilon_k}$, we need to prove some uniform bound on $\partial_t h_j^{\epsilon_k}$.
Let $\phi\in\mathcal{D}(\R^2)$ be a test function, thus
\begin{align}\label{hj-es2}
  \left|\int_{\R^2} \partial_t h_j^{\epsilon_k}(x,t)\,\phi(x) \dd x\right| &
  = \left|\int_{\R^2}\eta\left(\partial_j \D^{\beta-2}m(\Lambda) \big((1-\chi)S_0\partial_t\theta^{\epsilon_k}\big) \right)(x,t)\,\phi(x)\dd x  \right| \nonumber\\
  & = \left|\int_{\R^2} \partial_t\theta^{\epsilon_k}(x,t) \, S_0\big((1-\chi)\partial_j \D^{\beta-2}m(\Lambda)(\eta\phi)(x)\big)\, \dd x  \right| \\
  & \leq \|\partial_t \theta^{\epsilon_k}\|_{L^\infty(0,T; H^{-5}(\R^2))} \|S_0\big((1-\chi)\partial_j \D^{\beta-2}m(\Lambda)(\eta\phi)\big)\|_{H^5(\R^2)} \nonumber\\
  & \leq C \|\phi\|_{L^2(\R^2)}, \nonumber
\end{align}
where in the last line we have used the following estimate (using \eqref{kernel1}-\eqref{kernel2})
\begin{align*}
  \|S_0\big((1-\chi)\partial_j \D^{\beta-2}m(\Lambda)(\eta\phi)\big)\|_{H^5(\R^2)}
  \leq\,& C \left\|\int_{B_{2R}} |K_{\beta,j}(x-y)| \,|\eta \phi|(y)\dd y\right\|_{L^2_x(B_{8R}^c)} \\
  \leq \,& C \left\|\frac{m(|x|^{-1})}{|x|^{1+\beta}}\right\|_{L^2_x(B_{8R}^c)} \int_{B_{2R}} |\eta \phi|\dd y\leq C \|\phi\|_{L^2(\R^2)}.
\end{align*}
Hence, using  \eqref{hj-es1} and \eqref{hj-es2}, we observe that the strong convergence \eqref{eq:claim} follows from Ascoli's theorem together with the weak convergence of $\theta^{\epsilon_k}$.
Now for $I^{\epsilon_k}_{1,1}$ in \eqref{I1decom}, using H\"older's inequality and \eqref{eq:claim}, we get
\begin{equation}\label{eq:Iconv}
\begin{split}
  \lim_{k\rightarrow \infty}|I^{\epsilon_k}_{1,1}| & \leq C\lim_{k\rightarrow\infty} \|\theta^{\epsilon_k}\|_{L^2([0,T],L^2)} \|h_j^{\epsilon_k}-h_j\|_{L^2([0,T],L^2(B_R))} \\
  & \leq C T \|\theta_0\|_{L^2(\R^2)} \lim_{k\rightarrow \infty}\|h_j^{\epsilon_k}-h_j\|_{L^\infty([0,T],L^2(B_R))} =0.
\end{split}
\end{equation}
As for $I^{\epsilon_k}_{1,2}$, by \eqref{hj-es1}, we know that
 $$\left\|\eta \left(\nabla^\perp \D^{\beta-2}m(\Lambda) \big((1-\chi)S_0\theta\big) \right)\right\|_{L^2(0,T;L^2(\R^2))}\leq C,$$
hence, the weak convergence of $\theta^{\epsilon_k}$ in $L^2([0,T]\times \R^2)$ implies that $\lim_{k\rightarrow \infty} |I^{\epsilon_k}_{1,2}|=0$. By \eqref{I1decom} and these two convergence results, we obtain
\begin{equation}\label{eq:I1conv}
  \lim_{k\rightarrow \infty} \mathcal{I}_1(\theta^{\epsilon_k}) = \mathcal{I}_1(\theta).
\end{equation}

We then focus on the convergence of $\mathcal{I}_2(\theta^{\epsilon_k})$ defined in \eqref{eq:decom3}. We have
\begin{equation}\label{eq:I2conv}
\begin{split}
  |\mathcal{I}_2(\theta^{\epsilon_k}) - \mathcal{I}_2(\theta)| \leq \, & \left|\int_0^T \int_{\R^2} \chi\theta^{\epsilon_k} \left(\nabla^\perp \D^{\beta-2}m(\Lambda)
  \big( \chi S_0(\theta^{\epsilon_k}-\theta)\big) \right)\cdot\nabla \varphi(x,t) \,\dd x\dd t\right| \\
  \, & + \left|\int_0^T  \int_{\R^2} \chi(\theta^{\epsilon_k}-\theta) \left(\nabla^\perp \D^{\beta-2}m(\Lambda)
  \big( \chi (S_0\theta) \big) \right)\cdot\nabla \varphi(x,t) \,\dd x\dd t \right|.
\end{split}
\end{equation}
Following what we did in \eqref{eq:claim}, we claim that, as $k\rightarrow\infty$, we have the following convergence (up to a subsequence)
\begin{equation}\label{eq:claim2}
  \textrm{$\nabla^\perp \D^{\beta-2}m(\Lambda) \big(\chi (S_0\theta^{\epsilon_k})\big)\rightarrow\nabla^\perp \D^{\beta-2}m(\Lambda) \big(\chi( S_0\theta)\big)$ in $L^\infty([0,T]; L^2(B_R))$}.
\end{equation}
Observe that,  via \eqref{eq:claim2} and following the same arguments as what we previously did for $I^{\epsilon_k}_{1,1}$, $I^{\epsilon_k}_{1,2}$, we find
that the right-hand-sid of \eqref{eq:I2conv} converges to 0 as $k\rightarrow \infty$, this leads to
\begin{equation}
  \lim_{k\rightarrow \infty} \mathcal{I}_2(\theta^{\epsilon_k}) = \mathcal{I}_2(\theta).
\end{equation}
To estimate \eqref{eq:claim2}, we use Plancherel's theorem together with the assumption (A3) and \eqref{m-fac2}, we finally get by using Bernstein's inequality  along with \eqref{L1L2bdd}
\begin{align}\label{eq:inegalit1}
  & \left\|\nabla^\perp \D^{\beta-2}m(\Lambda) (\chi (S_0\theta^{\epsilon_k}))\right\|_{L^\infty([0,T]; H^1(\R^2))}\nonumber \\
  \leq & \,C_0 \|\D^{\beta-1}m(\Lambda) S_0(\chi(S_0 \theta^{\epsilon_k}))\|_{L^\infty([0,T]; L^2(\R^2))}
  + C_0 \|\D^{\beta-1} m(\Lambda) \mathcal{H}_0 (\chi S_0\theta^{\epsilon_k})\|_{L^\infty(0,T: H^1(\R^2))} \nonumber\\
  \leq&\nonumber \, C_0 \|\D^{\lambda+\beta-1} S_0(\chi(S_0 \theta^{\epsilon_k}))\|_{L^\infty(0,T; L^2(\R^2))}
  + C_0 \|\mathcal{H}_0 (\chi S_0\theta^{\epsilon_k})\|_{L^\infty(0,T: H^3(\R^2))} \\
  \leq & \nonumber
  \begin{cases}
  C_0 \|\chi (S_0\theta^{\epsilon_k})\|_{L^\infty(0,T;L^2)} + C_0\|\chi\|_{H^3} \|S_0 \theta^{\epsilon_k}\|_{L^\infty(0,T;H^3)},\quad & \textrm{for  }\lambda\in [1-\beta,1),\beta\in(0,1], \\
  C_0 \|\chi (S_0\theta^{\epsilon_k})\|_{L^\infty(0,T;L^{\frac{2}{2-\lambda-\beta}})} + C_0\|\chi\|_{H^3} \|S_0 \theta^{\epsilon_k}\|_{L^\infty(0,T;H^3)},\quad & \textrm{for  }\lambda\in [0,1-\beta],\beta\in(0,1],
  \end{cases}
  \\
  \leq & \|\chi\|_{H^3(\R^2)} \|\theta^{\epsilon_k}\|_{L^\infty(0,T; L^2(\R^2))}
  \leq C \|\theta_0\|_{L^2(\R^2)}.
\end{align}
Then, since  $\partial_t \theta^{\epsilon_k}\in L^\infty(0,T; H^{-5}(\R^2))$ uniformly in $\epsilon_k$, and for all $\phi\in\mathcal{D}(\R^2)$,
\begin{align*}
  & \left|\int_{\R^2} \nabla^\perp \D^{\beta-2}m(\Lambda)\big (\chi (S_0\partial_t\theta^{\epsilon_k})\big) \phi \,\dd x\right|
   \,= \left|\int_{\R^2} \partial_t \theta^{\epsilon_k}S_0 \left(\chi \big(\nabla^\perp \D^{\beta-2}m(\Lambda)\phi\big)\right)\dd x \right|,
\end{align*}
thus, we may estimate the right hand side of the previous inequality as follows,
\begin{align}\label{eq:inegalit2}
  & \left|\int_{\R^2} \partial_t \theta^{\epsilon_k}S_0 \left(\chi \big(\nabla^\perp \D^{\beta-2}m(\Lambda)\phi\big)\right)\dd x \right| \nonumber \\
  \leq\, & \|\partial_t\theta^{\epsilon_k}\|_{L^\infty(0,T; H^{-5}(\R^2))} \|S_0\big(\chi  \big(\nabla^\perp \D^{\beta-2}m(\Lambda)\phi\big)\big) \|_{H^5(\R^2)} \nonumber \\
  \nonumber \leq\, & C \|S_0(\chi (\nabla^\perp \D^{\beta-2}m(\Lambda) S_0 \phi))\|_{L^2(\R^2)} + C \|S_0 (\chi (\nabla^\perp \D^{\beta-2}m(\Lambda)\mathcal{H}_0 \phi))\|_{L^2(\R^2)} \\
  \nonumber\leq \, &
  \begin{cases}
  C \| \D^{\lambda+\beta-1} S_0 \phi\|_{L^2(\R^2)} + C \|\D^{\alpha+\beta-1}\mathcal{H}_0 \phi\|_{L^2(\R^2)},\quad & \textrm{for  }  \lambda\in [1-\beta,1), \beta\in(0,1], \\
  C \| \D^{\lambda+\beta-1}  S_0 \phi\|_{L^{\frac{2}{\lambda+\beta}}(\R^2)} + C \|\D^{\alpha+\beta-1}\mathcal{H}_0 \phi\|_{L^2(\R^2)},\quad & \textrm{for  }  \lambda\in [0,1-\beta],\beta\in(0,1],
  \end{cases}
  \\
  \leq\, & C \|S_0 \phi\|_{L^2(\R^2)} + C \|\phi\|_{H^1(\R^2)} \leq C \|\phi\|_{H^1(\R^2)}.
\end{align}
Hence, these two estimates \eqref{eq:inegalit1}-\eqref{eq:inegalit2} and the use of Ascoli's theorem along with the weak convergence of $\theta^{\epsilon_k}$ allow us to conclude the proof of the claim \eqref{eq:claim2}. \\

\noindent Now we turn to the estimation of $\mathcal{I}_3(\theta^{\epsilon_k},\psi^{\epsilon_k})$ and $\mathcal{I}_4(\theta^{\epsilon_k},\psi^{\epsilon_k})$ in \eqref{eq:decom4}. Noticing that for $j=1,2$, we have
\begin{align*}
  \partial_j S_0 f(x)= \big(\partial_j \mathcal{F}^{-1}(\widetilde{\chi})\big)*f(x),\quad \textrm{with  $\mathcal{F}^{-1}(\widetilde{\chi})\in \mathcal{S}(\R^2)$,}
\end{align*}
therefore by using \eqref{eq:psiEs3}, \eqref{psiConv2} and following the same steps as we did for the passage to the limit of $\mathcal{I}_1(\theta^{\epsilon_k})$ in \eqref{eq:I1conv}, one finds that
$\mathcal{I}_3(\theta^{\epsilon_k},\psi^{\epsilon_k})\rightarrow \mathcal{I}_3(\theta,\psi)$ as  $k\rightarrow \infty$, that is,
\begin{align}\label{eq:I3conv}
  \iint (\mathcal{H}_0\psi^{\epsilon_k})
  \,\eta \,\nabla^\perp S_0 \big(\theta^{\epsilon_k} (1-\chi)\big) \cdot \nabla\varphi\,\dd x \dd t\rightarrow \iint (\mathcal{H}_0\psi)
  \,\eta \,\nabla^\perp S_0 \big(\theta (1-\chi)\big) \cdot \nabla\varphi\,\dd x \dd t.
\end{align}
Moreover, following the same ideas as the control of $\mathcal{I}_2(\theta^{\epsilon_k})$  (see \eqref{eq:I2conv}), one analogously finds that $\mathcal{I}_4(\theta^{\epsilon_k},\psi^{\epsilon_k})\rightarrow \mathcal{I}_4(\theta,\psi)$ as  $k\rightarrow \infty$,
that is,
\begin{align}\label{eq:I4conv}
  \int_0^T\int_{\R^2} \eta(\mathcal{H}_0\psi^{\epsilon_k}) \,\big(\nabla^\perp S_0 ( \theta^{\epsilon_k}\chi)\big) \cdot \nabla\varphi\,\dd x \dd t\rightarrow\int_0^T\int_{\R^2} \eta(\mathcal{H}_0\psi) \,\big(\nabla^\perp S_0 ( \theta\chi)\big) \cdot \nabla\varphi\,\dd x \dd t.
\end{align}

Next we consider the convergence of $\mathcal{I}_i(\psi^{\epsilon_k})$, $i=5,6,7$ in \eqref{eq:Exp2}. We may write the difference as
\begin{align}\label{eq:decom5}
  \sum_{i=5,6,7}\big(\mathcal{I}_i(\psi^{\epsilon_k})- \mathcal{I}_i(\psi)\big)
  =& \frac{1}{2} \int_0^T \int_{\R^2} (\mathcal{H}_0\psi^{\epsilon_k} -\mathcal{H}_0\psi)\rho\, \left[\frac{\D^{2-\beta}\nabla^\perp \cdot}{m(\Lambda)},\nabla\varphi \right]\big((\mathcal{H}_0\psi^{\epsilon_k})\chi\big)\, \dd x \dd t \nonumber \\
  & + \frac{1}{2} \int_0^T \int_{\R^2} (\mathcal{H}_0\psi) \rho\, \left[\frac{\D^{2-\beta}\nabla^\perp \cdot}{m(\Lambda)},\nabla\varphi \right]
  \big((\mathcal{H}_0\psi^{\epsilon_k}-\mathcal{H}_0\psi)\chi\big)\, \dd x \dd t\nonumber \\
  & -  \int_0^T \int_{\R^2} (\mathcal{H}_0\psi^{\epsilon_k}-\mathcal{H}_0\psi) \rho \left(\frac{\D^{2-\beta}\nabla^\perp}{m(\Lambda)}
  \big((\mathcal{H}_0\psi^{\epsilon_k})(1-\chi)\big)\right) \cdot \nabla \varphi\, \dd x\dd t \nonumber\\
  & -  \int_0^T \int_{\R^2} (\mathcal{H}_0\psi) \rho \left(\frac{\D^{2-\beta}\nabla^\perp}{m(\Lambda)}
  \big((\mathcal{H}_0\psi^{\epsilon_k}-\mathcal{H}_0\psi)(1-\chi)\big)\right) \cdot \nabla \varphi \,\dd x\dd t \nonumber \\
  & -  \int_0^T \int_{\R^2} (\mathcal{H}_0\psi^{\epsilon_k}-\mathcal{H}_0\psi) \eta \left(\frac{\D^{2-\beta}\nabla^\perp}{m(\Lambda)}
  \big((\mathcal{H}_0\psi^{\epsilon_k})(1-\rho)\big)\right) \cdot \nabla \varphi\, \dd x\dd t \nonumber \\
  & -  \int_0^T \int_{\R^2} (\mathcal{H}_0\psi) \eta \left(\frac{\D^{2-\beta}\nabla^\perp}{m(\Lambda)}
  \big((\mathcal{H}_0\psi^{\epsilon_k}-\mathcal{H}_0\psi)(1-\rho)\big)\right) \cdot \nabla \varphi \,\dd x\dd t \nonumber\\
  \equiv\,& J_1^{\epsilon_k} + J_2^{\epsilon_k} + J_3^{\epsilon_k}+ J_4^{\epsilon_k}+J_5^{\epsilon_k}+J_6^{\epsilon_k} .
\end{align}
For the term $J_1^{\epsilon_k}$, by  \eqref{eq:comm}, \eqref{eq:prodEs}, \eqref{eq:psiEs2} and $\theta^{\epsilon_k}= \frac{\D^{2-\beta}}{m(\Lambda)} \psi^{\epsilon_k}$, we have
\begin{equation*}
\begin{split}
  |J_1^{\epsilon_k}| & \leq \frac{1}{2} \|(\mathcal{H}_0\psi^{\epsilon_k} -\mathcal{H}_0\psi)\rho\|_{L^2([0,T], L^2(\R^2))}
  \left\|\left[\frac{\D^{2-\beta}\nabla^\perp \cdot}{m(\Lambda)},\nabla\varphi \right]\big((\mathcal{H}_0\psi^{\epsilon_k})\chi\big)\right\|_{L^2([0,T], L^2(\R^2))} \\
  & \leq C  \|(\mathcal{H}_0\psi^{\epsilon_k} -\mathcal{H}_0\psi)\rho\|_{L^2([0,T], L^2)} \left( \left\|(\mathcal{H}_0\psi^{\epsilon_k})\chi\right\|_{L^2([0,T],L^2)}
  + \Big\|\frac{\D^{2-\beta}}{m(\Lambda)}\big((\mathcal{H}_0\psi^{\epsilon_k}) \chi\big)\Big\|_{L^2([0,T],L^2)}\right) \\
  & \leq C  \|(\mathcal{H}_0\psi^{\epsilon_k} -\mathcal{H}_0\psi)\rho\|_{L^2(0,T; L^2)} \left( \left\|\mathcal{H}_0\psi^{\epsilon_k}\right\|_{L^2([0,T],L^2)}
  + \left\|\mathcal{H}_0\theta^{\epsilon_k} \right\|_{L^2([0,T],L^2)}\right) (1+ \|\chi\|_{H^3}) \\
  & \leq C  T \|(\mathcal{H}_0\psi^{\epsilon_k} -\mathcal{H}_0\psi)\rho\|_{L^\infty([0,T], L^2)} ,
\end{split}
\end{equation*}
thus the local convergence \eqref{psiConv2} implies that
\begin{equation}\label{J1conv}
  \lim_{k\rightarrow \infty} |J_1^{\epsilon_k}| \leq C T \lim_{k\rightarrow \infty}\|(\mathcal{H}_0\psi^{\epsilon_k} -\mathcal{H}_0\psi)\rho\|_{L^\infty([0,T], L^2)}=0 .
\end{equation}
For the term $J_2^{\epsilon_k}$, recalling that $\mathcal{M}_{2-\beta}(\Lambda)$ is the multiplier operator given by \eqref{Ms}, and by using \eqref{eq:comm2} and \eqref{eq:prodEs2}, we obtain
\begin{align*}
   |J_2^{\epsilon_k}| & = \frac{1}{2} \left|\int_0^T \int_{\R^2} \frac{1}{\mathcal{M}_{2-\beta}(\Lambda)}\big((\mathcal{H}_0\psi) \rho\big)\,\, \left(\mathcal{M}_{2-\beta}(\Lambda)\left[\frac{\D^{2-\beta}\nabla^\perp \cdot}{m(\Lambda)},\nabla\varphi \right]
   \big((\mathcal{H}_0\psi^{\epsilon_k}-\mathcal{H}_0\psi)\chi\big)\right)\, \dd x \dd t\right| \\
  & \leq \frac{1}{2} \left\|\frac{1}{\mathcal{M}_{2-\beta}(\Lambda)}\big((\mathcal{H}_0\psi) \rho\big)\right\|_{L^2([0,T], L^2)}
  \left\| \mathcal{M}_{2-\beta}(\Lambda)\left[\frac{\Lambda^{2-\beta}\nabla^\perp \cdot}{m(\Lambda)},\nabla\varphi \right]\big((\mathcal{H}_0\psi^{\epsilon_k}-\mathcal{H}_0\psi)\chi\big)\right\|_{L^2(0,T;L^2)} \\
  & \leq C \left(\|\mathcal{H}_0\psi\|_{L^2([0,T],L^2)} + \|\mathcal{H}_0\theta\|_{L^2([0,T],L^2)}\right)\|\rho\|_{H^3} \|\varphi\|_{L^\infty([0,T],H^5)}  \|(\mathcal{H}_0\psi^{\epsilon_k}-\mathcal{H}_0\psi)\chi\|_{L^2([0,T], L^2)} \\
  & \leq C T \|(\mathcal{H}_0\psi^{\epsilon_k}- \mathcal{H}_0 \psi)\chi\|_{L^\infty([0,T], L^2(\R^2))},
\end{align*}
then \eqref{psiConv2} implies that
\begin{equation}
  \lim_{k\rightarrow \infty} |J_2^{\epsilon_k}| \leq C T \lim_{k\rightarrow \infty} \|(\mathcal{H}_0\psi^{\epsilon_k}-\mathcal{H}_0\psi)\chi\|_{L^\infty([0,T],L^2(\R^2))} = 0 .
\end{equation}

\noindent For the term $J_3^{\epsilon_k}$, we use \eqref{kernel3} and \eqref{eq:psiEs2} to infer that for all $x\in B_R$,
\begin{align}\label{eq:Est1}
  \left|\frac{\D^{2-\beta} \partial_j}{m(\Lambda)}\big((\mathcal{H}_0\psi^{\epsilon_k})(1-\chi)\big)(x)\right|
  &= \left|\int_{\R^2} \widetilde{K}_{\beta,j}(x-y) \mathcal{H}_0\psi^{\epsilon_k}(y)\big(1-\chi(y)\big) \dd y \right| \nonumber\\
  & \leq C\int_{|y|\geq 8R} \frac{1}{|x-y|^{5-\beta}m(|x-y|^{-1})} |\mathcal{H}_0\psi^{\epsilon_k}(y)| \dd y \nonumber \\
  & \leq C \int_{|y|\geq 8R} \frac{1}{|y|^{5-\beta} m(|y|^{-1})} |\mathcal{H}_0\psi^{\epsilon_k}(y)| \dd y  \nonumber \\
  & \leq C \left(\int_{|y|\geq 8R} \frac{1}{\big(|y|^{5-\beta}m(|y|^{-1})\big)^2} \dd y\right)^{1/2}  \|\mathcal{H}_0\psi^{\epsilon_k}\|_{L^2(\R^2)} \nonumber \\
  &\leq C \|\theta_0\|_{L^2(\R^2)},
\end{align}
thus, thanks to H\"older's inequality and \eqref{psiConv2} we get
\begin{equation}
\begin{split}
  \lim_{k\rightarrow \infty}|J_3^{\epsilon_k}| & \leq \lim_{k\rightarrow\infty}\|(\mathcal{H}_0\psi^{\epsilon_k}-\mathcal{H}_0\psi) \rho\|_{L^2([0,T], L^2)} \left\|\frac{\D^{2-\beta} \nabla^\perp}{m(\Lambda)}\big((\mathcal{H}_0\psi^{\epsilon_k})(1-\chi)\big)\cdot\nabla \varphi\right\|_{L^2([0,T], L^2)} \\
  & \leq C \lim_{k\rightarrow \infty}\|(\mathcal{H}_0\psi^{\epsilon_k}-\mathcal{H}_0\psi) \rho\|_{L^2([0,T], L^2(\R^2))} =0.
\end{split}
\end{equation}

\noindent Then, we consider the term $J_4^{\epsilon_k}$, and we first prove that the sequence $g^{\epsilon_k}_j\equiv\rho \left(\frac{\D^{2-\beta}\partial_j}{m(\Lambda)}\big((\mathcal{H}_0\psi^{\epsilon_k})(1-\chi)\big)\right)$
 for $j=1,2$, is locally convergent.
To obtain this convergence, we follow the same argument as in \eqref{eq:Est1} and by using \eqref{kernel3}-\eqref{kernel4}, we get
\begin{equation*}
\begin{split}
 \|g^{\epsilon_k}_j\|_{L^\infty([0,T]; H^1(\R^2))}\leq & \, C \left\|\frac{\D^{2-\beta}\partial_j}{m(\Lambda)}\big( (\mathcal{H}_0\psi^{\epsilon_k}) (1-\chi)\big)\right\|_{L^\infty([0,T], L^2(B_{4R}))}
 \\ & \ + \ C \left\|\frac{\nabla \D^{2-\beta}\partial_j}{m(\Lambda)}\big( (\mathcal{H}_0\psi^{\epsilon_k}) (1-\chi)\big)\right\|_{L^\infty([0,T], L^2(B_{4R}))} \\
\leq & \, C.
\end{split}
\end{equation*}
Besides, for any test function $\phi \in \mathcal{D}(\R^2)$, we see that, thanks to \eqref{kernel3}, we may write that for any $l=0,1,\cdots,5$,
\begin{align*}
  \left\| \frac{\Lambda^{2-\beta}\partial_j}{m(\Lambda)}\nabla^l\big(\rho \phi\big)\right\|_{L^2(B_{8R}^c)} = &\, \left\| \int_{\R^2} \widetilde{K}_{\beta,j}(x-y)\nabla^l (\phi \rho) \dd y\right\|_{L^2(B_{8R}^c)} \\
  \leq &\, C \left\| \int_{B_{4R}}  \frac{1}{|x-y|^{5-\beta}m(|x-y|^{-1})} |\nabla^l (\phi \rho)(y)| \dd y\right\|_{L^2(B_{8R}^c)} \\
  \leq &\, C \left\| \frac{1}{|x|^{5-\beta} m(|x|^{-1})} \right\|_{L^2(B_{8R}^c)} \left|\int_{B_{4R}}| \nabla^l (\phi \rho)(y)| \dd y\right| \\
  \leq &\, C \|\phi\|_{H^5(\R^2)},
\end{align*}
and by using \eqref{psi-td2}, we infer that
\begin{align*}
  \left|\int_{\R^2} \partial_t g^{\epsilon_k}_j(x,t)\, \phi(x) \dd x\right| & = \left|\int_{\R^2} \rho \left( \frac{\D^{2-\beta}\partial_j}{m(\Lambda)}\big(\partial_t \psi^{\epsilon_k}(1-\chi)\big)\right) \phi\, \dd x \right|\\
  & = \left| \int_{\R^2} \partial_t \psi^{\epsilon_k} (1-\chi) \left( \frac{\D^{2-\beta}\partial_j}{m(\Lambda)}\big(\rho \phi\big)\right)\,\dd x \right| \\
  & \leq \|\partial_t \psi^{\epsilon_k}(t)\|_{H^{-5}(\R^2)} \left\|(1-\chi) \left(\frac{\D^{2-\beta}\nabla^\perp}{m(\Lambda)}\right)(\rho \phi)\right\|_{H^5(\R^2)}\\
  & \leq C \|\phi\|_{H^5(\R^2)},
\end{align*}
which ensures that, for $j=1,2$,
\begin{equation*}
  \partial_t g^{\epsilon_k}_j \in L^\infty([0,T], H^{-5}(\R^2)),\quad \textrm{uniformly in $\epsilon_k$}.
\end{equation*}
Hence, using to Ascoli's theorem  and the weak convergence of $\mathcal{H}_0\psi^{\epsilon_k}$ to $\mathcal{H}_0\psi$ (which is a consequence of the weak convergence of $\theta^{\epsilon_k}$),
we obtain that the sequence $g^{\epsilon_k}_j$ (up to a subsequence, still denoted $g^{\epsilon_k}_j$) satisfies
\begin{equation*}
  g_j^{\epsilon_k}\rightarrow g_j= \rho \Big(\frac{\D^{2-\beta}\partial_j}{m(\Lambda)}\big((\mathcal{H}_0\psi)(1-\chi)\big)\Big),\quad \textrm{in  $L^\infty([0,T],L^2(B_R))$},
\end{equation*}
and it implies that
\begin{equation}\label{J4conv}
  \lim_{k\rightarrow\infty} |J_4^{\epsilon_k}| \leq C \|\mathcal{H}_0\psi\|_{L^2([0,T],L^2(\R^2))} \lim_{k\rightarrow 0} \sup_{j=1,2}\big\|g_j^{\epsilon_k}-g_j\big\|_{L^2([0,T], L^2(B_R))}=0.
\end{equation}

\noindent The estimates of both $J_5^{\epsilon_k}$ and $J_6^{\epsilon_k}$ are the same as $J_3^{\epsilon_k}$ and $J_4^{\epsilon_k}$, and we have
\begin{equation}\label{J56conv}
  \lim_{k\rightarrow\infty } \left(|{J_5^{\epsilon_k}}|+ |{J_6^{\epsilon_k}}|\right) =0.
\end{equation}

Hence \eqref{eq:I1conv}, \eqref{eq:I2conv}, \eqref{eq:I3conv}, \eqref{eq:I4conv}, \eqref{J1conv}-\eqref{J56conv} and the decompositions \eqref{eq:decom0}-\eqref{eq:decom1}, \eqref{eq:decom5} allow us to get the desired convergence \eqref{Targ-conv}. \\

Therefore, we have proved \eqref{Targ-conv} for all $\alpha\in(0,1)$, $\beta\in(0,1]$, $\lambda\in [0,1)$, and this ends the proof of Theorem \ref{thm:GE}.
\qed

\section{Proof of Theorem \ref{thm:RegLog}}

Let $\epsilon>0$, and $\theta^\epsilon$ be a smooth solution to the following regularized equation 
\begin{equation}\label{ap-gSQG2}
(gSQG)_{\beta,\epsilon} : \begin{cases}
  \partial_t \theta^\epsilon + u^\epsilon\cdot\nabla \theta^\epsilon + \nu \Lambda^\beta\theta^{\epsilon} - \epsilon\Delta \theta= 0, &\quad (t,x)\in \mathbb{R}^+\times \mathbb{R}^2, \\
  u^{\epsilon} = \nabla^\perp \Lambda^{\beta-2} m(\Lambda)\theta^{\epsilon}, \\
  \theta^{\epsilon}(0,x)= \theta^{\epsilon}_0(x)*\phi_\epsilon, & \quad x\in\mathbb{R}^2,
\end{cases}
\end{equation}
with $\theta_0\in L^1\cap L^2(\R^2)$, $\phi_\epsilon(x) =\epsilon^{-2}\phi(\epsilon^{-1}x)$ and $\phi$ a mollifier.
Obviously, the global well-posedness of a solution $\theta^\epsilon$ of equation \eqref{ap-gSQG2} which is such that $\theta^\epsilon\in C([0,\infty);H^s(\R^2))\cap C^\infty([0,\infty)\times \R^2)$ with $s>2$  follows in a straightforward manner from Proposition \ref{prop:glob}.
Indeed, the additional dissipation term $\D^\beta \theta^\epsilon$ has a regularizing effect.
Since $\theta_0\in L^1\cap L^2(\R^2)$, we also have, uniformly in $\epsilon$, the following $L^p$-estimate, for all $T\geq0$
\begin{equation}\label{L1L2Est}
  \|\theta^\epsilon\|_{L^\infty([0,T];L^1\cap L^2(\R^2))}\leq \|\theta_0\|_{L^1\cap L^2(\R^2)},
\end{equation}
and the following energy estimate, uniformly in $\ep$,
\begin{equation}\label{EneEst}
  \|\theta^\epsilon(T)\|_{L^2(\R^2)}^2 + \int_0^T \|\theta^\epsilon(\tau)\|_{\dot H^{\beta/2}(\R^2)}^2\dd \tau \leq \|\theta_0\|_{L^2(\R^2)}^2,\quad \forall\; T\geq 0.
\end{equation}
Thanks to Theorem \ref{thm:GE}, we know that there exists a global weak solution $\theta$ to the $(gSQG)_\beta$ equation \eqref{eq:gSQG}.
Indeed, we have seen that we have enough compactness to pass to the limit as $\epsilon$ goes to 0, hence, {\it{a fortiori}} we can pass to the limit in \eqref{ap-gSQG2},
and both the $L^{1}\cap L^{2}$ maximum principle \eqref{L1L2Est} and the energy inequality \eqref{EneEst} obviously hold by replacing $\theta^\epsilon$ by its weak limit $\theta$. \\

In the sequel we divide the proof of Theorem \ref{thm:RegLog} into 4 subsections. Note that we first prove the point (2) of Theorem \ref{thm:RegLog}, 
that is, the eventual regularity result,
this will be done in the subsections \ref{subsec:EvReg}-\ref{subsec:T*}; and then the proof of point (1) of Theorem \ref{thm:RegLog} in the subsection \ref{subsec:grLog}. Before starting the proofs, we also give some regularity criteria for the weak solution to be smooth.

\subsection{Regularity criteria}\label{subsec:RC}

We first have the following regularity criteria for the global weak solution of the equation \eqref{eq:gSQG}.
\begin{lemma}\label{lem:RC}
  Let $\nu=1$, $\beta\in(0,1]$, $\theta \in L^\infty([0,\infty); L^1\cap L^2)\cap L^2([0,\infty); \dot H^{\beta/2})$ be the global weak solution to the $(gSQG)_\beta$ equation
\eqref{eq:gSQG} with $m$ satisfying (A1)-(A4).
Suppose that, for all times $T_1,T_2$ such that $0\leq T_1 < T_2 <\infty$, we have
\begin{equation}\label{eq:RCcd}
  \|\theta \|_{L^\infty([T_1,T_2]; \dot C^\sigma(\R^2))}<\infty,\quad \textrm{for some  }\sigma >\alpha,
\end{equation}
then we have
\begin{equation*}
  \theta \in C^\infty (\R^2\times(T_1,T_2]).
\end{equation*}
\end{lemma}

If $m(\Lambda)=\D^{1-\beta}$ ($\beta\in (0,1]$), the equation \eqref{eq:gSQG} becomes the classical $(SQG)_\beta$ equation \eqref{qg} with either critical or supercritical dissipation,
and Lemma \ref{lem:RC} is nothing but  the well-known regularity criteria that can be found for instance in \cite{CV,CW,Silv}. If $m(\Lambda)=\D^\alpha$ ($\alpha\in(0,1)$) and $\alpha+\beta>1$, then
\eqref{eq:gSQG} becomes the generalized $(SQG)$ equation which has been studied in \cite{MX11}, and Lemma \ref{lem:RC} is a direct consequence of \cite{MX11} (see Prop 5.1).
For the $(gSQG)_\beta$ equation \eqref{eq:gSQG} we are interested in,
due to the key estimate \eqref{eq:mBern}, we may for instance apply the Bony's paradifferential method as it was done in \cite{CW} and \cite{MX11} to conclude Lemma \ref{lem:RC}, and we omit the details here. \\

In order to prove the eventual regularity, we have to show that the regularized solution $\theta^\epsilon$ of equation \eqref{ap-gSQG2}
verifies the regularity criteria \eqref{eq:RCcd} uniformly in $\ep$ after a finite time.
To this end, we shall present an H\"older regularity criterion in terms of a suitable modulus of continuity,
which reduces the problem to showing that the solution $\theta^\epsilon$ uniformly in $\epsilon$ obeys this modulus of continuity at some time $t_0>0$.
More precisely, we have the following lemma.
\begin{lemma}\label{lem:RCpres}
Let $\nu=1$, $\beta\in(0,1]$, and assume that $m$ satisfies  (A1)(A2) (A4)(A5) with $\alpha\in (0,1)$, $\lambda\in [0,1)$.
For all $\sigma\in(\alpha,\min\{\alpha+\beta,1\})$, we define the following modulus of continuity
\begin{equation}\label{moc1}
  \omega(\xi)=
  \begin{cases}
    \kappa \frac{1}{m(\delta^{-1})} \delta^{-\sigma} \xi^\sigma,\quad & \mathrm{for   }\;\,0<\xi\leq \delta, \\
    \kappa \frac{1}{m(\delta^{-1})} + \gamma \int_\delta^\xi \frac{1}{\eta\, m(\eta^{-1})}\dd \eta,\quad & \mathrm{for   }\;\,\xi>\delta,
  \end{cases}
\end{equation}
with $0<\gamma<\kappa<1$ and $\delta>0$. Then, under some smallness condition on $\gamma$ and $\kappa$, namely that
\begin{equation}\label{kg-cd1}
\begin{split}
  & 0<\kappa < \min\left\{\frac{C_1(1-\sigma)(\alpha+\beta-\sigma)}{16 C_2},\frac{1}{2C_2\sigma}\right\}, \quad \\
  & 0<\gamma < \min\left\{ \frac{1}{2C_2},\sigma\kappa,\frac{(1-C_\alpha)^\alpha}{2} \kappa,\frac{C_1 \beta C_\alpha (1-C_\alpha)}{12 C_2}\right\},
\end{split}
\end{equation}
with $C_\alpha=\frac{2^\alpha -1}{\alpha}$, and $C_1$, $C_2$ the constants appearing in Lemmas \ref{lem-mocdiss} and \ref{lem-mocdrf}, the following assertion holds true:
\begin{equation}\label{eq:RCcd2}
  \textrm{if there exits $T_0 \in [0,\infty)$ such that $\theta^\epsilon(T_0)$ uniformly in $\epsilon$ obeys $\omega(\xi)$},
\end{equation}
then, the solution $\theta^\epsilon(t)$ preserves this modulus of continuity $\omega(\xi)$ for all time $t\in [T_0,\infty)$,
consequently, $\theta^\epsilon \in L^\infty([T_0,\infty);\dot C^\sigma(\R^2))$ uniformly in $\epsilon$.
\end{lemma}

\begin{remark}\label{rmk:RCgene}
If additionally we have
\begin{equation}\label{eq:RCcd3}
  \sup_{t\in[T_0,\infty)}\|\theta^\epsilon(t)\|_{L^\infty(\R^2)}< \frac{1}{2}\omega\left(\frac{1}{2b_1}\right)
\end{equation}
where $b_1$ is the constant appearing in the condition (A3), then by replacing $\omega(\xi)$ in \eqref{moc1} with the following modulus of continuity
\begin{equation}\label{moc1.2}
  \omega(\xi)=
  \begin{cases}
    \kappa \frac{1}{m(\delta^{-1})} \delta^{-\sigma} \xi^\sigma,\quad & \mathrm{for   }\;\,0<\xi\leq \delta, \\
    \kappa \frac{1}{m(\delta^{-1})} + \gamma \int_\delta^\xi \frac{1}{\eta\, m(\eta^{-1})}\dd \eta,\quad & \mathrm{for   }\;\,\delta<\xi\leq \frac{1}{2b_1}, \\
    \omega(\frac{1}{2b_1}),\quad & \mathrm{for   }\;\, \xi\geq \frac{1}{2b_1},
  \end{cases}
\end{equation}
the same conclusion as  Lemma \ref{lem:RCpres} also holds for the  function $m$ under the assumptions (A1)-(A4).
Indeed, one can similarly check that $\omega(\xi)$ defined by \eqref{moc1.2} is a modulus of continuity satisfying the condition (3) of Proposition \ref{prop-GC};
then due to \eqref{eq:RCcd3}, we only need to check the inequality \eqref{Targ1} for all $0<\xi\leq \frac{1}{2b_1}$, futhermore since we have \eqref{Ome-es1-2} and $\omega'(\eta)=0$ for all $\eta>\frac{1}{2b_1}$,
we observe that the justification is the same as that of Lemma \ref{lem:claim}.
Moreover, \eqref{efact1} can also be obtained in a similar manner, thus the preservation of the modulus of continuity $\omega(\xi)$ implies the desired $\sigma$-H\"older regularity, namely
\begin{equation}\label{CsgmEs}
  \sup_{t\in [T_0,\infty)}\|\theta^\epsilon(t)\|_{\dot C^\sigma(\R^2)} \leq \frac{\kappa }{m(\delta^{-1})} \delta^{-\sigma}.
\end{equation}
\end{remark}

\begin{proof}[{\bf{Proof of Lemma \ref{lem:RCpres}}}]

We first check that $\omega(\xi)$ defined in \eqref{moc1} is indeed a modulus of continuity.
Let us show that $\omega$ is an increasing and concave function for all $\xi>0$. To do so, we observe that for every $\xi \in (0,\delta),$ one has
\begin{equation*}
  \omega'(\xi)= \frac{\kappa \sigma }{m(\delta^{-1})} \delta^{-\sigma} \xi^{\sigma-1}>0,\quad\textrm{and}\quad \omega''(\xi) = - \frac{\kappa \sigma (1-\sigma)}{m(\delta^{-1})} \delta^{-\sigma} \xi^{\sigma-2}<0,
\end{equation*}
and for all $\xi>\delta$,
\begin{equation}\label{omeg-fact1}
  \omega'(\xi) = \frac{\gamma}{\xi m(\xi^{-1})}>0,\quad \textrm{and}\quad \omega''(\xi) = - \frac{\gamma }{\xi^2 m(\xi^{-1})}\left(1-\frac{\xi^{-1} m'(\xi^{-1})}{m(\xi^{-1})}\right)\leq - \frac{\gamma (1-\alpha)}{\xi^2 m(\xi^{-1})}<0,
\end{equation}
and for $\xi =\delta$, then
\begin{equation*}
  \omega'(\delta-)= \frac{\kappa \sigma}{ \delta m(\delta^{-1})},\quad \textrm{and} \quad \omega'(\delta+)=\frac{\gamma}{\delta m(\delta^{-1})},
\end{equation*}
thus if $\gamma< \sigma\kappa$, we infer that $\omega(\xi)$ is increasing and concave for all $\xi>0$.

Besides, it is easy to see that $\omega(0+)=0$ and that
$\omega'(0+)=\displaystyle\lim_{\xi\rightarrow 0+}\kappa\sigma\frac{1}{m(\delta^{-1})}\delta^{-\sigma} \xi^{\sigma-1}= +\infty,$
hence, $\omega$ satisfies the condition (3) in Proposition \ref{prop-GC}. \\

Then, according to Proposition \ref{prop-GC}, it suffices to prove that for all $t>T_0$ and all $\xi\in\{\xi>0:\omega(\xi) \leq 2 \|\theta^\epsilon(t)\|_{L^\infty}\}$,
\begin{equation}\label{Targ1}
  \Omega(\xi,t)\omega'(\xi) + D(\xi,t) <0,
\end{equation}
where $\Omega(\xi,t)$, $D(\xi,t)$ are respectively defined by \eqref{GC-Om} and \eqref{GC-D}
under the scenario \eqref{scena} with $\omega(\cdot)$ in place of $\omega(\cdot,t)$. Using Lemma \ref{lem-mocdiss} and Lemma \ref{lem-mocdrf}, we obtain
\begin{equation*}
\begin{split}
  D(\xi,t) \leq
  \,C_1 \int_0^{\frac{\xi}{2}}  \frac{\omega(\xi+2\eta)+\omega (\xi-2\eta) -2\omega (\xi)}{\eta^{1+\beta}}{d} \eta
   + C_1 \int_{\frac{\xi}{2}}^\infty  \frac{\omega (2\eta+\xi)-\omega (2\eta-\xi) -2\omega (\xi)}{\eta^{1+\beta}}{d}\eta,
\end{split}
\end{equation*}
and
\begin{equation}\label{Omega2}
  \Omega(\xi,t) \leq -C_2\xi m(\xi^{-1}) D(\xi,t) + C_2 \xi\int_\xi^\infty \frac{\omega(\eta)m(\eta^{-1})}{\eta^{1+\beta}}\mathrm{d}\eta +C_2 \xi^{1-\beta} m(\xi^{-1})\omega(\xi).
\end{equation}

In order to show \eqref{Targ1}, we shall divide the study into two cases. \\

\noindent \textbf{Case 1: $0<\xi \leq \delta$.} \\

We first focus on the contribution coming from the velocity, and we see that,
\begin{align*}
 \mathcal{Q}(\xi) \equiv \int_\xi^\infty\frac{\omega(\eta)m(\eta^{-1})}{\eta^{1+\beta}}\dd \eta & =  -\int_\xi^\infty {\omega(\eta)m(\eta^{-1})}\partial_{\eta}\left\{\frac{\eta^{-\beta}}{\beta}\right\}\dd \eta \
\end{align*}
then we integrate by parts, and by using the following inequalities
\begin{equation}\label{eq:dom}
  (\omega(\eta)m(\eta^{-1}))'=\omega'(\eta)m(\eta^{-1})-\frac{1}{\eta^{2}}\omega(\eta)m'(\eta^{-1})\leq\omega'(\eta)m(\eta^{-1}),
\end{equation}
and
\begin{align}\label{eq:fact}
  \lim_{\eta\rightarrow \infty} \frac{\omega(\eta) m(\eta^{-1})}{\eta^\beta} & = \lim_{\eta\rightarrow \infty} \frac{\kappa m(\delta^{-1}) m(\eta^{-1})}{\eta^\beta}
  + \lim_{\eta\rightarrow \infty} \frac{\gamma m(\eta^{-1})}{\eta^\beta} \int_\delta^\eta \frac{1}{\tau m(\tau^{-1})}\dd \tau \nonumber \\
  & \leq \lim_{\eta\rightarrow \infty} \frac{\kappa m(\delta^{-1}) m(1)}{\eta^\beta}
  +  \lim_{\eta\rightarrow \infty}\frac{\gamma}{\eta^\beta} \int_\delta^\eta \frac{1}{\tau}\dd \tau =0,
\end{align}
we infer that
\begin{align*}
  \mathcal{Q}(\xi) \leq \frac{\omega(\xi) m(\xi^{-1})}{\beta \xi^\beta} + \frac{1}{\beta} \int_\xi^\infty \frac{\omega'(\eta) m(\eta^{-1})}{\eta^\beta} \dd \eta.
\end{align*}
Then, since $\omega'(\eta)= \kappa \sigma  \frac{1}{m(\delta^{-1})} \delta^{-\sigma} \eta^{\sigma-1}$ for $\eta\leq \delta$ and $\omega'(\eta)= \gamma \frac{1}{\eta\,m(\eta^{-1})}$ for $\eta>\delta$, we find that
\begin{align*}
 \mathcal{Q}(\xi) &\leq \frac{\omega(\xi) m(\xi^{-1})}{\beta\xi^\beta } + \frac{\kappa \sigma}{\beta \delta^\sigma m(\delta^{-1})} \int_\xi^\delta \frac{m(\eta^{-1})}{\eta^{1+\beta-\sigma}}\dd \eta +
  \gamma \int_\delta^\infty \frac{1}{\beta \eta^{1+\beta}} \dd \eta, \\
  &\leq \frac{\omega(\xi) m(\xi^{-1})}{\beta\xi^\beta } + \frac{\kappa \sigma \delta^{\alpha-\sigma}}{\beta } \int_\xi^\delta \frac{1}{\eta^{1+\beta+\alpha-\sigma}}\dd \eta +
  \frac{\gamma}{\beta^2 \delta^\beta} \\
  &\leq \frac{\omega(\xi) m(\xi^{-1})}{\beta\xi^\beta } + \frac{\kappa \sigma \delta^{\alpha-\sigma}}{\beta} \frac{1}{\alpha +\beta-\sigma}  \xi^{-(\alpha+\beta-\sigma)} +
  \frac{\gamma}{\beta^2 \delta^\beta}\\
  &\leq \frac{\omega(\xi) m(\xi^{-1})}{\beta\xi^\beta } + \frac{2\sigma\kappa }{\beta(\alpha+\beta-\sigma)}\frac{1}{\xi^\beta},
\end{align*}
where in the last line we used $\gamma<\sigma\kappa$.
By using the nondecreasing property of $r\mapsto r^\sigma m(r^{-1})$ ($\sigma>\alpha$), we get
\begin{equation*}
\begin{split}
  \xi^{1-\beta} m(\xi^{-1})\omega(\xi) \omega'(\xi) =\sigma\kappa^2 \frac{ \xi^\sigma m(\xi^{-1})\xi^{\sigma-\beta}}{ (\delta^\sigma m(\delta^{-1}))^2}
  \leq\sigma \kappa^2 \frac{\xi^{\sigma-\beta}}{\delta^\sigma m(\delta^{-1})},
\end{split}
\end{equation*}
and
\begin{equation*}
  \xi m(\xi^{-1}) \omega'(\xi) = \kappa \sigma\frac{\xi^\sigma m(\xi^{-1})}{\delta^\sigma m(\delta^{-1}) }\leq \kappa \sigma.
\end{equation*}
Then, \eqref{Omega2} and the above estimates allow us to conclude that, for all $\kappa \leq \frac{1}{2 C_2\sigma}$, we have
\begin{align*}
  \Omega(\xi,t) \omega'(\xi) & \leq  -C_2 \sigma\kappa D(\xi,t) + C_2\sigma \kappa \frac{1}{\delta^\sigma m(\delta^{-1})}
  \xi^{\sigma-\beta}\left(\frac{2 \kappa}{\beta} +\frac{2\sigma \kappa }{\beta(\alpha+\beta-\sigma)} \right) \\
  & \leq -\frac{1}{2} D(\xi,t) +  \frac{4 C_2\sigma \kappa^2}{\beta(\alpha+\beta-\sigma)}\frac{1}{\delta^\sigma m(\delta^{-1})}  \xi^{\sigma-\beta} .
\end{align*}

Then, we consider the contribution from the diffusion term. By using that $$\omega''(\xi)= -\kappa \sigma(1-\sigma)\frac{1}{\delta^\sigma m(\delta^{-1})} \xi^{\sigma-2},$$ we obtain
\begin{align*}
  D(\xi,t) & \leq C_1 \int_0^{\xi/2}\frac{\omega(\xi+2\eta)+\omega(\xi-2\eta)-2\omega(\xi)}{\eta^{1+\beta}}d\eta\\
  &\leq C_1 \int_0^{\xi/2}\frac{\omega''(\xi)2\eta^{2}}{\eta^{1+\beta}}d\eta \\
  & \leq -2 C_1  \kappa \sigma(1-\sigma)\frac{\xi^{\sigma-2}}{\delta^\sigma m(\delta^{-1})} \int_0^{\xi/2} \eta^{1-\beta} \dd \eta \\
  &\leq -2C_1 \sigma(1-\sigma)\kappa \frac{\xi^{\sigma-\beta}}{\delta^\sigma m(\delta^{-1})(2-\beta)2^{2-\beta}} \\
  & \leq -\frac{C_1 \sigma(1-\sigma)\kappa}{2} \frac{1}{\delta^\sigma m(\delta^{-1})}\xi^{\sigma-\beta}.
\end{align*}
Hence, we finally get
\begin{equation*}
  \Omega(\xi,t)\omega'(\xi) + D(\xi,t) \leq \sigma \kappa \frac{1}{\delta^\sigma m(\delta^{-1})}\xi^{\sigma-\beta}
  \left(\frac{4 C_2\kappa }{\beta(\alpha+\beta-\sigma)} - \frac{C_1 (1-\sigma)}{4}  \right)<0,
\end{equation*}
and the last inequality is ensured by choosing $\kappa < \min\left\{\frac{C_1\beta(1-\sigma)(\alpha+\beta-\sigma)}{16 C_2},\frac{1}{2C_2\sigma}\right\}$. \\

\noindent \textbf{Case 2: $\xi >\delta$.} \\

Once again, we first focus on the contribution coming from the velocity \eqref{Omega2}. We integrate by parts, and following what we did in Case 1, we find
\begin{equation}\label{est-omeg}
\begin{split}
  \int_\xi^\infty\frac{\omega(\eta)m(\eta^{-1})}{\eta^{1+\beta}} \dd \eta & \leq \frac{\omega(\xi)m(\xi^{-1})}{\beta\xi^\beta }
  +\frac{1}{\beta}\int_\xi^\infty\frac{\omega'(\eta)m(\eta^{-1})}{\eta^\beta } \dd \eta\\
  & \leq \frac{\omega(\xi)m(\xi^{-1})}{\beta\xi^\beta }+\frac{\gamma}{\beta}\int_\xi^\infty \frac{1}{\eta^{1+\beta}} \dd \eta
  \leq \frac{\omega(\xi)m(\xi^{-1})}{\beta\xi^\beta }+\frac{\gamma}{\beta^2\xi^\beta }.
\end{split}
\end{equation}
From \eqref{Omega2} and \eqref{est-omeg}, we see that
\begin{align*}
  \omega'(\xi)\Omega(\xi,t) &\leq-C_2\xi \omega'(\xi) m(\xi^{-1}) D(\xi,t) + C_2\omega'(\xi) \xi^{1-\beta}
  {\omega(\xi)m(\xi^{-1})}\left(\frac{1}{\beta}+\frac{\gamma}{\omega(\xi)m(\xi^{-1})\beta^2} +1\right),
\end{align*}
but we know that $\omega'(\xi) = \frac{\gamma}{\xi m(\xi^{-1})}$, thus $\omega'(\xi)\xi m(\xi^{-1})=\gamma$, and the last inequality becomes
\begin{eqnarray*}
\omega'(\xi)\Omega(\xi,t) &\leq&-C_2 \gamma D(\xi,t) + C_2 \gamma \xi^{-\beta}
{\omega(\xi)}\left(\frac{1}{\beta}+\frac{\gamma}{\omega(\xi)m(\xi^{-1})\beta^2} +1\right).
\end{eqnarray*}
Then, we may assume that $\gamma\leq \frac{1}{2C_2}$ so that
\begin{eqnarray}\label{Vitesse}
  \omega'(\xi)\Omega(\xi,t) &\leq&-\frac{1}{2}  D(\xi,t) + C_2 \gamma \xi^{-\beta}
  {\omega(\xi)}\left(\frac{1}{\beta}+\frac{\gamma}{\omega(\xi)m(\xi^{-1})\beta^2} +1\right).
\end{eqnarray}

On the other hand,  we have
\begin{eqnarray*}
  \omega(2\xi)&=&\omega(\xi)+\int_\xi^{2\xi}\frac{\gamma}{\eta m(\eta^{-1})}\dd \eta= \omega(\xi)+\int_\xi^{2\xi} \frac{1}{\eta^{1-\alpha}} \frac{\gamma}{\eta^{\alpha} m(\eta^{-1})} \dd \eta,
\end{eqnarray*}
since $\xi \mapsto \xi^{-\alpha} m(\xi^{-1})$ is an non-decreasing function (see \eqref{m-fac3}), we get
\begin{eqnarray}\label{10sip}
  \omega(2\xi)&\leq&\omega(\xi)+\frac{1}{\xi^\alpha m(\xi^{-1})}\int_\xi^{2\xi}\frac{\gamma}{\eta^{1-\alpha}}\dd\eta = \omega(\xi)+\frac{2^{\alpha}-1}{\alpha} \frac{\gamma}{m(\xi^{-1})}.
\end{eqnarray}
In the following, we need $\frac{2^{\alpha}-1}{\alpha} \frac{\gamma}{m(\xi^{-1})} <\omega(\xi)$, and to do so, we shall use the following lemma.
\begin{lemma}\label{lem:claim}
For all $\overline{C}>1$, let $\gamma$ be sufficiently small, namely
\begin{equation}\label{gam-cd}
  \gamma< \left(\frac{\overline{C}-1}{\overline{C}}\right)^\alpha\kappa,
\end{equation}
then we have
\begin{equation}\label{claim1}
  \frac{\gamma}{m(\xi^{-1})}\leq \overline{C}\omega(\xi),\qquad \forall \xi>\delta.
\end{equation}
\end{lemma}
\begin{proof}[{\bf Proof of Lemma \ref{lem:claim}}]
Indeed, for $\xi\in(\delta,\frac{\overline{C}}{\overline{C}-1}\delta)$, we infer, by \eqref{m-fac3}, that $m(\delta^{-1})\leq (\frac{\overline{C}}{\overline{C}-1})^\alpha m(\xi^{-1})$,
and by assuming \eqref{gam-cd}, we get
$\omega(\xi)\geq \omega(\delta)=\frac{\kappa}{m(\delta^{-1})}\geq\frac{(\overline{C}-1)^\alpha \kappa}{\overline{C}^\alpha m(\xi^{-1})}\geq \frac{ \gamma}{m(\xi^{-1})}$;
while for $\xi\in[\frac{\overline{C}}{\overline{C}-1}\delta,\infty)$, we find that
\begin{align*}
  \omega(\xi)\geq \gamma\int_\delta^{\xi}\frac{1}{\eta m(\eta^{-1})}\dd \eta
  &\geq\frac{\gamma}{\xi^\alpha m(\xi^{-1})}\int_\delta^{\xi}\frac{1}{\eta^{1-\alpha}}\dd \eta\\
  &\geq\frac{\gamma}{\xi^\alpha m(\xi^{-1})}\int_{(1-1/\overline{C})\xi}^{\xi}\frac{1}{\eta^{1-\alpha}}\dd \eta \geq \overline{C}^{-1}\frac{\gamma}{m(\xi^{-1})},
\end{align*}
which gives the claim \eqref{claim1} and ends the proof of Lemma \ref{lem:claim}.
\end{proof}
Then, we come back to \eqref{10sip}, and thanks to the previous lemma we obtain
\begin{equation*}
  \omega(2\xi)\leq\omega(\xi)+\frac{2^\alpha -1}{\alpha}\overline{C}\omega(\xi)
  =(1+C_\alpha \overline{C})\omega(\xi),
\end{equation*}
where we set $C_\alpha=\frac{2^\alpha -1}{\alpha}\in(0,1)$. Thus, by choosing $\overline{C}=\frac{1+C_\alpha }{2C_\alpha }>1$, we find
\begin{equation}\label{key-est1}
  \omega(2\xi)\leq(3/2+C_\alpha /2)\omega(\xi).
\end{equation}
Since $\omega(2\eta+\xi)-\omega(2\eta-\xi)\leq\omega(2\xi)$, we infer that
\begin{equation*}
\begin{split}
  D(\xi,t)\leq - C_1 \int_{\xi/2}^\infty\frac{(1/2-C_\alpha /2)\omega(\xi)}{\eta^{1+\beta}}d\eta \leq -\frac{C_1(1-C_\alpha )}{2\beta} \frac{\omega(\xi)}{\xi^\beta }.
\end{split}
\end{equation*}

We shall, once again, use Lemma \ref{lem:claim} to get a nice estimate of the contribution coming from the nonlinear part. Using \eqref{claim1}, one finds
\begin{equation}
\begin{split}
  \Omega(\xi,t) \omega'(\xi)   & \leq -\frac{1}{2}D(\xi,t) + C_2\gamma\frac{\omega(\xi)}{\xi^\beta }\left(\frac{1}{\beta}+\frac{\overline{C}}{\beta^2}+1 \right),
\end{split}
\end{equation}
then, using that $\overline{C}=\frac{1+C_\alpha }{2C_\alpha }\leq \frac{1}{C_\alpha}$, we obtain
\begin{equation*}
  \Omega(\xi,t)\omega'(\xi) + D(\xi,t) \leq  \frac{\omega(\xi)}{\xi^\beta }
  \left(C_2\gamma \frac{3}{\beta^2 C_\alpha} - \frac{C_1 (1-C_\alpha)}{4\beta}  \right)<0,
\end{equation*}
where the last inequality holds by assuming $\gamma< \frac{C_1 \beta C_\alpha (1-C_\alpha)}{12 C_2}$. \\

Therefore, for~$\sigma\in(\alpha,\min\{\alpha+\beta,1\})$ and $\delta>0$, the solution $\theta^\epsilon(t)$ of the evolution equation \eqref{ap-gSQG2} on the time interval $[T_0,\infty)$
preserves the modulus of continuity $\omega(\xi)$ as long as $0<\gamma<\kappa <1$ are fixed constants satisfying
\begin{equation}\label{kg-cd0}
\begin{split}
  & 0<\kappa < \min\left\{\frac{C_1(1-\sigma)(\alpha+\beta-\sigma)}{16 C_2},\frac{1}{2C_2\sigma}\right\}, \quad \\
  & 0<\gamma < \min\left\{ \frac{1}{2C_2}, \frac{(1-C_\alpha)^\alpha}{2} \kappa,\frac{C_1 \beta C_\alpha (1-C_\alpha)}{12 C_2}\right\},
\end{split}
\end{equation}
where $C_\alpha =\frac{2^\alpha -1}{\alpha}\in (0,1)$ for $\alpha\in (0,1)$. \\

To end the proof of Lemma \ref{lem:RCpres}, we need to show the uniform $\dot C^\sigma$-regularity of $\theta^\epsilon$, and to this purpose, it suffices to prove the following claim
\begin{equation}\label{efact1}
  \textrm{the mapping $\xi\mapsto \frac{\omega(\xi)}{\xi^\sigma}$ for every $\xi >0$ is nonincreasing.}
\end{equation}
Indeed, if $\xi\in(0,\delta]$, then \eqref{efact1} follows easily from the definition of $\omega$ ; while if $\xi\in (\delta,\infty)$, then we have $(\frac{\omega(\xi)}{\xi^\sigma })'=\frac{\xi\omega'(\xi)-\sigma\omega(\xi)}{\xi^{\sigma+1}}$, and  thanks to \eqref{omeg-fact1} and the condition $\sigma>\alpha$, we infer that
\begin{equation*}
  (\xi\omega'(\xi)-\sigma\omega(\xi))'=\omega'(\xi)+\xi\omega''(\xi)-\sigma\omega'(\xi)\leq \frac{(1-\sigma-(1-\alpha))\gamma }{\xi m(\xi^{-1})}<0.
\end{equation*}
Besides, if $\gamma<\sigma\kappa$, then we have
\begin{equation}
\delta\omega'(\delta+)-\sigma\omega(\delta)=\frac{\gamma}{m(\delta^{-1})}-\frac{\sigma\kappa}{m(\delta^{-1})}<0,
\end{equation}
thus we find that $\frac{d}{d\xi}(\frac{\omega(\xi)}{\xi^\sigma })<0$, which implies that \eqref{efact1} holds for all $\xi\in (\delta,\infty)$. Note that we used a new condition on $\gamma$ which is $\gamma<\sigma\kappa$, therefore, the conditions on the constants $\kappa$  and $\gamma$ of the preserved modulus of continuity \eqref{kg-cd0} becomes \eqref{kg-cd1} as stated in Lemma \ref{lem:RCpres}. \\

Then, since  the modulus of continuity $\omega(\xi)$ associated to  $\theta^\epsilon(t)$, which is  defined by \eqref{kg-cd1}, is preserved in the time interval $[T_0,\infty)$, we obtain
by using  \eqref{efact1}  that
\begin{equation}\label{Cbetaest}
  \sup_{t\in [T_0,\infty)}\|\theta^\epsilon(t)\|_{\dot C^\sigma}= \sup_{t\in[T_0,\infty)} \sup_{x\neq y\in\R^2}\frac{|\theta^\epsilon(x,t)-\theta^\epsilon(y,t)|}{|x-y|^\sigma}
  \leq \sup_{x\neq y\in\R^2}\frac{\omega(|x-y|)}{|x-y|^\sigma} \leq \kappa \frac{1}{m(\delta^{-1})} \delta^{-\sigma},
\end{equation}
which corresponds to the $\sigma$-H\"older regularity of $\theta^\epsilon(t)$ uniformly in $\epsilon$.
This allows us to conclude Lemma \ref{lem:RCpres}.
\end{proof}

\subsection{{\bf{Eventual regularity of the weak solutions of the $(gSQG)_{\beta}$ equation \eqref{eq:gSQG}.}}}\label{subsec:EvReg}

As far as the regularized equation \eqref{ap-gSQG2} is concerned, we have already seen that, uniformly with respect to $\ep$, both the  $L^p$-estimate \eqref{L1L2Est} and the  energy estimate \eqref{EneEst} are verified.
Then, using \eqref{EneEst} along with the divergence-free property of $u^\epsilon$, one can show the following  $L^\infty$ estimate, that is, for all $t'>0$,
\begin{equation}\label{LinfEst}
  \|\theta^\epsilon\|_{L^\infty([t',+\infty)\times\R^2)} \leq \frac{C_\beta}{t'^{1/\beta}} \|\theta_0\|_{L^2(\R^2)},
\end{equation}
where $C_\beta>0$ is a constant independent of $\epsilon$. The argument used to prove \eqref{LinfEst} is the De Giorgi iteration method, which does, luckily,
not depend on the nature of the velocity $u^\epsilon$ since only the uniform energy estimate
and the divergence-free property of $u^\epsilon$ are used. Hence, the proof of \eqref{LinfEst} is quite similar to those in \cite{CV}, \cite{CW2} for instance, and we therefore omit the details. \\

Now according to Lemma \ref{lem:RC} and Lemma \ref{lem:RCpres}, in order to prove the eventual regularity of the global weak solution $\theta$,
it suffices to show that the smooth solution $\theta^\epsilon$ of the approximate equation \eqref{ap-gSQG2} is, uniformly in $\epsilon$,  $\sigma$-H\"olderian for some time $t'+T_*$.
To this end, we shall prove that the solution of the approximate equation \eqref{ap-gSQG2} obeys (uniformly in $\epsilon$) the following moduli of continuity $\omega(\xi,\xi_0)$: \\

\noindent for all $\xi_0>\delta$,
\begin{equation}\label{moc2}
  \omega(\xi,\xi_0)=
  \begin{cases}
    \frac{(1-\sigma)\kappa }{m(\delta^{-1})} + \gamma \int_\delta^{\xi_0}\frac{1}{\eta m(\eta^{-1})} \dd \eta
    -\frac{\gamma }{\xi_0 m(\xi_0^{-1})}(\xi_0-\delta) + \frac{\sigma \kappa }{\delta m(\delta^{-1})} \xi, \quad &\textrm{if   } 0<\xi\leq \delta, \\
    \frac{\kappa}{m(\delta^{-1})} + \gamma \int_\delta^{\xi_0} \frac{1}{\eta m(\eta^{-1})}\dd \eta
    - \frac{\gamma }{m(\xi_0^{-1})} + \frac{ \gamma }{\xi_0 m(\xi_0^{-1})} \xi ,\quad &\textrm{if   } \delta<\xi\leq \xi_0, \\
    \frac{\kappa}{m(\delta^{-1})} + \gamma \int_\delta^\xi \frac{1}{\eta m(\eta^{-1})}\dd \eta,\quad &\textrm{if   } \xi>\xi_0 ,
  \end{cases}
\end{equation}
and for all $\xi_0\leq\delta$,
\begin{equation}\label{moc3}
  \omega(\xi,\xi_0)=
  \begin{cases}
    \frac{(1-\sigma)\kappa}{m(\delta^{-1})} \delta^{-\sigma} \xi_0^\sigma  + \frac{\sigma \kappa }{m(\delta^{-1})}\delta^{-\sigma} \xi_0^{\sigma-1} \xi, \quad &\textrm{if   } 0<\xi\leq \xi_0, \\
    \frac{\kappa}{m(\delta^{-1})} \delta^{-\sigma} \xi^\sigma,\quad &\textrm{if   } \xi_0<\xi\leq \delta, \\
    \frac{\kappa }{m(\delta^{-1})} + \gamma \int_\delta^\xi \frac{1}{\eta m(\eta^{-1})}\dd \eta,\quad &\textrm{if   } \xi>\delta ,
  \end{cases}
\end{equation}
where $\sigma\in(\alpha, \min\{\alpha+\beta,1\})$, $\kappa,\gamma,\delta>0$ are constants that will be chosen later
and $\xi_0=\xi_0(t)$ is a time-dependent function given by \eqref{xi0}.
Motivated by \cite{Kis}, the basic idea to construct such moduli $\omega(\xi,\xi_0)$ consists in taking a tangent line at the point $\xi=\xi_0$ to $\omega(\xi)$ (given by \eqref{moc1})
and replacing $\omega(\xi)$ with this tangent line in the interval $0<\xi \leq\xi_0$.
However, since the one-sided derivatives of $\omega(\xi)$ at the point $\xi=\delta$ do not coincide, we make a crucial modification in the case $\xi_0 >\delta$,
that is, the tangent line mentioned above in  the interval $\delta\leq\xi\leq \xi_0$ is once again taken,
but for $0<\xi\leq \delta$ it is replaced by a straight line crossing $\omega(\delta+,\xi_0)$ with the larger slope $\omega'(\delta-)=\frac{\sigma \kappa}{ \delta m(\delta^{-1})}$.

 It is not difficult to check that $\omega(\xi,\xi_0)$  is an increasing concave function in $\xi$, for every $\xi>0$ and $\xi_0>0$. Furthermore, we may easily see that, for all $\xi_0>0$, we have $\omega(0+,\xi_0)>0$, and thus the condition (3) in Proposition \ref{prop-GC} holds.
For $\xi_0 = A_0 > \delta$ where $A_0$ is a  constant that will be chosen later, one observes that
\begin{align}\label{Omeg0A0}
  \omega(0+,A_0)&=(1-\sigma)\kappa\frac{1}{m(\delta^{-1})}+\gamma\int_\delta^{A_0}\frac{1}{\eta m(\eta^{-1})}\dd\eta-\frac{\gamma}{A_0m(A_0^{-1})}(A_0-\delta) \nonumber\\
  & \geq(1-\sigma)\kappa\frac{1}{m(\delta^{-1})}+\frac{\gamma}{  A_0^\alpha m(A_0^{-1})} \frac{A_0^\alpha -\delta^\alpha}{\alpha}-\frac{\gamma}{m(A_0^{-1})} \\
  & \geq \big((1-\sigma)\kappa-\gamma\big)\frac{1}{m(\delta^{-1})} + \frac{(1-\alpha)\gamma}{\alpha} \frac{ A_0^\alpha -\delta^\alpha}{A_0^\alpha m(A_0^{-1})}, \nonumber
\end{align}
and by assuming that $\gamma <(1-\sigma)\kappa$ and using \eqref{LinfEst}, one gets that the solution $\theta^\epsilon(t')$ obeys the modulus of continuity $\omega(\xi,A_0)$ provided that
\begin{equation}\label{eq:Adel-cd}
  \frac{(1-\alpha)\gamma}{\alpha} \frac{A_0^\alpha -\delta^\alpha}{A_0^\alpha m(A_0^{-1})}\geq \frac{2 C_\beta}{t'^{1/\beta}} \|\theta_0\|_{L^2(\R^2)}.
\end{equation}

The following key lemma shows that the breakdown of such moduli of continuity after $t'$ cannot happen.
\begin{lemma}\label{lem-mocev}
  Suppose that $\theta^\epsilon(t')$ obeys the modulus of continuity $\omega(\xi,A_0)$ given by \eqref{moc2}.
For $\rho >0$, let $\xi_0=\xi_0(t)$ be a time-dependent function defined by
\begin{equation}\label{xi0}
  \frac{d}{dt}\xi_0= - \rho \, \xi_0^{1-\beta} , \quad \xi_0(0)= A_0.
\end{equation}
Then, for some sufficiently small positive constants $\delta$, $\kappa$, $\gamma$ and $\rho$ (verifying \eqref{eq:Adel-cd} and \eqref{rkg-cdsum} below), the solution $\theta^\epsilon(x,t+t')$ to the regularized generalized $(gSQG)_{\beta,\epsilon}$ equation \eqref{ap-gSQG2}, uniformly in $\epsilon$, preserves the modulus of continuity $\omega(\xi,\xi_0(t))$ for every $\xi_0(t)> 0$.
\end{lemma}

Using this lemma, we shall show the desired eventual H\"older regularity of global weak solution to \eqref{eq:gSQG}.
Indeed, using \eqref{xi0}, we have that $\xi_0(t)= (A_0^\beta - \rho \beta t)^{\frac{1}{\beta}}$, thus at the time $t_1$ defined by
\begin{equation}\label{t1}
  t_1= A_0^\beta/(\beta\rho)
\end{equation}
we have $\xi_0(t_1)= 0$ and consequently $\theta^\epsilon(x,t_1+t')$ obeys the modulus of continuity $\omega(\xi,0+) = \omega(\xi)$, with $\omega(\xi)$ being the modulus of continuity defined by \eqref{moc1}.
Hence, the property \eqref{eq:RCcd2} is satisfied at the time $T_0=t_1+t'$ and thus, Lemma \ref{lem:RCpres} implies that the modulus of continuity $\omega(\xi)$ given by \eqref{moc1}
is (uniformly in $\epsilon$) preserved by  the evolution of $\theta^\epsilon(t)$ in the time interval $t\geq t_1+t'$ (note that $\kappa,\gamma$ here satisfies \eqref{kg-cd1}). Therefore, from \eqref{Cbetaest} we get
\begin{equation}\label{the-Ces}
  \sup_{t\in [t_1+t',+\infty)}\|\theta^\epsilon(t)\|_{\dot C^\sigma(\R^2)} \leq \kappa \frac{1}{m(\delta^{-1})} \delta^{-\sigma},
\end{equation}
where $\delta>0$ is a fixed constant satisfying \eqref{eq:Adel-cd}, which combined with Lemma \ref{lem:RC} concludes the eventual regularity result.\\

Now it remains to prove Lemma \ref{lem-mocev}.

\begin{proof}[{\bf{{Proof of Lemma \ref{lem-mocev}}}.}]

According to Proposition \ref{prop-GC} and since we have (via \eqref{eq:Adel-cd})
$$\omega(A_0,\xi_0)> \omega(0+,A_0)> 2\|\theta^\epsilon\|_{L^\infty([t',+\infty)\times\R^2)},$$
we observe that it suffices to prove that for all $t>0$ so that $\xi_0(t)>0$ and $0<\xi \leq A_0$,
\begin{equation}\label{Targ3}
  - \partial_{\xi_0}\omega(\xi,\xi_0) \xi'_0(t)+ \Omega(\xi,t)\partial_\xi\omega(\xi,\xi_0) + D(\xi,t) + \epsilon \partial_{\xi\xi}\omega(\xi,\xi_0)<0,
\end{equation}
where $\omega(\xi,\xi_0)$ is given by \eqref{moc2}-\eqref{moc3}, and
\begin{equation}\label{Dxit3}
\begin{split}
  D(\xi,t)\leq &
  \,C_1 \int_0^{\frac{\xi}{2}}  \frac{\omega(\xi+2\eta,\xi_0)+\omega (\xi-2\eta,\xi_0) -2\omega (\xi,\xi_0)}{\eta^{1+\beta}}{d} \eta \\
  & + C_1 \int_{\frac{\xi}{2}}^\infty  \frac{\omega (2\eta+\xi,\xi_0)-\omega (2\eta-\xi,\xi_0) -2\omega (\xi,\xi_0)}{\eta^{1+\beta}}{d}\eta,
\end{split}
\end{equation}
and,
\begin{equation}\label{Omega3}
  \Omega(\xi,t) \leq -C_2\xi m(\xi^{-1}) D(\xi,t) + C_2 \xi\int_\xi^\infty \frac{\omega(\eta,\xi_0)m(\eta^{-1})}{\eta^{1+\beta}}{d}\eta +C_2 \xi^{1-\beta} m(\xi^{-1})\omega(\xi,\xi_0) .
\end{equation}
Note that in \eqref{Targ3}, if $\partial_{\xi_0}\omega(\xi,\xi_0)$ or $\partial_\xi\omega(\xi,\xi_0)$ does not exist, the larger value of the one-sided derivative will be taken.

Depending on the values of $\xi_0$ and $\xi$, we shall divide the study into 5 cases to prove \eqref{Targ3}. \\

\noindent \textbf{Case 1: $\xi_0 > \delta$, $0 <\xi \leq \delta$.} \\

In that case, \eqref{moc2} implies that
$$\omega(\xi,\xi_0)=(1-\sigma)\kappa \frac{1}{m(\delta^{-1})} + \gamma \int_\delta^{\xi_0}\frac{1}{\eta m(\eta^{-1})} \dd \eta
    -\gamma \frac{1}{\xi_0 m(\xi_0^{-1})}(\xi_0-\delta) + \frac{\sigma \kappa }{\delta m(\delta^{-1})} \xi,$$
hence,
\begin{equation}\label{omeges1}
  \partial_{\xi_0}\omega(\xi,\xi_0)  \leq \gamma \frac{1}{\xi_0 m(\xi_0^{-1})}, \quad\textrm{and}\quad
  \partial_\xi \omega(\xi,\xi_0)  = \, \sigma\kappa  \frac{1}{\delta m(\delta^{-1})} ,
  \end{equation}
and
\begin{eqnarray}\label{Mxi0}
  \omega(\xi,\xi_0)\geq \omega(0+,\xi_0)&=& (1-\sigma)\kappa \frac{1}{m(\delta^{-1})} + \gamma \int_\delta^{\xi_0}\frac{1}{\eta m(\eta^{-1})} \dd \eta
    -\gamma \frac{1}{\xi_0 m(\xi_0^{-1})}(\xi_0-\delta) \nonumber \\
   &\geq&  (1-\sigma)\kappa \frac{1}{m(\delta^{-1})} +\gamma \frac{1}{\xi_0^\alpha m(\xi_0^{-1})} \int_\delta^{\xi_0} \frac{1}{\eta^{1-\alpha}}\dd\eta - \gamma \frac{1}{\xi_0m(\xi_0^{-1})} (\xi_0-\delta) \nonumber \\
  &=&  (1-\sigma)\frac{\kappa}{ m(\delta^{-1})} +\frac{\gamma}{\alpha} \frac{1}{\xi_0^\alpha m(\xi_0^{-1})} \left( \xi_0^\alpha -\delta^\alpha\right) -\gamma  \frac{1}{\xi_0m(\xi_0^{-1})}
  \left(\xi_0 -\delta \right),\nonumber \\
  & \equiv& M_{\xi_0,\delta}.
\end{eqnarray}
We also have
\begin{equation}\label{omegdel}
  \omega(\xi,\xi_0)-\omega(0+,\xi_0)\leq \omega(\delta,\xi_0)-\omega(0+,\xi_0)=\sigma\kappa\frac{1}{ m(\delta^{-1})}.
\end{equation}
Therefore, by \eqref{xi0} and \eqref{omeges1} we obtain
\begin{equation}\label{xi0-est0}
  -\partial_{\xi_0}\omega(\xi,\xi_0) \xi'_0(t)
  \leq \rho  \gamma \frac{1}{ m(\xi_0^{-1}) \xi_0^\beta}.
\end{equation}
Using \eqref{moc2}, \eqref{eq:dom}-\eqref{eq:fact} and \eqref{m-fac3},
by integrating by parts, one gets that, for all $\gamma\leq \sigma \kappa$,
\begin{align}\label{omegint}
  &\int_\xi^\infty\frac{\omega(\eta,\xi_0)m(\eta^{-1})}{\eta^{1+\beta}}\dd\eta
  \leq \frac{\omega(\xi,\xi_0)m(\xi^{-1})}{\beta\xi^\beta } + \int_\xi^\infty\frac{\partial_\eta \omega(\eta,\xi_0) m(\eta^{-1})}{ \beta \eta^\beta }\dd\eta \nonumber\\
  =&\frac{\omega(\xi,\xi_0)m(\xi^{-1})}{\beta\xi^\beta } + \frac{\sigma \kappa }{\beta \delta m(\delta^{-1})} \int_\xi^\delta \frac{m(\eta^{-1})}{\eta^\beta }\dd\eta
  + \frac{\gamma}{\beta \xi_0 m(\xi_0^{-1})}\int_\delta^{\xi_0} \frac{ m(\eta^{-1})}{\eta^\beta} \dd\eta
  + \frac{\gamma}{\beta} \int_{\xi_0}^\infty \frac{1}{\eta^{1+\beta}} \dd\eta \nonumber \\
  \leq & \frac{\omega(\xi,\xi_0) m(\xi^{-1})}{\beta\xi^\beta } + \frac{\sigma\kappa }{\beta\delta^{1-\alpha}}\int_\xi^\delta\frac{1}{\eta^{\alpha+\beta}}\dd\eta
  +\frac{\gamma }{\beta\xi_0^{1-\alpha} }\int_\delta^{\xi_0}\frac{1}{\eta^{\alpha+\beta}}\dd\eta + \frac{\gamma}{\beta^2 \xi_0^\beta} \\
  = &\frac{\omega(\xi,\xi_0) m(\xi^{-1})}{\beta\xi^\beta } + \nonumber
  \begin{cases}
    \frac{\sigma\kappa }{\beta\delta^{1-\alpha}} \frac{\delta^{1-\alpha-\beta}-\xi^{1-\alpha-\beta}}{1-\alpha-\beta}
    + \frac{\gamma }{\beta\xi_0^{1-\alpha} } \frac{\xi_0^{1-\alpha-\beta} - \delta^{1-\alpha-\beta}}{1-\alpha-\beta}  + \frac{\gamma}{\beta^2 \xi_0^\beta},
    &\quad \textrm{if  }\alpha+\beta<1, \\
     \frac{\sigma\kappa }{\beta\delta^{1-\alpha}} \log \frac{\delta}{\xi}
    + \frac{\gamma }{\beta\xi_0^{1-\alpha} } \log \frac{\xi_0}{\delta}  + \frac{\gamma}{\beta^2 \xi_0^\beta}, &\quad \textrm{if  }\alpha+\beta = 1, \\
    \frac{\sigma\kappa }{\beta\delta^{1-\alpha}} \frac{\xi^{1-\alpha-\beta}-\delta^{1-\alpha-\beta}}{\alpha+\beta-1}
    + \frac{\gamma }{\beta\xi_0^{1-\alpha} } \frac{\delta^{1-\alpha-\beta}-\xi_0^{1-\alpha-\beta} }{\alpha+\beta-1}  + \frac{\gamma}{\beta^2 \xi_0^\beta},
    &\quad \textrm{if  }\alpha+\beta>1,
  \end{cases} \\
  \leq &\frac{\omega(\xi,\xi_0) m(\xi^{-1})}{\beta\xi^\beta } + \nonumber
  \begin{cases}
    \frac{3\sigma\kappa }{\beta} \frac{\delta^{-\beta}}{1-\alpha-\beta},
    &\quad \textrm{if  }\alpha+\beta<1, \\
    \frac{\sigma\kappa }{\beta\delta^{1-\alpha}} \log \frac{\delta}{\xi}
    + \frac{\gamma }{\beta\xi_0^{1-\alpha} } \log \frac{\xi_0}{\delta}  + \frac{\gamma}{\beta^2 \xi_0^\beta}, &\quad \textrm{if  }\alpha+\beta = 1, \\
    \frac{\sigma\kappa }{\beta\delta^{1-\alpha}} \frac{\xi^{1-\alpha-\beta}}{\alpha+\beta-1}
    + \frac{\gamma}{\beta^2 \xi_0^\beta},
    &\quad \textrm{if  }\alpha+\beta>1.
  \end{cases}
\end{align}
Then, applying \eqref{Omega3}, \eqref{omeges1}, \eqref{omegdel} and \eqref{omegint}, we obtain, by choosing $\gamma,\kappa$ such that $\gamma\leq \sigma\kappa$ and $\kappa<\frac{1}{2 C_2\sigma} $, the following inequality
\begin{align}\label{omeg-est0}
  \Omega(\xi,t)\partial_\xi\omega(\xi,\xi_0)
  \leq  \nonumber&
  -C_2 \sigma \kappa D(\xi,t) + 2C_2 \sigma \kappa\frac{ \omega (\xi,\xi_0)\xi^{1-\beta} m(\xi^{-1})}{\beta \delta m(\delta^{-1})}  + \nonumber\\
  & \, +
  \begin{cases}
    \frac{3C_2\sigma^2\kappa^2 }{\beta (1-\alpha-\beta)} \frac{\xi \delta^{-\beta}}{\delta m(\delta^{-1})},
    &\quad \textrm{if  }\alpha+\beta<1, \nonumber\\
    C_2\left(\frac{\sigma\kappa }{\beta\delta^{1-\alpha}} \xi \log \frac{\delta}{\xi}
    + \frac{\gamma }{\beta\xi_0^{1-\alpha} } \xi\log \frac{\xi_0}{\delta}  + \frac{\gamma\xi}{\beta^2 \xi_0^\beta}\right)\frac{\sigma \kappa}{\delta m(\delta^{-1})},
    &\quad \textrm{if  }\alpha+\beta = 1, \nonumber\\
    C_2\left(\frac{\sigma\kappa }{\beta\delta^{1-\alpha}} \frac{\xi^{2-\alpha-\beta}}{\alpha+\beta-1}
    + \frac{\gamma \xi}{\beta^2 \xi_0^\beta} \right) \frac{\sigma \kappa}{\delta m(\delta^{-1})},
    &\quad \textrm{if  }\alpha+\beta>1,\nonumber
  \end{cases} \\
  \leq  &
  - \frac{1}{2} D(\xi,t) + \frac{2C_2 \sigma \kappa }{\beta} \frac{\omega (0+,\xi_0) }{\xi^\beta} + \frac{2 C_2 \sigma^2 \kappa^2}{\beta m(\delta^{-1})} \frac{\xi^{1-\beta-\alpha}}{\delta^{1-\alpha}} + \nonumber\\
  & \, +
  \begin{cases}
    \frac{3}{ 1-\alpha-\beta} \frac{C_2 \sigma^2\kappa^2}{\beta m(\delta^{-1})}\delta^{-\beta},
    &\quad \textrm{if  }\alpha+\beta<1, \nonumber\\
    \left(\frac{2C_0 }{\beta }
    + \frac{C_\alpha' \gamma  }{\sigma \kappa }   \right)
    \frac{C_2 \sigma^2 \kappa^2}{\beta m(\delta^{-1})}\delta^{-(1-\alpha)},
    &\quad \textrm{if  }\alpha+\beta = 1,\nonumber \\
    \left(\frac{1 }{\alpha+\beta-1} \frac{\xi^{2-\alpha}}{\delta^{2-\alpha}}
    + \frac{\gamma}{\beta \sigma\kappa } \frac{\xi^{1+\beta}}{\delta\xi_0^\beta} \right) \frac{C_2\sigma^2 \kappa^2}{\beta m(\delta^{-1})}\xi^{-\beta},
    &\quad \textrm{if  }\alpha+\beta>1 ,\nonumber
  \end{cases}  \\
  \leq & - \frac{1}{2} D(\xi,t) + \frac{2C_2 \sigma \kappa }{\beta} \frac{\omega (0+,\xi_0) }{\xi^\beta}
  + \frac{ C_2 \sigma^2 \kappa^2 B_{\alpha,\beta}}{\beta m(\delta^{-1})} \frac{1}{ \xi^\beta},
\end{align}
where in the case $\alpha+\beta=1$ we also used that $\frac{\xi}{\delta}\left( \log\frac{\delta}{\xi}\right)\leq C_0$ and
$\frac{\delta^{1-\alpha}}{\xi_0^{1-\alpha}}\log \frac{\xi_0}{\delta}\leq C_\alpha'$, and
\begin{equation}\label{Balpbet}
B_{\alpha,\beta}\equiv
\begin{cases}
  \frac{5}{1-\alpha-\beta}, \quad & \textrm{if  }\alpha+\beta<1, \\
  \frac{4C_0}{\beta} + C_\alpha',\quad & \textrm{if  } \alpha+\beta=1, \\
  \frac{4}{\alpha+\beta-1},\quad & \textrm{if  }\alpha+\beta>1.
\end{cases}
\end{equation}

For the contribution from the diffusion term, using the concavity of the function $\omega(\eta,\xi_0)-\omega(0+,\xi_0)$, we obtain that
\begin{equation}\label{D2xit00}
\begin{split}
  D(\xi,t)
  \leq -2 C_1 \omega(0+,\xi_0)\int_{\frac{\xi}{2}}^\infty \frac{1}{\eta^{1+\beta}}\dd \eta
  \leq -\frac{2C_1}{\beta}  \frac{\omega(0+,\xi_0)}{\xi^\beta},
\end{split}
\end{equation}
and thanks to \eqref{Mxi0}, we get
\begin{equation}\label{D2xit}
\begin{split}
  D(\xi,t) \leq  -\frac{2C_1}{\beta} \frac{M_{\xi_0,\delta}}{\xi^\beta}.
\end{split}
\end{equation}
If $\xi_0 \geq N \delta$ with $N\in\N$ which is chosen such that
$\frac{1}{\alpha}\left( \xi_0^\alpha  -\delta^\alpha \right)\geq \frac{1-(1/N)^\alpha }{\alpha}\xi_0^\alpha \geq \frac{2}{1+\alpha}\xi_0^\alpha$
which means that $1-(1/N)^\alpha \geq \frac{2\alpha}{1+\alpha}$, that is,
$N\geq \left(\frac{1+\alpha}{1-\alpha}\right)^{\frac{1}{\alpha}}$, thus we may choose
\begin{equation}\label{Ncond}
  N \equiv \left[\left(\frac{1+\alpha}{1-\alpha}\right)^{\frac{1}{\alpha}}\right]+1.
\end{equation}
Thus for such an $N$, we have in the case $\xi_0\geq N\delta$ that
\begin{equation}\label{Mxi0-est1}
\begin{split}
  M_{\xi_0,\delta}  \geq (1-\sigma)\frac{\kappa}{ m(\delta^{-1})}  + \left(\frac{2}{1+\alpha} -1 \right) \gamma\, \frac{1}{m(\xi_0^{-1})}
  \geq (1-\sigma)\frac{\kappa}{ m(\delta^{-1})}  + \frac{1-\alpha}{1+\alpha}\, \frac{\gamma}{m(\xi_0^{-1})} .
\end{split}
\end{equation}
Hence, inequality \eqref{D2xit} becomes
\begin{equation}\label{D2xit-es}
\begin{split}
  D(\xi,t)\leq - \frac{2C_1(1-\sigma)\kappa}{\beta} \frac{1}{ m(\delta^{-1}) \xi^\beta}  - \frac{C_1(1-\alpha)\gamma}{\beta}\, \frac{1}{m(\xi_0^{-1}) \xi^\beta}.
\end{split}
\end{equation}
Therefore, for $\xi_0\geq N \delta$, we choose $\kappa$ such that $\kappa\leq \frac{C_1}{4C_2\sigma}$ so that inequality \eqref{D2xit00} gives
\begin{equation}\label{diff-cont}
  \frac{2C_2 \sigma \kappa }{\beta} \frac{\omega (0+,\xi_0) }{\xi^\beta} \leq \frac{C_1}{2\beta}  \frac{\omega(0+,\xi_0)}{\xi^\beta}\leq -\frac{1}{4}D(\xi,t),
\end{equation}
and by collecting \eqref{xi0-est0}, \eqref{omeg-est0} and \eqref{D2xit-es}, we deduce that
\begin{equation*}
\begin{split}
  \textrm{L.H.S. of \eqref{Targ3}}
  \leq  \,
  \left(  \frac{ C_2 \sigma \kappa B_{\alpha,\beta}}{\beta} - \frac{C_1(1-\sigma)}{2\beta} \right) \frac{\kappa }{ m(\delta^{-1}) \xi^\beta}
  + \left( \rho  - \frac{C_1 (1-\alpha)}{4\beta}\right) \frac{\gamma }{m(\xi_0^{-1}) \xi^\beta} < 0,
\end{split}
\end{equation*}
where the last inequality is verified as long as $\rho$, $\kappa$, $\gamma$ satisfy
\begin{equation}\label{rkg-cd1}
  \rho < \frac{C_1 (1-\alpha)}{4\beta}, \quad
  \kappa< \min \set{ \frac{C_1}{4C_2\sigma}, \frac{C_1 (1-\sigma) }{2C_2 \sigma^2  B_{\alpha,\beta}}}, \quad
  \gamma < \sigma\kappa.
\end{equation}
Conversely, if $\xi_0\leq N\delta$ with $N$ satisfying \eqref{Ncond}, using the following fact that
\begin{equation}\label{mfact2}
  \frac{1}{m(\xi_0^{-1})}\leq \frac{1}{m\left((N\delta)^{-1}\right)} \leq \frac{N^\alpha}{m(\delta^{-1})} \leq \frac{4}{1-\alpha} \frac{1}{ m(\delta^{-1})},
\end{equation}
and using \eqref{diff-cont} once again, we obtain that the positive contribution is, thanks to \eqref{xi0-est0} and \eqref{omeg-est0}, given by
\begin{equation*}
\begin{split}
  -\partial_{\xi_0}\omega(\xi,\xi_0)\xi'_0(t) + \Omega(\xi,t)\partial_\xi \omega(\xi,\xi_0)
  \leq\, - \frac{3}{4} D(\xi,t) + \frac{\kappa}{m(\delta^{-1}) \xi^\beta}
  \left( \frac{4\rho}{1-\alpha} \frac{\gamma}{\kappa} + \frac{ C_2 \sigma^2 \kappa B_{\alpha,\beta}}{\beta}  \right).
\end{split}
\end{equation*}
For the negative contribution coming from the dissipation term, via \eqref{Mxi0}, \eqref{D2xit} and \eqref{mfact2}, we easily get  that for every $\gamma\leq \frac{(1-\sigma)(1-\alpha)}{8}\kappa$,
\begin{align}\label{D2xit-es2}
  D(\xi,t) & \leq -\frac{2C_1}{\beta} \frac{1}{\xi^\beta} \left( (1-\sigma)\kappa\, \frac{1}{m(\delta^{-1})} -\gamma  \frac{1}{m(\xi_0^{-1})} \right) \nonumber \\
  & \leq  -\frac{2C_1}{\beta} \left((1-\sigma)\kappa- \frac{4\gamma}{1-\alpha} \right)\, \frac{1}{m(\delta^{-1}) \xi^\beta}
  \leq -\frac{C_1 \left( 1-\sigma\right) \kappa}{\beta}  \frac{1}{m(\delta^{-1}) \xi^\beta} .
\end{align}
Hence for $\xi_0\leq N\delta$, we  finally have that
\begin{equation*}
\begin{split}
  \textrm{L.H.S. of \eqref{Targ3}}\,\leq &\,\frac{\kappa}{m(\delta^{-1}) \xi^\beta}\,
  \left( \frac{\rho(1-\sigma)}{2} + \frac{ C_2 \sigma^2 \kappa B_{\alpha,\beta}}{\beta} - \frac{C_1(1-\sigma)}{4\beta}\right) < 0,
\end{split}
\end{equation*}
where the last inequality is verified if
\begin{equation}\label{rkg-cd2}
\begin{split}
  \rho < \frac{C_1 }{4\beta},\quad  \kappa< \min\set{\frac{C_1 (1-\sigma)}{8C_2\sigma^2 B_{\alpha,\beta}},\frac{1}{4C_2\sigma}}, \quad
  \gamma <\min\set{\frac{(1-\sigma)(1-\alpha)}{8}\kappa,\sigma \kappa}.
\end{split}
\end{equation}

\textbf{Case 2: $\xi_0>\delta$, $\delta<\xi\leq \xi_0$.} \\

In that case, we have
$$\omega(\xi,\xi_0) = \kappa \frac{1}{m(\delta^{-1})} + \gamma \int_\delta^{\xi_0} \frac{1}{\eta m(\eta^{-1})}\dd \eta
-\gamma \frac{1}{m(\xi_0^{-1})} + \gamma \frac{1}{\xi_0 m(\xi_0^{-1})} \xi,$$
 hence, we infer that, $ $
\begin{equation}\label{est1}
\begin{split}
  \partial_\xi \omega(\xi,\xi_0) = \gamma \frac{1}{\xi_0 m(\xi_0^{-1})},\quad\textrm{and} \quad \partial_{\xi_0}\omega(\xi,\xi_0)
  \leq \frac{\gamma}{\xi_0m(\xi_0^{-1})}+\frac{\alpha\gamma\xi}{\xi_0^{2}m(\xi_0^{-1})} \leq \frac{2\gamma}{\xi_0 m(\xi_0^{-1})},
\end{split}
\end{equation}
and (recalling $M_{\xi_0,\delta}$ defined in \eqref{Mxi0})
\begin{equation}
\begin{split}
  \omega(\xi,\xi_0)\geq \omega(\delta,\xi_0)
  & \geq \kappa \frac{1}{m(\delta^{-1})} + \gamma \int_\delta^{\xi_0} \frac{1}{\eta m(\eta^{-1})}\dd \eta -\gamma \frac{1}{\xi_0 m(\xi_0^{-1})} (\xi_0-\delta) \\
  & \geq \frac{\kappa}{ m(\delta^{-1})} +\frac{\gamma}{\alpha} \frac{1}{\xi_0^\alpha m(\xi_0^{-1})} \left( \xi_0^\alpha -\delta^\alpha\right) -\gamma  \frac{1}{\xi_0m(\xi_0^{-1})}
  \left(\xi_0 -\delta \right) \\
  & = M_{\xi_0,\delta} + \frac{\sigma \kappa}{ m(\delta^{-1})},
\end{split}
\end{equation}
and we also have that
\begin{equation}\label{omeg-xi0}
  \omega(\xi,\xi_0)-\omega(0+,\xi_0)\leq \omega(\xi_0,\xi_0)-\omega(0+,\xi_0)\leq \frac{\gamma}{ m(\xi_0^{-1})} + \frac{ \sigma \kappa}{ m(\delta^{-1})} \leq \frac{2\kappa}{m(\xi_0^{-1})}.
\end{equation}
Therefore, thanks to \eqref{xi0} and \eqref{est1}, we get
\begin{equation}\label{xi0-est2}
\begin{split}
  -\partial_{\xi_0}\omega(\xi,\xi_0)  \xi'_0(t)  \leq \frac{2\rho\gamma}{\xi_0m(\xi_0^{-1})} \xi_0^{1-\beta} \leq \frac{2 \rho\gamma}{\xi^\beta m(\xi_0^{-1})}.
\end{split}
\end{equation}
We moreover obtain that
\begin{align}\label{omeg-es2}
  \int_\xi^\infty \frac{\omega(\eta,\xi_0) m(\eta^{-1})}{\eta^{1+\beta}} \dd\eta & \leq \frac{\omega(\xi,\xi_0) m(\xi^{-1})}{\beta\xi^\beta }
  + \int_\xi^\infty\frac{\partial_\eta\omega(\eta,\xi_0) m(\eta^{-1})}{\beta\eta^{\beta}}d\eta \nonumber\\
  & = \frac{\omega(\xi,\xi_0) m(\xi^{-1})}{\beta \xi^\beta} + \frac{\gamma}{\beta \xi_0 m(\xi_0^{-1})} \int_\xi^{\xi_0} \frac{ m(\eta^{-1})}{\eta^\beta} \dd\eta
  + \int_{\xi_0}^\infty \frac{\gamma}{\beta \eta^{1+\beta}} \dd\eta\nonumber \\
  & \leq \frac{\omega(\xi,\xi_0) m(\xi^{-1})}{\beta \xi^\beta} + \frac{\gamma}{\beta\xi_0^{1-\alpha}} \int_\xi^{\xi_0} \frac{1}{\eta^{\alpha+\beta}}\dd \eta + \frac{\gamma}{\beta^2 \xi_0^\beta} \\
  & \leq \frac{\omega(\xi,\xi_0) m(\xi^{-1})}{\beta \xi^\beta} +
  \begin{cases}
  \frac{\gamma}{\beta(1-\alpha-\beta)\xi_0^\beta} + \frac{\gamma}{\beta^2\xi_0^\beta}, &\quad \textrm{if  }  \alpha+\beta<1 \nonumber\\
  \frac{\gamma}{\beta\xi_0^\beta}\log\frac{\xi_0}{\xi} + \frac{\gamma}{\beta^2\xi_0^\beta}, &\quad \textrm{if  } \alpha+\beta=1 \nonumber\\
  \frac{\gamma}{\beta(\beta+\alpha-1)\xi^\beta} + \frac{\gamma}{\beta^2\xi_0^\beta}, &\quad \textrm{if  } \alpha+\beta>1 \nonumber.
  \end{cases}
\end{align}
Thus, by combining \eqref{Omega3} with \eqref{omeg-es2}, and using $\partial_\xi \omega(\xi,\xi_0) = \gamma \frac{1}{\xi_0 m(\xi_0^{-1})}$ and $\xi m(\xi^{-1})\leq \xi_0 m(\xi_0^{-1}),$
we obtain the following control
\begin{align}
  \Omega(\xi,t) \partial_\xi\omega(\xi,\xi_0) \leq &
  -C_2 \gamma D(\xi,t) + \frac{2C_2\gamma \omega(\xi,\xi_0)}{\beta \xi^\beta}  +
  \begin{cases}
  \frac{C_2\gamma^2}{\beta(1-\alpha-\beta)\xi_0^\beta m(\xi_0^{-1})} + \frac{C_2\gamma^2}{\beta^2\xi_0^\beta m(\xi_0^{-1})}, &\textrm{if  }\alpha+\beta<1 \nonumber\\
  \frac{C_2\gamma^2}{\beta\xi_0^\beta m(\xi_0^{-1})} \frac{\xi}{\xi_0} \log\frac{\xi_0}{\xi} + \frac{C_2\gamma^2}{\beta^2\xi_0^\beta m(\xi_0^{-1})}, &\textrm{if  }\alpha+\beta=1 \nonumber\\
  \frac{C_2\gamma^2}{\beta(\alpha+\beta-1)\xi^\beta m(\xi_0^{-1})} + \frac{C_2\gamma^2}{\beta^2\xi_0^\beta m(\xi_0^{-1})}, &\textrm{if  } \alpha+\beta>1,
  \end{cases}
\end{align}
then, using  \eqref{omeg-xi0} and the fact that $\frac{\xi}{\xi_0}\log\frac{\xi_0}{\xi}\leq C_0$ we observe that if we set
\begin{equation}\label{overB}
\overline{B}_{\alpha,\beta} =
\begin{cases}
  \frac{3}{1-\alpha-\beta} +\frac{3}{\beta},&\quad\textrm{if  }\alpha+\beta<1,\\
  C_0 +\frac{5}{\beta},&\quad\textrm{if  } \alpha+\beta=1,\\
  \frac{6}{\alpha+\beta-1},&\quad \textrm{if  }\alpha+\beta>1,
\end{cases}
\end{equation}
then we get that if $\gamma$ is chosen such that  $\gamma<\frac{1}{4C_2}$, we find
\begin{align}\label{Ome-es3}
  & \Omega(\xi,t) \partial_\xi\omega(\xi,\xi_0) \nonumber\\
   \leq & -\frac{1}{4}D(\xi,t) +\frac{2C_2 \gamma \omega(0+,\xi_0)}{\beta \xi^\beta} + \frac{4C_2\kappa \gamma }{\beta \xi^\beta m(\xi_0^{-1})} +
  \begin{cases}
  \left(\frac{1}{1-\alpha-\beta} +\frac{1}{\beta}\right)\frac{C_2\gamma^2}{\beta\xi_0^\beta m(\xi_0^{-1})}, & \;\textrm{if  } \alpha+\beta<1 \nonumber\\
  \left(C_0 +\frac{1}{\beta}\right) \frac{C_2 \gamma^2}{\beta \xi_0^\beta m(\xi_0^{-1})}, & \; \textrm{if  }\alpha+\beta=1  \nonumber\\
  \left(\frac{1}{\alpha+\beta-1}+ \frac{1}{\beta}\right)\frac{  C_2\gamma^2}{\beta\xi^\beta m(\xi_0^{-1})}, &\; \textrm{if  }\alpha+\beta>1  \nonumber
  \end{cases}\\
  \leq & -\frac{1}{4}D(\xi,t) +\frac{2C_2 \gamma\omega(0+,\xi_0)}{\beta \xi^\beta} + \frac{C_2 \overline{B}_{\alpha,\beta} \kappa\gamma}{\beta} \frac{1}{\xi^\beta m(\xi_0^{-1})}.
\end{align}

For the contribution from the diffusion term, we still have \eqref{D2xit00} and \eqref{D2xit}.
Now, as we did in Case 1 above, we split the study into $\xi_0\geq N\delta$ and $\xi_0\leq N\delta$ where $N$ has been defined in \eqref{Ncond}. If $\xi_0\geq N\delta$, then,
by using \eqref{D2xit-es} and by choosing $\gamma$ such that $\gamma<\frac{C_1}{4C_2}$, we find that
\begin{equation*}
\begin{split}
  \textrm{L.H.S. of \eqref{Targ3}}\;\leq\;
  &  \frac{\gamma}{\beta \xi^\beta m(\xi_0^{-1})} \left( 2\rho \beta
  +  C_2 \overline{B}_{\alpha,\beta} \kappa  - \frac{C_1(1-\alpha )}{2} \right)  < 0,
\end{split}
\end{equation*}
where the last inequality holds true as long as
\begin{equation}\label{rkg-cd3}
  \rho<\frac{C_1(1-\alpha )}{8\beta},\quad \kappa < \frac{C_1(1-\alpha)}{4C_2 \overline{B}_{\alpha,\beta}},\quad
  \gamma<\min\set{\frac{1}{4C_2},\frac{C_1}{4C_2}, \sigma \kappa}.
\end{equation}
Otherwise, if $\xi_0\leq N\delta$, and we have, using \eqref{mfact2} and choosing $\gamma$ such that $\gamma<\frac{C_1}{4C_2}$,
 that the positive contribution treated by \eqref{xi0-est2} and \eqref{Ome-es3} can further be bounded as 
\begin{equation*}
\begin{split}
  -\partial_{\xi_0}\omega(\xi,\xi_0)\xi'_0(t) + \Omega(\xi,t)\partial_\xi \omega(\xi,\xi_0)
  \leq - \frac{1}{2} D(\xi,t) + \frac{4\gamma}{(1-\alpha)\beta} \frac{1}{\xi^\beta m(\delta^{-1})}\,
  \left( 2 \beta \rho + C_2 \overline{B}_{\alpha,\beta} \,\kappa \right).
\end{split}
\end{equation*}
For the negative contribution from the diffusion term, by arguing as \eqref{D2xit-es2} we obtain that for $\gamma$ choosen such that $\gamma \leq \frac{(1-\sigma)(1-\alpha)}{8}\kappa$, one has
\begin{equation*}
\begin{split}
  D(\xi,t) \leq -\frac{2C_1}{\beta} \left((1-\sigma)\kappa- \frac{4\gamma}{1-\alpha} \right)\, \frac{1}{m(\delta^{-1}) \xi^\beta}
  \leq  - 2C_1 \frac{4\,\gamma}{(1-\alpha)\beta}  \frac{1}{\xi^\beta m(\delta^{-1})} .
\end{split}
\end{equation*}
Hence for $\xi_0\leq N\delta$ with $N$ given by \eqref{Ncond}, we deduce
\begin{equation*}
\begin{split}
  \textrm{L.H.S. of \eqref{Targ3}}\,\leq \, \frac{4\,\gamma}{(1-\alpha)\beta}  \frac{1}{\xi^\beta m(\delta^{-1})}
  \left(  2 \rho \beta  + C_2 \overline{B}_{\alpha,\beta}\, \kappa - C_1 \right)  < 0,
\end{split}
\end{equation*}
where the last inequality holds true if
\begin{equation}\label{rkg-cd4}
\begin{split}
  \rho < \frac{C_1 }{4\beta},\quad \kappa < \frac{C_1}{2C_2\overline{B}_{\alpha,\beta}}, \quad\gamma \leq \min\set{\frac{(1-\sigma)(1-\alpha)}{8}\kappa,\frac{1}{4C_2},\frac{C_1}{4C_2}}.
\end{split}
\end{equation}

\textbf{Case 3: $\xi_0>\delta$, $\xi_0<\xi\leq A_0$.} \\

In this case, $$\omega(\xi,\xi_0) =\kappa \frac{1}{m(\delta^{-1})} + \gamma \int_\delta^\xi \frac{1}{\eta m(\eta^{-1})}\dd \eta.$$
We see that $\partial_{\xi_0}\omega(\xi,\xi_0)=0$, $\partial_\xi\omega(\xi,\xi_0)=\gamma \frac{1}{ \xi m(\xi^{-1})}$, and following the same arguments as we did to obtain \eqref{est-omeg}, we find, using \eqref{claim1} and $\omega(\xi,\xi_0)= \omega(\xi)$ (with $\omega(\xi)$ defined in \eqref{moc1}) that
\begin{equation*}
\begin{split}
  \int_\xi^\infty \frac{\omega(\eta,\xi_0) m(\eta^{-1})}{\eta^{1+\beta}} \dd\eta  \leq \frac{\omega(\xi,\xi_0) m(\xi^{-1})}{\beta\xi^\beta }+\frac{\gamma}{\beta^2 \xi^\beta }
  \leq \frac{2 \overline{C} \omega(\xi,\xi_0) m(\xi^{-1})}{\beta^2 \xi^\beta}.
\end{split}
\end{equation*}
where $\overline{C}=\frac{1+C_\alpha }{2C_\alpha }$ and $C_\alpha=\frac{2^\alpha -1}{\alpha}\in(0,1)$.
Thus, thanks to \eqref{Omega3}, if one chooses $\gamma$ so that $\gamma < \frac{1}{2C_2}$, then, one gets
\begin{equation*}
  \Omega(\xi,t)\partial_\xi\omega(\xi,\xi_0) \leq - \frac{1}{2} D(\xi,t) + \frac{ 3C_2 \overline{C }\gamma}{\beta^2} \frac{\omega(\xi,\xi_0)}{\xi^\beta}.
\end{equation*}

\noindent As far as the contribution from the diffusion term is concerned, since $$\omega(2\eta+\xi,\xi_0)-\omega(2\eta-\xi,\xi_0)\leq \omega(2\xi,\xi_0)< 2\omega(\xi,\xi_0),$$
then, by following the same idea as \eqref{key-est1}, we obtain
\begin{equation*}
  D(x,t) \leq -C_1 \frac{(1-C_\alpha) \omega(\xi,\xi_0)}{2}  \int_{\frac{\xi}{2}}^\infty\frac{1}{\eta^{1+\beta}} \dd\eta
  \leq -\frac{C_1(1-C_\alpha )}{2\beta} \frac{\omega(\xi,\xi_0)}{\xi^\beta }.
\end{equation*}
Hence,
\begin{equation*}
  \Omega(\xi,t)\partial_\xi\omega(\xi,\xi_0) + D(\xi,t) \leq \frac{\omega(\xi,\xi_0)}{\beta\xi^\beta }
  \left(C_2\gamma \frac{3\overline{C}}{\beta} - \frac{C_1 (1-C_\alpha)}{4}  \right)<0,
\end{equation*}
provided that the constants $\kappa, \gamma$ are chosen so that
\begin{equation}\label{rkg-cd5}
\begin{split}
   0<\kappa<1,\quad 0<\gamma < \min\left\{ \frac{1}{2C_2},\frac{(1-C_\alpha)^\alpha}{2} \kappa,\frac{C_1 \beta C_\alpha (1-C_\alpha)}{12 C_2}\right\}.
\end{split}
\end{equation}

\textbf{Case 4: $0<\xi_0\leq \delta$, $0<\xi\leq\xi_0$.} \\

In this case,
$$\omega(\xi,\xi_0)=(1-\sigma)\kappa \frac{1}{m(\delta^{-1})} \delta^{-\sigma} \xi_0^\sigma  + \sigma \kappa\frac{1}{m(\delta^{-1})}\delta^{-\sigma} \xi_0^{\sigma-1} \xi,$$
and thus
\begin{equation*}
\begin{split}
  \partial_{\xi_0}\omega(\xi,\xi_0) =\sigma(1-\sigma)\kappa \frac{1}{m(\delta^{-1})} \delta^{-\sigma} \xi_0^{\sigma-1} \left(1 -\frac{\xi}{\xi_0} \right), \quad \textrm{and}\quad
  \partial_\xi\omega(\xi,\xi_0)  = \sigma \kappa \frac{1}{m(\delta^{-1})} \delta^{-\sigma} \xi_0^{\sigma-1},
\end{split}
\end{equation*}
and
\begin{equation}\label{eq:fact2}
\begin{split}
  \omega(\xi,\xi_0)  \geq \omega(0+,\xi_0)\geq (1-\sigma)\kappa \frac{1}{m(\delta^{-1})} \delta^{-\sigma} \xi_0^\sigma , \;\; \textrm{and}\;\;
  \omega(\xi,\xi_0)  \leq \omega(\delta,\xi_0)\leq \kappa \frac{1}{m(\delta^{-1})} \delta^{-\sigma} \xi_0^\sigma .
\end{split}
\end{equation}
We have
\begin{equation}\label{omegxi0-3}
\begin{split}
  -\partial_{\xi_0}\omega(\xi,\xi_0)\dot\xi_0(t) \leq \rho \sigma(1-\sigma)\kappa \frac{1}{m(\delta^{-1})} \delta^{-\sigma} \xi_0^{\sigma-\beta}(1 -\frac{\xi}{\xi_0})
  \leq \frac{C_2 \rho\sigma \omega(\xi,\xi_0)}{\xi^\beta}.
\end{split}
\end{equation}
A straightforward use of \eqref{m-fac3} gives both
$$\xi m(\xi^{-1})\leq \xi_0 m(\xi_0^{-1}) \ \ {\rm{and}} \ \ m(\xi_0^{-1})\leq \left(\frac{\delta}{\xi_0}\right)^\alpha m(\delta^{-1}),$$
and hence, one finds
\begin{equation}
  \xi m(\xi^{-1}) \partial_\xi \omega(\xi,\xi_0)\leq \sigma\kappa \frac{\xi_0^\sigma m(\xi_0^{-1})}{\delta^\sigma m(\delta^{-1})}
  \leq \sigma\kappa \frac{\xi_0^{\sigma-\alpha}}{\delta^{\sigma-\alpha}}  \leq \sigma\kappa,
\end{equation}
which also leads to
\begin{equation}
  -C_2\xi m(\xi^{-1})D(\xi,t) \partial_\xi \omega(\xi,\xi_0) \leq -C_2\sigma\kappa D(\xi,t).
\end{equation}
By integrating by parts, and using \eqref{eq:dom}-\eqref{eq:fact}, \eqref{m-fac3} along with formula \eqref{moc3}, we see that,
\begin{align*}
  & \int_\xi^\infty \frac{\omega(\eta,\xi_0) m(\eta^{-1})}{\eta^{1+\beta}}\dd\eta
  \leq  \ \frac{\omega(\xi,\xi_0) m(\xi^{-1})}{\beta\xi^\beta} + \int_\xi^\infty \frac{\partial_\eta\omega(\eta,\xi_0) m(\eta^{-1})}{\beta\eta^\beta} \dd\eta \\
  \leq & \ \frac{\omega(\xi,\xi_0) m(\xi^{-1})}{\beta\xi^\beta}
  + \frac{\sigma \kappa \xi_0^{\sigma-1}}{\beta \delta^\sigma m(\delta^{-1})} \int_\xi^{\xi_0} \frac{ m(\eta^{-1})}{\eta^\beta} \dd\eta
  + \frac{\sigma \kappa }{\beta \delta^\sigma m(\delta^{-1})} \int_{\xi_0}^\delta  \frac{ m(\eta^{-1})}{\eta^{1+\beta-\sigma}} \dd\eta
  + \frac{\gamma}{\beta}\int_\delta^\infty \frac{1}{\eta^{1+\beta}} \dd\eta \\
  \leq & \ \frac{\omega(\xi,\xi_0) m(\xi^{-1})}{\beta\xi^\beta}
  + \frac{\sigma \kappa \xi_0^{\sigma-1}}{\beta \delta^{\sigma-\alpha}} \int_\xi^{\xi_0} \frac{1}{\eta^{\alpha+\beta}} \dd \eta
  + \frac{\sigma \kappa }{\beta \delta^{\sigma-\alpha} } \int_{\xi_0}^\delta  \frac{1}{\eta^{1+\beta+\alpha-\sigma}} \dd\eta
  + \frac{\gamma}{\beta^2} \delta^{-\beta}  \\
  \leq & \ \frac{\omega(\xi,\xi_0) m(\xi^{-1})}{\beta \xi^\beta} +
   \begin{cases}
    \frac{\sigma\kappa}{\beta(1-\alpha-\beta)\xi_0^\beta}+\frac{\sigma\kappa}{\beta(\alpha+\beta-\sigma)\xi_0^\beta}+\frac{\gamma}{\beta^2\delta^\beta},
    &\quad \textrm{if   }\alpha+\beta<1, \\
     \frac{\sigma\kappa  }{\beta \xi_0^\beta} \log\frac{\xi_0}{\xi} + \frac{\sigma\kappa}{\beta(1-\sigma)\xi_0^\beta}+\frac{\gamma}{\beta^2\delta^\beta},
    &\quad \textrm{if   }\alpha+\beta = 1, \\
    \frac{\sigma\kappa}{\beta(\alpha+\beta-1)\xi^\beta} + \frac{\sigma\kappa}{\beta(\alpha+\beta-\sigma)\xi_0^\beta}+\frac{\gamma}{\beta^2\delta^\beta},
    &\quad \textrm{if   }\alpha+\beta>1.
  \end{cases}
\end{align*}
By using \eqref{Omega3} together with the fact that $\xi_0 \partial_\xi \omega(\xi,\xi_0) \leq \frac{\sigma }{1-\sigma} \omega(\xi,\xi_0)$, the previous inequality
allows to conclude that, by choosing $\kappa$ and $\gamma$ such that  $\kappa\leq\frac{1}{2C_2\sigma}$ and $\gamma\leq\sigma\kappa$, one has the following inequality
\begin{equation}\label{Omeg-est3}
\begin{split}
  \Omega(\xi,t)\partial_\xi\omega(\xi,\xi_0) \leq & -C_2\sigma\kappa D(\xi,t) + \frac{2C_2 \sigma \kappa}{\beta} \frac{\omega(\xi,\xi_0)}{\xi^\beta} + \\
  & +
  \begin{cases}
    \frac{C_2\sigma\kappa}{\beta\xi^\beta}\left(\frac{1}{1-\alpha-\beta} + \frac{1}{\alpha+\beta-\sigma}+\frac{1}{\beta}\right)\xi_0 \partial_\xi \omega(\xi,\xi_0), &\quad \textrm{if   }\alpha+\beta<1,\\
    \frac{C_2\sigma\kappa}{\beta\xi^\beta}\left(C_0 + \frac{1}{\alpha+\beta-\sigma}+\frac{1}{\beta}\right)\xi_0 \partial_\xi \omega(\xi,\xi_0), &\quad \textrm{if   }\alpha+\beta=1,\\
    \frac{C_2\sigma\kappa}{\beta\xi^\beta}\left(\frac{1}{\alpha+\beta-1}+ \frac{1}{\alpha+\beta-\sigma}+\frac{1}{\beta}\right)\xi_0 \partial_\xi \omega(\xi,\xi_0), &\quad \textrm{if   }\alpha+\beta>1,\\
  \end{cases}
  \\
  \leq &   -\frac{1}{2}D(\xi,t) + \frac{3C_2\sigma\kappa K_{\alpha,\beta,\sigma}}{\beta(1-\sigma)}\frac{\omega(\xi,\xi_0)}{\xi^\beta},
\end{split}
\end{equation}
where we have used $\frac{\xi}{\xi_0}\log\frac{\xi_0}{\xi}\leq C_0$ and have set
\begin{equation}
\begin{split}
  K_{\alpha,\beta,\sigma} \equiv
\begin{cases}
  \frac{1}{1-\alpha-\beta} + \frac{2}{\alpha+\beta-\sigma}, &\quad\textrm{if   }\alpha+\beta<1,\\
  C_0 + \frac{2}{1-\sigma}, &\quad\textrm{if   }\alpha+\beta=1,\\
  \frac{3}{\alpha+\beta-1}, &\quad\textrm{if   }\alpha+\beta>1.\\
\end{cases}
\end{split}
\end{equation}

\noindent For the contribution from the diffusion term, by following the same idea  as the proof of \eqref{D2xit00} and using \eqref{eq:fact2}, we get
\begin{equation}\label{D2xit3}
\begin{split}
  D(\xi,t)  \leq - 2 C_1\omega(0+,\xi_0)\int_{\frac{\xi}{2}}^\infty \frac{1}{\eta^{1+\beta}}\dd\eta
  \leq\frac{-2C_1 (1-\sigma)\omega(\xi,\xi_0)}{\beta \xi^\beta}.
\end{split}
\end{equation}

Then, thanks to \eqref{omegxi0-3}, \eqref{Omeg-est3}, \eqref{D2xit3} along with \eqref{Targ3}, one finds that
\begin{equation}
  \textrm{L.H.S. of \eqref{Targ3}}\,\leq \frac{\omega(\xi,\xi_0)}{\beta \xi^\beta} \left(\beta\sigma \rho + \frac{3C_2\sigma\kappa K_{\alpha,\beta,\sigma}}{1-\sigma} - C_1(1-\sigma)\right)<0,
\end{equation}
where the last inequality is ensured by setting
\begin{equation}\label{rkg-cd6}
\begin{split}
  \rho < \frac{C_1 (1-\sigma)}{2\beta\sigma},\quad \kappa <\min\left\{ \frac{C_1(1-\sigma)^2}{6C_2\sigma K_{\alpha,\beta,\sigma}}, \frac{1}{2C_2\sigma}\right\}, \quad\gamma \leq \sigma \kappa.
\end{split}
\end{equation}

\textbf{Case 5: $0<\xi_0\leq \delta$, $\xi >\xi_0$.} \\

In this case, we may follow what we did in Case 1 and Case 2 in the proof of Lemma \ref{lem:RCpres}, and we omit the details.
However, it is worth saying that the conditions on $\kappa$, $\gamma$, in that case, are given by \eqref{kg-cd1}. \\

Finally, we have proved that \eqref{Targ3} is verified for all $\xi>0$ and $t >0$ for the moduli of continuity $\omega(\xi,\xi_0)$ defined by \eqref{moc2}-\eqref{moc3} and $\xi_0=\xi_0(t)$ given by \eqref{xi0}. Recall that $\rho,\kappa,\gamma$ are constants satisfying all the conditions coming from each cases, namely \eqref{kg-cd1}, \eqref{rkg-cd1}, \eqref{rkg-cd2}, \eqref{rkg-cd3}, \eqref{rkg-cd4}, \eqref{rkg-cd5}, \eqref{rkg-cd6}. This ends the proof of Lemma \ref{lem-mocev}.

\end{proof}

\begin{remark}
By suppressing the dependence on $C_0$, $b_0,b_2$ (which are fixed constants), and recalling that $C_2=\frac{C_0}{\beta}$ (the constant appearing in Lemma \ref{lem-mocdrf}),
we find that the conditions on $\rho, \kappa$ and $\gamma$ can be written as
\begin{equation}\label{rkg-cdsum}
\begin{split}
\begin{cases}
  \rho \leq\frac{1}{C}\min\set{\frac{C_1(1-\alpha)}{\beta},\frac{C_1(1-\sigma)}{\beta\sigma}}, \\
  \kappa\leq \frac{C_1\beta(1-\sigma)^2}{C}\min\set{1-\alpha-\beta,\alpha+\beta-\sigma},   & \quad \textrm{if   }\alpha+\beta<1, \\
  \kappa\leq \frac{C_1\beta(1-\sigma)}{C}\min\set{\frac{1}{C_\alpha'},(1-\sigma)^2},  & \quad \textrm{if   }\alpha+\beta=1, \\
  \kappa\leq \frac{C_1}{C}\beta(1-\sigma)^2(\alpha+\beta-1),  & \quad \textrm{if   }\alpha+\beta>1, \\
  \gamma \leq \frac{1}{C} \min \set{\sigma \kappa,(1-C_\alpha)^\alpha \kappa,(1-\sigma)(1-\alpha)\kappa,C_1(1-C_\alpha)\beta^2},
\end{cases}
\end{split}
\end{equation}
where $C>0$ is fixed constant that depends neither on $\alpha,\sigma,\beta$, nor $C_1$ (the constant appearing in Lemma \ref{lem-mocdiss}) and the other constants are defined by $C_\alpha\equiv\frac{2^\alpha -1}{\alpha}\geq \inf_{\alpha\in (0,1)}\frac{2^\alpha-1}{\alpha}=c_0>0$, and $C_\alpha'\equiv\displaystyle\sup_{x\in [1,+\infty)}\frac{1}{x^{1-\alpha}}\log x<\infty$.
\end{remark}

\subsection{Proof of (\ref{T*}): remark on the time of the eventual regularity}\label{subsec:T*}

The aim of this subsection is to show that for any fixed initial data $\theta_0\in L^2$, $t'>0$ and $\beta\in(0,1)$, the time of eventual regularity $t_1= A_0^\beta/(\beta\rho)$ tends to 0 as $\alpha\rightarrow 0$.

Without loss of generality we may assume $A_0<1$. Recalling that \eqref{eq:Adel-cd} is verified if $$\frac{(1-\alpha)\gamma}{\alpha\, m(1)} (A_0^\alpha -\delta^\alpha) \geq \frac{2 C_\beta}{t'^{1/\beta}} \|\theta_0\|_{L^2},$$ therefore we may choose
\begin{equation}\label{A0del1}
  A_0= \left(\frac{4 C_\beta m(1)}{(1-\alpha)\gamma t'^{1/\beta}}\alpha \|\theta_0\|_{L^2}\right)^{\frac{1}{\alpha}},\quad
  \delta= \left( \frac{ C_\beta m(1)}{(1-\alpha)\gamma t'^{1/\beta}}\alpha \|\theta_0\|_{L^2}\right)^{\frac{1}{\alpha}}.
\end{equation}
Since at the end we are going to consider the limit as $\alpha\rightarrow 0$, we may assume that
$$\alpha<\min\Big\{\frac{1-\beta}{2},\frac{\beta}{2},\frac{1}{4}\Big\}\quad \textrm{ for all   $\beta\in(0,1)$,}$$
and since $\sigma\in(\alpha, \min\{\alpha+\beta,1\})$, we may set $\sigma=\min\{\frac{1}{3},\frac{2\beta}{3}\}$, thus since $C_\alpha=\frac{2^\alpha -1}{\alpha}\approx (\ln 2)^2$ and $ C_\alpha'=\sup_{x\in [1,+\infty)}\frac{1}{x^{1-\alpha}}\log x\approx 1$ for $\alpha$ close to 0, we see that \eqref{rkg-cdsum} reduces to
\begin{equation}\label{rkg-cond}
  0<\rho \leq \frac{C_1}{C\beta}, \quad
  0<\kappa<\frac{C_1}{C}\min\set{\beta(1-\beta),\beta^2},\quad
   0<\gamma< \frac{C_1}{C}\min \set{\beta(1-\beta), \beta^2},
\end{equation}
with $C>0$ and $C_1>0$ some constant depending only on $\beta$.
Hence, by choosing  $\rho,\kappa,\gamma$ satisfying the above conditions, and via \eqref{A0del1}, we obtain that
for all $\beta\in(0,1)$ for all fixed time $t'>0$ and data $\theta_0\in L^2(\R^2)$,
\begin{equation}\label{eq:t1conv}
  t_1 = \frac{1}{\beta\rho} \left(\frac{4 C_\beta m(1)}{(1-\alpha)\gamma t'^{1/\beta}}\alpha \|\theta_0\|_{L^2}\right)^{\frac{\beta}{\alpha}}
  =C(\beta)\big(C(\beta,t') \alpha \|\theta_0\|_{L^2}\big)^{\frac{\beta}{\alpha}}\rightarrow 0,\quad \textrm{as   }\alpha\rightarrow 0.
\end{equation}
This ends the proof of the remark (\ref{T*}) on the time of eventual regularity.\\
\qed
\vspace{-0,6cm}
\subsection{Global regularity of weak solution for $(gSQG)_\beta$ equation \eqref{eq:gSQG} at the logarithmically supercritical case}\label{subsec:grLog}

The aim of this subsection is to prove the first point of Theorem \ref{thm:RegLog}, that is the global regularity of the weak solution for a logarithmically slightly supercritical case.
Unlike the second point of Theorem \ref{thm:RegLog}, we now assume that the multiplier $m$ verifies (A1)-(A4). \\

We use the same method as the proof of the eventual regularity as above. However, in this case, the moduli of continuity $\omega(\xi,\xi_0)$ are constructed from $\omega(\xi)$ defined by \eqref{moc1.2} instead of \eqref{moc1}.
Hence, the moduli of continuity $\omega(\xi,\xi_0)$ used here are slightly different from those given by \eqref{moc2}-\eqref{moc3}. It is defined as follows,
\begin{equation}\label{moc4}
\omega(\xi,\xi_0)=
\begin{cases}
  \textrm{$\omega(\xi,\xi_0)$  given by \eqref{moc2}},\quad & \textrm{if   }\; 0<\xi\leq \frac{1}{2b_1},\xi_0> \delta, \\
  \textrm{$\omega(\xi,\xi_0)$  given by \eqref{moc3}},\quad & \textrm{if   }\; 0<\xi\leq \frac{1}{2b_1},\xi_0\leq \delta, \\
  \omega(\frac{1}{2b_1},\xi_0),\quad & \textrm{if   }\; \xi\geq \frac{1}{2b_1},
\end{cases}
\end{equation}
where $\xi_0=\xi_0(t)$ is defined by \eqref{xi0} which has an explicit expression namely $\xi_0(t)= (A_0^\beta - \rho \beta t)^{\frac{1}{\beta}}$.
Clearly, $\omega(\xi,\xi_0)$ given by \eqref{moc4} are moduli of continuity satisfying the condition (3) of Proposition \ref{prop-GC}.

Without loss of generality, we may assume that $\xi_0(0)=A_0$ is small enough, more precisely we shall assume that $A_0\leq \min\{\frac{1}{2b_1},\frac{1}{b_3}\}$ where $b_1\geq 1$ is  the constant appearing in (A3) and $b_3\geq 1$ the constant verifying \eqref{mcd1},  that is
$1/b_3 \leq m(r) \leq b_3 (\log r)^\mu$ for all $r \geq b_3$.
Then, using \eqref{mcd1} we have
\begin{align}\label{omeg0A}
  \omega(0+,A_0) & =(1-\sigma)\kappa \frac{1}{m(\delta^{-1})} + \gamma \int_\delta^{A_0}\frac{1}{\eta m(\eta^{-1})} \dd \eta
    -\gamma \frac{1}{A_0 m(A_0^{-1})}(A_0-\delta) \nonumber\\
  & > \frac{\gamma}{b_3} \int_\delta^{A_0} \frac{1}{\eta(\log \eta^{-1} )^\mu}\dd \eta - \frac{\gamma}{m(1)} \nonumber \\
  & \geq \frac{\gamma}{b_3} \int_{\frac{1}{A_0}}^{\frac{1}{\delta}} \frac{1}{\eta(\log \eta )^\mu}\dd \eta - \frac{\gamma}{m(1)} \\
  & \geq \nonumber
  \begin{cases}
    \frac{\gamma}{b_3(1-\mu)}\left( \left( \log\frac{1}{\delta}\right)^{1-\mu}- \left( \log \frac{1}{A_0}\right)^{1-\mu} \right)- \frac{\gamma}{m(1)} , & \quad \textrm{if   }\mu\in[0,1), \\
    \frac{\gamma}{b_3} \left( \log\log \frac{1}{\delta} -\log\log \frac{1}{A_0}\right) - \frac{\gamma}{m(1)}, & \quad \textrm{if   }\mu=1.
  \end{cases}
\end{align}
Then, by using  \eqref{LinfEst}, we know that, for all $t_*>0$,
\begin{equation}\label{theLinf}
  \left\|\theta^\epsilon \right\|_{L^\infty([\frac{t_*}{2},\infty)\times\R^2 )} \leq C_\beta \Big(\frac{ t_*}{2}\Big)^{-1/\beta} \|\theta_0\|_{L^2(\R^2)},
\end{equation}
thus in order for the inequality $\omega(0+,A_0)> 2\|\theta^\epsilon(\frac{t_*}{2})\|_{L^\infty}$ to be verified, we can let
\begin{equation}\label{eq:A0est}
  \omega(0+,A_0)> 2C_\beta \Big(\frac{ t_*}{2}\Big)^{-1/\beta} \|\theta_0\|_{L^2(\R^2)}.
\end{equation}
Hence, if $\mu\in [0,1)$, we need
\begin{equation*}
  \log \frac{1}{\delta} \geq  \left[\left( \log \frac{1}{A_0}\right)^{1-\mu} + \frac{b_3(1-\mu)}{\gamma}\left(2C_\beta  \Big(\frac{t_*}{2}\Big)^{-1/\beta} \|\theta_0\|_{L^2} +\frac{\gamma}{m(1)} \right)\right]^{\frac{1}{1-\mu}},
\end{equation*}
and using the inequality $(a+b)^{\frac{1}{1-\mu}}\leq C_\mu (a^{\frac{1}{1-\mu}} + b^{\frac{1}{1-\mu}})$ for $a,b>0$, one observes that it suffices to choose $\delta$ so that
\begin{equation}\label{del-cd1}
  \delta = A_0^{C_\mu} \exp\bigg(- C_\mu \bigg(\frac{3C_\beta b_3(1-\mu)}{\gamma} \left(\frac{t_*}{2}\right)^{-1/\beta} \|\theta_0\|_{L^2} + \frac{b_3(1-\mu)}{m(1)}\bigg)^{1/(1-\mu)}\bigg),
\end{equation}
whereas if $\mu=1$, it suffices to set $\delta$ as
\begin{equation*}
  \log\log\frac{1}{\delta} =\log \log \frac{1}{A_0} + \frac{b_3}{\gamma}\left(3C_\beta\Big(\frac{t_*}{2}\Big)^{-1/\beta} \|\theta_0\|_{L^2}+\frac{\gamma}{m(1)} \right),
\end{equation*}
that is,
\begin{equation}\label{del-cd2}
  \delta =A_0^{\exp \left(\frac{3C_\beta b_3}{\gamma}(\frac{t_*}{2})^{-1/\beta} \|\theta_0\|_{L^2} + \frac{b_3}{m(1)} \right)}.
\end{equation}

Next, we will prove that such moduli of continuity $\omega(\xi,\xi_0(t))$ are preserved by the regular solution $\theta^\epsilon(x,t+\frac{t_*}{2})$.
By choosing $\delta$ as \eqref{del-cd1}-\eqref{del-cd2} and using what we did before, we get that  $\theta^\epsilon(x,\frac{t_*}{2})$ obeys the modulus of continuity $\omega(\xi,A_0)$.
Then according to Proposition \ref{prop-GC}, and by using \eqref{theLinf}-\eqref{eq:A0est} and the fact $\omega(A_0,\xi_0)\geq \omega(0+,A_0)$ for all $\xi_0\geq 0$, we see that it suffices to verify that
\begin{equation}\label{Targ4}
  - \partial_{\xi_0}\omega(\xi,\xi_0) \xi'_0(t)+ \Omega(\xi,t)\partial_\xi\omega(\xi,\xi_0) + D(\xi,t) + \epsilon \partial_{\xi\xi}\omega(\xi,\xi_0)<0,
\end{equation}
for all $t$ so that $\xi_0(t)> 0$ and all $0<\xi\leq A_0$.
Since we also have \eqref{Ome-es1-2} for the estimation of $\Omega(\xi,t)$ and since $\partial_\eta \omega(\eta,\xi_0)\equiv 0$ for all $\eta>\frac{1}{2b_1}$, we see that
the proof of \eqref{Targ4} is the same as that of \eqref{Targ3} in the subsection \ref{subsec:EvReg}.
The conditions on $\rho,\kappa,\gamma$ can also be chosen as \eqref{rkg-cdsum}, and by setting $\sigma=\min\{\alpha+\frac{\beta}{2}, \frac{\alpha+1}{2}\}$, we may choose the coefficients $\rho,\kappa,\gamma$ as
\begin{equation}\label{rkg-cd*}
\begin{split}
\begin{cases}
  \rho =\frac{C_1}{2C}\min\set{\frac{1-\alpha}{\beta},\frac{1-\sigma}{\beta\sigma}}, \\
  \kappa= \frac{C_1\beta(1-\sigma)^2}{2C}\min\set{1-\alpha-\beta,\alpha+\beta-\sigma},   & \quad \textrm{if   }\alpha+\beta<1, \\
  \kappa= \frac{C_1\beta(1-\sigma)}{2C}\min\set{\frac{1}{C_\alpha'},(1-\sigma)^2},  & \quad \textrm{if   }\alpha+\beta=1, \\
  \kappa= \frac{C_1}{2C}\beta(1-\sigma)^2(\alpha+\beta-1),  & \quad \textrm{if   }\alpha+\beta>1, \\
  \gamma = \frac{1}{2C} \min \set{\sigma \kappa,(1-C_\alpha)^\alpha \kappa,(1-\sigma)(1-\alpha)\kappa,C_1(1-C_\alpha)\beta},
\end{cases}
\end{split}
\end{equation}
where $C_\alpha=\frac{2^\alpha -1}{\alpha}$, $C_\alpha'=\sup_{x\in [1,+\infty)}(\frac{1}{x^{1-\alpha}}\log x)$, $C>0$ is some absolute constant and $C_1=C_1(\beta)$ is the constant appearing in Lemma \ref{lem-mocdiss}.

Then after some finite time $t_1= \frac{A_0^\beta}{\beta \rho}$ so that $\xi_0(t_1)=0$, from the continuous property of $\theta^\epsilon(x,t)$ and $\omega(\xi,\xi_0(t))$,
we find that $\theta^\epsilon(x,\frac{t_*}{2}+t_1)$ obeys the modulus of continuity $\omega(\xi,0+)=\omega(\xi)$ given by \eqref{moc1.2}.
Thanks to \eqref{theLinf} and \eqref{eq:A0est}, the condition \eqref{eq:RCcd3} is immediately satisfied, then Lemma \ref{lem:RCpres} and Remark \ref{rmk:RCgene}
imply that the approximate solution $\theta^\epsilon$ (uniformly in $\epsilon$) obeys the modulus of continuity $\omega(\xi)$ given by \eqref{moc1.2} on the time interval $[\frac{t_*}{2}+t_1,\infty)$,
and thanks to \eqref{CsgmEs} we get
\begin{equation}\label{the-est1}
\begin{split}
    \sup_{t\in [t_1+t_*/2,\infty)} &  \|\theta^\epsilon(t)\|_{\dot C^\sigma(\R^2)}\leq \kappa \frac{1}{m(\delta^{-1})} \delta^{-\sigma} \leq \frac{\kappa}{m(1)}  \delta^{-\sigma}
\end{split}
\end{equation}
where $\delta$ is given by \eqref{del-cd1}-\eqref{del-cd2}, $\rho,\kappa,\gamma$ are fixed constants that appear in \eqref{rkg-cd*} and $0<A_0\leq \min\{\frac{1}{2b_1},\frac{1}{b_3}\}$.

Then, in order to let $t_1\leq \frac{t_*}{4}$, we also need that $A_0$ satisfies $ \frac{A_0^\beta}{\beta\rho }\leq \frac{t_*}{4}$, {\emph{i.e.}}
$A_0\leq  \left(\beta \rho t_*/4\right)^{1/\beta}$,
thus for each $\alpha\in (0,1)$ and $\beta\in (0,1]$, we can set $A_0$ to be
\begin{equation}
  A_0 = \min \set{\Big(\frac{\beta \rho t_*}{4}\Big)^{1/\beta},\frac{1}{2b_1},\frac{1}{b_3}},
\end{equation}
so that the uniform-in-$\epsilon$ H\"older estimate \eqref{the-est1} holds true with such an $A_0$ in the formula of $\delta$ defined by \eqref{del-cd1}-\eqref{del-cd2}.
By using Lemma \ref{lem:RC}, we can further show that $\theta^\epsilon \in C^\infty([t_*,\infty)\times\R^2)$ uniformly in $\epsilon$, which by passing $\epsilon\rightarrow 0$
implies that the global weak solution $\theta$ of the $(gSQG)_\beta$ equation \eqref{eq:gSQG} satisfies $\theta\in C^\infty([t_*,\infty)\times\R^2)$. This ends the proof of Theorem \ref{thm:RegLog}.

\qed

\section{Appendix}

\begin{proof}[\bf{Proof of the global part of Proposition \ref{prop:glob}}]
Assume that $T^*$ is the maximal time of existence for the solution $\theta^\epsilon$ to the equation \eqref{ap-gSQG2} in
$C([0,T^*),H^s(\R^2))\cap C^\infty([0,T^*)\times\R^2)$ with $s>2$. Then, according to the classical regularity criteria (see \cite{XZ}), it suffices to prove that the norm
$\|\nabla\theta^\epsilon\|_{L^\infty([0,T^*), L^\infty(\R^2))}$ is bounded.

In the sequel, we shall find some stationary modulus of continuity
\begin{equation}\label{eq:omelmd}
  \omega_\lambda(\xi)\equiv \lambda^{2-\alpha-\beta} \omega(\lambda \xi),\quad \lambda\in (0,\infty)
\end{equation}
where
\begin{equation}\label{eq:moc5}
\omega(\xi)=
\begin{cases}
\xi-\xi^{\frac{3}{2}}, \quad & \text{if} \quad 0< \xi\leq \delta, \\
 \delta-\delta^{\frac{3}{2}}, \quad & \text{if} \quad \xi>\delta,
\end{cases}
\end{equation}
with some $0<\delta<1$ chosen later, so that it is preserved by the evolution of equation \eqref{ap-gSQG2}. Clearly, $\omega_\lambda$ is a modulus of continuity, moreover, it satisfies
$\omega_\lambda(0+)=0$, $\omega_\lambda'(0+)=\lambda$ and $\omega_\lambda''(0+)= -\infty$.

We first notice that by choosing $\lambda$ as
\begin{equation}\label{eq:lambda}
 \lambda \equiv \max\left\{\Big(\frac{4\|\theta^\epsilon_0\|_{L^\infty}}{\delta/2-(\delta/2)^{3/2}}\Big)^{\frac{1}{2-\alpha-\beta}},
 \frac{\delta  \|\nabla\theta^\epsilon_0\|_{L^\infty}}{2\|\theta^\epsilon_0\|_{L^\infty}},1 \right\},
\end{equation}
we have that $\theta^\epsilon_0$ obeys this $\omega_\lambda$ for $\lambda$ sufficiently large. Indeed, to prove this claim, it suffices to observe that
$\min\{2\|\theta^\epsilon_0\|_{L^\infty}, \|\nabla\theta^\epsilon_0\|_{L^\infty}\xi \} < \omega_\lambda(\xi)$. Therefore, by setting $a_1=\frac{2\|\theta^\epsilon_0\|_{L^\infty}}{\|\nabla\theta^\epsilon_0\|_{L^\infty}}$,
and by using the concavity of $\omega_\lambda(\xi)$, we see that it suffices to show that
\begin{equation}\label{econd2}
  \omega_\lambda(a_1)=\lambda^{2-\alpha-\beta} \omega(\lambda a_1) > 2\|\theta^\epsilon_0\|_{L^\infty}.
\end{equation}
Therefore, in order to prove \eqref{econd2}, we only need to let $\lambda$ large enough so that
\begin{equation}\label{eq:fact5}
  \omega_\lambda(a_1)>\omega_\lambda\Big(\frac{\delta}{2\lambda}\Big) = \lambda^{2-\alpha-\beta} \omega\Big(\frac{\delta}{2}\Big) > 2\|\theta^\epsilon_0\|_{L^\infty},
\end{equation}
that is, $\lambda a_1 > \frac{\delta}{2}$ and $\omega\big(\frac{\delta}{2}\big)> \frac{2\|\theta^\epsilon_0\|_{L^\infty}}{\lambda^{2-\alpha-\beta}}$,
hence, we can choose $\lambda$ as \eqref{eq:lambda} and this proves the claim. \\

Then according to Proposition \ref{prop-GC}, it remains to check that for all $t\in (0,T^*)$ and
$0<\xi\in \{\omega_\lambda(\xi)\leq 2\|\theta^\epsilon(t)\|_{L^\infty}\}\subset \{\omega_\lambda(\xi)\leq 2\|\theta^\epsilon_0\|_{L^\infty}\}$, we have
\begin{equation}\label{TargApd}
  \Omega_\lambda(\xi) \omega_\lambda'(\xi) + \epsilon\, \omega_\lambda''(\xi)<0,
\end{equation}
where $\Omega_\lambda(\xi)$ is given by \eqref{Ome-es2}, namely
$$
  \Omega_\lambda(\xi)  \leq  C \int_0^\xi \frac{\omega_\lambda(\eta) m(\eta^{-1})}{\eta^\beta} \dd \eta + C \xi\int_\xi^\infty \frac{\omega_\lambda(\eta)m(\eta^{-1})}{\eta^{1+\beta}}\mathrm{d}\eta.
$$
By making a change of variables and using the fact $m(\eta^{-1})\leq \lambda^\alpha m((\lambda \eta)^{-1})$, for all $\eta>0$, one finds that
\begin{equation*}
\begin{split}
  \Omega_\lambda(\xi)  
  & \leq C \lambda \int_0^{\lambda\xi} \frac{\omega(\eta) m(\eta^{-1})}{\eta^\beta} \dd \eta
  + C \lambda^2 \xi\int_{\lambda\xi}^\infty \frac{\omega(\eta) m(\eta^{-1})}{\eta^{1+\beta}}\mathrm{d}\eta \equiv \lambda \Omega(\lambda\xi) .
\end{split}
\end{equation*}
Note that, using \eqref{eq:lambda} and \eqref{eq:fact5}, we have $\omega_\lambda(\frac{\delta}{2\lambda})>2\|\theta^\epsilon_0\|_{L^\infty}$, thus it suffices to prove that
\begin{equation*}
  \lambda^{4-\alpha-\beta}\big(\Omega\,\omega'+ \epsilon\omega''\big)(\lambda \xi)<0, \quad \textrm{for all}\;\; \xi\in \Big(0,\frac{\delta}{2\lambda}\Big).
\end{equation*}
Hence, our aim is to show that, the modulus of continuity $\omega$ defined by \eqref{eq:moc5} verifies
\begin{equation}\label{TargApd2}
  \Omega(\xi)\omega'(\xi)+ \epsilon\, \omega''(\xi)<0,  \quad \textrm{for all}\quad\xi\in (0,\delta/2),
\end{equation}
with $\Omega(\xi)\equiv C \big(\int^\xi_0 \frac{\omega(\eta) m(\eta^{-1})}{\eta^\beta}\dd \eta +  \xi\int_\xi^\infty \frac{\omega(\eta) m(\eta^{-1})}{\eta^{1+\beta}}\dd \eta \big)$ and $C>0$ a constant depending only on $\beta$.

Since $\eta^\alpha m(\eta^{-1})\leq \delta^\alpha m(\delta^{-1}) \leq m(1)$, $\forall\eta \in(0,\delta)$, we have
$\int_0^\xi \frac{\omega(\eta)m(\eta^{-1})}{\eta^\beta} \textrm{d} \eta \leq \frac{m(1)}{2-\alpha-\beta} \xi^{2-\alpha-\beta},$
 and
\begin{equation*}
  \int_\xi^\delta \frac{\omega(\eta)m(\eta^{-1})}{\eta^{1+\beta}} \textrm{d}
  \eta \leq m(1) \int_\xi^\delta \frac{1}{\eta^{\alpha+\beta}}\textrm{d} \eta \leq
  \begin{cases}
    \frac{m(1)}{1-\alpha-\beta} \delta^{1-\alpha-\beta}, &\quad \textrm{if   }\alpha+\beta<1, \\
    m(1) \log \frac{\delta}{\xi}, & \quad \textrm{if   } \alpha+\beta=1, \\
    \frac{m(1)}{\alpha+\beta-1} \xi^{1-\alpha-\beta}, & \quad \textrm{if  }\alpha+\beta>1.
  \end{cases}
\end{equation*}
Moreover,
$$\int_\delta^\infty \frac{\omega(\eta)m(\eta^{-1})}{\eta^{1+\beta}} \textrm{d} \eta  \leq \omega(\delta) m(\delta^{-1}) \int_\delta^\infty \frac{1}{\eta^{1+\beta}}\dd \eta
  \leq \frac{m(1)}{\beta} \delta^{1-\alpha-\beta}.$$
Obviously, $\omega'(\xi)\leq\omega'(0)=1$, so we get
$$\Omega(\xi) \omega'(\xi) \leq C_{\alpha,\beta} \delta^{2-\alpha-\beta},$$
where $C_{\alpha,\beta}=C \frac{m(1)}{1-\alpha-\beta}$ if $\alpha+\beta<1,$ $C_{\alpha,\beta}=Cm(1)$ if $\alpha+\beta=1$, $C_{\alpha,\beta}=C \frac{m(1)}{\alpha+\beta-1}$ if $\alpha+\beta>1.$

Since $\omega''(\xi)=-\frac{3}{4}\xi^{-\frac{1}{2}}<0$, then by choosing $\delta>0$ small enough, we find
\begin{align*}
  \Omega(\xi)\omega'(\xi)+ \epsilon\, \omega''(\xi)\leq  C_{\alpha,\beta} \delta^{2-\alpha-\beta} - \epsilon \frac{3}{4} \xi^{-\frac{1}{2}} \leq
  \delta^{-\frac{1}{2}} \Big(C_{\alpha,\beta}\delta^{5/2-\alpha-\beta}-\frac{3}{4}\epsilon\Big)<0,\quad \rm{for \ all} \ \,\xi\in (0,\delta/2).
\end{align*}

Hence, the solution $\theta^\epsilon$ to equation \eqref{ap-gSQG2} obeys such a modulus of continuity $\omega_\lambda$ with $\lambda$ given by \eqref{eq:lambda} for all $t\in [0,T^*)$,
which implies that
$\displaystyle\sup_{t\in[0,T^*)}\|\nabla\theta^\epsilon(t,\cdot)\|_{L^\infty}\leq \lambda$, hence,  Prop \ref{prop:glob} is proved.
\end{proof}

\section*{Acknowledgments}
\noindent O.L. was supported by the Marie-Curie Grant, acronym: TRANSIC, from the FP7-IEF program, and the ERC through the Starting Grant project H2020-EU.1.1.-63922. He was also supported by the National Grant MTM2014-59488-P from the Spanish government.
L.X. was partially supported by National Natural Science Foundation of China (Grants Nos. 11401027, 11671039 and 11771043).


\end{document}